\title{An Ensemble Kalman Filter Implementation Based on Modified Cholesky Decomposition for Inverse Covariance Matrix Estimation}
\author{
        Elias D. Nino \\
        Computational Science Laboratory\\
        Department of Computer Science\\
        Virginia Tech \\
        Blacksburg, VA 24060, \underline{USA}, E-mail: {enino@vt.edu}
            \and
        Adrian Sandu \\
        Computational Science Laboratory\\
        Department of Computer Science\\
        Virginia Tech \\
        Blacksburg, VA 24060, \underline{USA}, E-mail: {asandu7@vt.edu}
         \and
        Xinwei Deng \\
        Department of Statistics \\
        Virginia Tech \\
        Blacksburg, VA 24060, \underline{USA}, E-mail: {xdeng@vt.edu}
}
\date{\today}

\documentclass[12pt]{article}

\usepackage{caption}
\usepackage{amsmath}
\usepackage{amssymb}
\usepackage[numbers]{natbib}
\usepackage{bbm, dsfont}
\usepackage{algorithmic}
\usepackage{algorithm}
\usepackage{graphicx}
\usepackage{epstopdf}
\usepackage{tabularx}
\usepackage{rotating}
\usepackage{multirow}
\usepackage{color, colortbl}
\usepackage{subcaption}
\usepackage{amsthm}
\usepackage{url,color}

\newtheorem{theorem}{Theorem}
\newtheorem{lemma}[theorem]{Lemma}
\newtheorem{definition}[theorem]{Definition}

\newcommand{\classU}{\mathcal{U}}

\newcommand{\CT}{{\bf C}}
\newcommand{\CE}{\widehat{\bf C}}

\newcommand{\MV}{{\bf M}}

\renewcommand{\AE}{\widehat{\bf A}}

\newcommand{\err}{\boldsymbol{\varepsilon}}

\newcommand{\WV}{{\bf W}}
\newcommand{\h}{{\bf h}}

\newcommand{\BE}{\widehat{\bf B}}
\newcommand{\DE}{\widehat{\bf D}}
\newcommand{\TE}{\widehat{\bf T}}

\newcommand{\be}{{\boldsymbol \beta}}

\newcommand{\Eig}{\boldsymbol{\Sigma}}

\newcommand{\de}{\widehat{d}}

\newtheorem{comment}{Comment}

\newcommand{\pre}{p_i}

\newcommand{\errR}{\boldsymbol{\theta}}

\newcommand{\BESTT}{\widehat{\bf T}}
\newcommand{\BESTD}{\widehat{\bf D}}
\newcommand{\CTc}{\widehat{\bf c}}
\newcommand{\CTt}{{\bf c}}

\newcommand{\evar}{\widehat{\bf cov}}
\newcommand{\var}{{\bf cov}}

\newcommand{\diag}{{\bf diag}}

\newcommand{\ra}{{\zeta}}

\newcommand{\zeros}{{\bf 0}}

\newcommand{\inno}{\boldsymbol{\Delta}}

\newcommand{\ub}{{\bf u}}
\newcommand{\vv}{{\bf v}}

\newcommand{\T}{{\bf T}}
\newcommand{\QEST}{{\bf \widehat{\bf Q}}}
\newcommand{\s}{{\bf s}}

\newcommand{\Nens}{{\textnormal{N}_\textnormal{ens}}} %Number of ensemble members
\newcommand{\Nobs}{m} %Number of observations
\newcommand{\Nstate}{n} %Number of components in the vector state
\newcommand{\X}{{\bf X}} %Ensemble
\newcommand{\x}{{\bf x}} %Ensemble member
\newcommand{\lp}{\left (} %Left parenthesis
\newcommand{\rp}{\right )} %Right parenthesisover
\newcommand{\lb}{\left [} %Left block [
\newcommand{\rb}{\right ]} %Right block ]
\renewcommand{\ln}{\big\|} %Left norm
\newcommand{\rn}{\big\|}
\newcommand{\lab}{\left |}
\newcommand{\rab}{\right |}
\newcommand{\B}{{\bf B}} %The background error covariance matrix
\newcommand{\R}{{\bf R}} %The data error covariance matrix
\newcommand{\N}{M} %Number of snapshots
\newcommand{\y}{{\bf y}} %Observations
\renewcommand{\H}{{\bf H}} %Linear operator
\newcommand{\xm}{{\overline{\bf x}}} %Mean
\newcommand{\W}{{\boldsymbol \alpha}} %Vector of weights
\newcommand{\I}{{\bf I}} %Identity matrix
\newcommand{\M}{\mathcal{M}} %Model operator
\newcommand{\Nor}{\mathcal{N}} %Normal distribution
\newcommand{\Y}{{\bf Y}} %Observations
\newcommand{\Ho}{{\mathcal{H}}} %Observational operator
\newcommand{\lle}{\left \{ } %Key left
\newcommand{\rle}{\right \}} %Keyb right

\newcommand{\DX}{{{\boldsymbol \delta} {\bf X}}}
\newcommand{\Q}{{\bf Q}} %Matrix Q = H(BAS)
\newcommand{\BO}[1]{\mathcal{O}\lp #1\rp} %Big Oh operator
\renewcommand{\P}{{\bf P}} %Matrix P holding DX'*DX
\renewcommand{\Re}{\mathbbm{R}}
\renewcommand{\S}{{\bf S}} %Sample covariance matrix in the background
\newcommand{\errobs}{{\boldsymbol \epsilon}} %Observational errors
\newcommand{\errbac}{{\boldsymbol \nu}} %\nuBackground errors
\newcommand{\D}{{\bf D}}

\newcommand{\Z}{{\bf Z}} %POD basis  product H(POD basis)
\newcolumntype{N}{>{\centering\arraybackslash} m{0.30\textwidth} }
\newcolumntype{V}{>{\centering\arraybackslash} m{0.02\textwidth} }

\newcommand{\zero}{{\bf 0}}
\newcommand{\BEST}{\widehat{\bf B}}

\newcommand{\U}{{\bf U}}
\newcommand{\V}{{\bf V}}
\newcommand{\A}{{\bf A}}

\newcommand{\ones}{{\bf 1}}

\begin{document}

\thispagestyle{empty}
\setcounter{page}{0}

\begin{Huge}
\begin{center}
{\bf Computer Science Technical Report CSTR-2/2016 }\\
\today
\end{center}
\end{Huge}
\vfil
\begin{huge}
\begin{center}
Elias D. Ni\~no, Adrian Sandu and Xinwei Deng
\end{center}
\end{huge}

\vfil
\begin{huge}
\begin{it}
\begin{center}
An Ensemble Kalman Filter Implementation Based on Modified Cholesky Decomposition for Inverse Covariance Matrix Estimation
\end{center}
\end{it}
\end{huge}
\vfil

\begin{large}
\begin{center}
Computational Science Laboratory \\
Computer Science Department \\
Virginia Polytechnic Institute and State University \\
Blacksburg, VA 24060 \\
Phone: (540)-231-2193 \\
Fax: (540)-231-6075 \\ 
Email: \url{sandu@cs.vt.edu} \\
Web: \url{http://csl.cs.vt.edu}
\end{center}
\end{large}

\vspace*{1cm}

\begin{tabular}{ccc}
\includegraphics[width=2.5in]{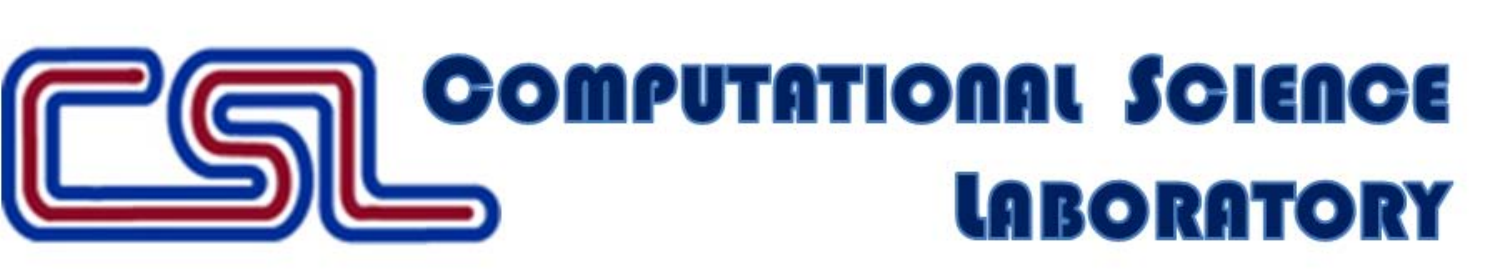}
&\hspace{2.5in}&
\includegraphics[width=2.5in]{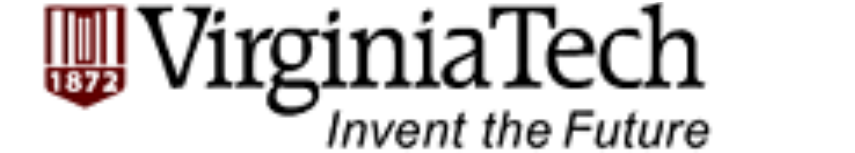} \\
{\bf\em Innovative Computational Solutions} &&\\
\end{tabular}

\newpage

\maketitle

\begin{abstract}

This paper develops an efficient implementation of the ensemble Kalman filter based on a modified Cholesky decomposition for inverse covariance matrix estimation. This implementation is named EnKF-MC. Background errors corresponding to distant model components with respect to some radius of influence are assumed to be conditionally independent. This allows to obtain sparse estimators of the inverse background error covariance matrix. The computational effort of the proposed method is discussed and different formulations based on various matrix identities are provided. Furthermore, an asymptotic proof of convergence with regard to the ensemble size is presented. In order to assess the performance and the accuracy of the proposed method, experiments are performed making use of the Atmospheric General Circulation Model SPEEDY. The results are compared against those obtained using the local ensemble transform Kalman filter (LETKF). Tests are performed for dense observations ($100\%$ and $50\%$ of the model components are observed) as well as for sparse observations (only $12\%$, $6\%$, and $4\%$ of model components are observed). The results reveal that the use of modified Cholesky for inverse covariance matrix estimation can reduce the impact of spurious correlations during the assimilation cycle, i.e., the results of the proposed method are of better quality than those obtained via the LETKF in terms of root mean square error. 
\end{abstract}

%%%%%%%%%%%%%%%%%%%%%%%%%%%%%%%
\section{Introduction}
\label{sec:introduction}
%%%%%%%%%%%%%%%%%%%%%%%%%%%%%%%

The goal of sequential data assimilation is to estimate the true state of a dynamical system $\x^\textnormal{true} \in \Re^{\Nstate \times 1}$ using information from numerical models, priors, and observations. A numerical model captures  (with some approximation) the physical laws of the system and evolves its state forward in time \cite{Cheng2010}:
\begin{eqnarray}
\label{eq:intro-numerical-model}
\x_{k} = \M_{t_{k-1} \rightarrow t_k} \lp \x_{k-1} \rp \in \Re^{\Nstate \times 1},\, \text{ for $\x \in \Re^{\Nstate \times 1}$} ,\,
\end{eqnarray}
where $\Nstate$ is the dimension of the model state, $k$ denotes time index, and $\M$ can represent, for example, the dynamics of the ocean and/or atmosphere. A prior estimation $\x_k^\textnormal{b} \in \Re^{\Nstate \times 1}$ of $\x_k^\textnormal{true}$ is available, and the prior error $\errbac$ is usually assumed to be normally distributed:
\begin{eqnarray}
\label{eq:background-errors}
\displaystyle
\x_k^\textnormal{b} - \x^\textnormal{true} = \errbac_k \sim \Nor \lp \zero,\, \B_k \rp \in \Re^{\Nstate \times 1},
\end{eqnarray}
where $\B_k \in \Re^{\Nstate \times \Nstate}$ is the background error covariance matrix. Noisy observations (measurements) of the true state $\y_k \in \Re^{\Nobs \times 1}$ are taken, and the observation errors $\errobs$ are usually assumed to be normally distributed:
\begin{eqnarray}
\label{eq:noisy-observations}
\displaystyle
\y_k -\Ho \lp \x_k^\textnormal{true} \rp  = \errobs_k  \sim \Nor \lp \zero,\, \R_k \rp\, \in \Re^{\Nobs \times 1} ,\,
\end{eqnarray}
where $\Nobs$ is the number of observed components, $\Ho: \Re^{\Nstate\times 1} \rightarrow \Re^{\Nobs \times 1}$ is the observation operator, and $\R_k \in \Re^{\Nobs \times \Nobs}$ is the data error covariance matrix.  

Making use of Bayesian statistics and matrix identities, the assimilation of the observation \eqref{eq:noisy-observations} is performed as follows:

\begin{equation}
\label{eq:analysis-state}
\begin{split}
\x_k^\textnormal{a} &= \x_k^\textnormal{b} + \B_k \cdot \H_k^T \cdot \lb \H_k \cdot {\B_k} \cdot \H_k^T +\R_k \rb^{-1} \cdot \lb  \y_k - \Ho(\x_k^\textnormal{b}) \rb \in \Re^{\Nstate \times 1} ,\,\\
\A_k &= \lb \I - \B_k \cdot \H_k^T \cdot \lb \R_k+\H_k \cdot \B_k \cdot \H_k^T \rb^{-1} \cdot \H_k \rb \cdot \B_k \in \Re^{\Nstate \times \Nstate},\,
\end{split}
\end{equation}
where $ \H_k \approx \Ho'(\x_k^\textnormal{b})  \in \Re^{\Nobs \times \Nstate}$ is a linear approximation of the observational operator, and $\x_k^\textnormal{a} \in \Re^{\Nstate \times 1}$ is the analysis (posterior) state. 

According to equation \eqref{eq:analysis-state} the elements of $\B_k$ determine how the information about the observed model components contained in the innovations $ \y_k - \Ho(\x_k^\textnormal{b}) \in \Re^{\Nobs \times 1}$ is distributed to  properly adjust all model components, including the unobserved ones. Thus, the successful assimilation of the observation \eqref{eq:noisy-observations} will rely, in part, on how well the background error statistics are approximated. 

In the context of ensemble based methods, an ensemble of model realizations
\begin{eqnarray}
\label{eq:background-ensemble}
\X^\textnormal{b}_k = \lb \x^{\textnormal{b}[1]}_k,\,\x^{\textnormal{b}[2]}_k,\,\ldots,\, \x^{\textnormal{b}[\Nens]}_k\rb \in \Re^{\Nstate \times \Nens} ,\,
\end{eqnarray}
is used in order to estimate the unknown moments of the background error distribution:
\begin{subequations}
\label{eq:moments-ensemble}
\begin{eqnarray}
\label{eq:ensemble-mean}
\displaystyle
\xm^\textnormal{b}_k &=& \frac{1}{\Nens} \cdot \sum_{i=1}^{\Nens} \x_k^{\textnormal{b}[i]} \in \Re^{\Nstate \times 1} ,\\
\label{eq:covariance-ensemble}
\displaystyle
\B_k \approx \P^\textnormal{b} &=& \frac{1}{\Nens-1} \cdot \U_k^\textnormal{b}  \cdot \left(\U_k^\textnormal{b}\right)^T \in \Re^{\Nstate \times \Nstate} ,\,
\end{eqnarray}
where $\Nens$ is the number of ensemble members, $\x_k^{\textnormal{b}[i]} \in \Re^{\Nstate \times 1}$ is the $i$-th ensemble member, $\xm_k^\textnormal{b} \in \Re^{\Nstate \times 1}$ is the background ensemble mean, $\P_k^\textnormal{b}$ is the background ensemble covariance matrix,  and $\U_k \in \Re^{\Nstate \times \Nens}$ is the matrix of member deviations:
\begin{eqnarray}
\label{eq:matrix-of-member-deviations}
\displaystyle
\U_k^\textnormal{b} = \X_k^\textnormal{b} - \xm_k^\textnormal{b} \cdot \ones_{\Nens}^T \in \Re^{\Nstate \times \Nens}.
%\S_k = \frac{1}{\sqrt{\Nens-1}} \cdot \U_k^\textnormal{b}.
\end{eqnarray}
\end{subequations}
One attractive feature of $\P_k^\textnormal{b}$ is its flow-dependency which allows to approximate the background error correlations based on the dynamics of the numerical model \eqref{eq:intro-numerical-model}. However, in operational data assimilation, the number of model components is much larger than the number of model realizations $\Nstate \gg \Nens$ and therefore $\P_k^\textnormal{b}$ is rank-deficient. Spurious correlations (e.g., correlations between distant model components in space) can degenerate the quality of the analysis corrections. One of the most succesful EnKF formulations is the local ensemble transform Kalman filter (LETKF) in which the impact of spurious analysis corrections is avoided by making use of local domain analyses. In this context, every model component is surrounded by a box of a prescribed radius, and then the assimilation is performed within every local box. In this case the background error correlations are provided by the local ensemble covariance matrix. The local analyses  are mapped back onto the global domain to obtain the global analysis state. Nevertheless, when sparse observational networks are considered many boxes can contain no observations, in which case the local analyses coincide with the background. The local box sizes can be increased in order to include observations within the local domains, in which case local analysis corrections can be impacted by spurious correlations. Moreover, in practice, the size of local boxes can be still larger than the number of ensemble members and therefore, the local sample covariance matrix can be rank-deficient.

In order to address the above issues this paper proposes a better estimation of the inverse background error covariance matrix $\B^{-1}$ obtained via a modified Cholesky decomposition. By imposing conditional independence between errors in remote model components we obtain sparse approximations of $\B^{-1}$. 

This paper is organized as follows. In Section \ref{sec:background} ensemble based methods and the modified Cholesky decomposition are introduced. Section \ref{sec:EnKF-MC} discusses the proposed ensemble Kalman filter based on a modified Cholesky decomposition for inverse covariance matrix estimation; a theoretical convergence of the estimator in the context of data assimilation as well as its computational effort are discussed. Section \ref{sec:experimental-settings} presents numerical experiments using the Atmospheric General Circulation Model SPEEDY; the results of the new filter are compared against those obtained by the local ensemble transform Kalman filter. Future work is discussed in Section \ref{sec:future-work} and conclusions are drawn in Section \ref{sec:conclusions}.

%%%%%%%%%%%%%%%%%%%%%%%%%%%%%%%
\section{Background}
\label{sec:background}
%%%%%%%%%%%%%%%%%%%%%%%%%%%%%%%

The ensemble Kalman filter is a sequential Monte Carlo method for state and parameter estimation of  non-linear models such as those found in atmospheric and oceanic sciences \cite{TELA:TELA299,EnKFEvensen,EnKF1657419}. The EnKF popularity is due to its basic theoretical formulation and its relative ease of implementation \cite{EnKFEvensen}. Given the \textit{background} ensemble \eqref{eq:background-ensemble} EnKF builds the \textit{analysis} ensemble as follows:
\begin{subequations}
\begin{eqnarray}
\label{eq:EnKF-analysis-ensemble}
\displaystyle
\X^\textnormal{a} = \X^\textnormal{b} + \P^\textnormal{b} \cdot \H^T \cdot \lb \R + \H \cdot \P^\textnormal{b} \cdot \H^T \rb \cdot \inno  \in \Re^{\Nstate \times \Nens} ,\,
\end{eqnarray}
where:
\begin{eqnarray}
\label{eq:EnKF-innvo-observations}
\displaystyle
\inno = \Y^\textnormal{s} - \Ho (\X^\textnormal{b}) \in \Re^{\Nobs \times \Nens} \,,
\end{eqnarray}
and the matrix of perturbed observations $\Y^\textnormal{s} \in \Re^{\Nobs \times \Nens}$ is:
\begin{equation}
\label{eq:EnKF-synthetic-data}
\begin{split}
\Y^\textnormal{s} &= \lb \y+\errobs^{[1]},\, \y+\errobs^{[2]},\, \ldots, \,  \y+\errobs^{[\Nens]} \rb \in \Re^{\Nobs \times \Nens} \,,\\
\errobs^{[i]} &\sim \Nor \lp \zero,\, \R \rp, \quad 1 \le i \le \Nens \,.
\end{split}
\end{equation}
\end{subequations}
For ease of notation we have omitted the time index superscripts.

The use of perturbed observations \eqref{eq:EnKF-synthetic-data} during the assimilation provides asymptotically correct analysis-error covariance estimates for large ensemble sizes and makes the formulation of the EnKF statistically consistent \cite{Thomas2002}. However, it also has been shown that the inclusion of perturbed observations introduces sampling errors in the assimilation \cite{SamplingErrors1,SamplingErrors2}. 

One of the important problems faced by current ensemble based methods is that spurious correlations between distant components in the physical space lead to spurious analysis corrections. Better approximations of the background error covariance matrix are proposed in the literature in order to alleviate this problem. A traditional approximation of $\B$ is the Hollingworth and Lonnberg method \cite{TELA:TELA460} in which the difference between observations and background states are treated as a combination of background and observations errors. However, this method provides statistics of background errors in observation space, and requires dense observing networks (not the case in practice). Another method has been proposed by Benedetti and Fisher \cite{QJ:QJ37}  based on forecast differences in which the spatial correlations of background errors are assumed to be similar at 24 and 48 hours forecasts. This method can be efficiently implemented in practice, however, it does not perform well in data-sparse regions, and  the statistics provided are a mixture of analysis and background errors. Another way to reduce the impact of spurious correlations is based on adaptive modeling \cite{AdaptativeModeling}. In this context, the model learns and changes with regard to the data collected (i.e., parameters values and model structures). This allows to calibrate, in time, the error subspace rank (i.e., number of empirical orthogonal functions used in the assimilation process), the tapering parameter (i.e.,  local domain sizes), and the ensemble size, among others. Yet another method based on error subspace statistical estimation is proposed in \cite{LemusDA}. This approach develops an evolving error subspace, of variable size, that targets the processes where the dominant errors occur. Then, the dominant errors are minimized in order to estimate the best model state trajectory with regard to the observations. We proposed approximations based on autoregressive error models \cite{Sandu_2007_ARMA} and using hybrid subspace techniques.\cite{Sandu_2010_hybridCovariance}.

Covariance matrix localization artificially reduces correlations between distant model components via a Schur product with a localization matrix $\boldsymbol{\Pi} \in \Re^{\Nstate \times \Nstate}$:
\begin{eqnarray}
\label{eq:EnKF-localized-ensemble-covariance}
\displaystyle
\widehat{\P}^\textnormal{b} = \boldsymbol{\Pi}\circ \P^\textnormal{b} \in \Re^{\Nstate \times \Nstate}
\end{eqnarray}
and then $\P^\textnormal{b}$ is replaced by $\widehat{\P}^\textnormal{b} \in \Re^{\Nstate \times \Nstate}$ in the EnKF analysis equation \eqref{eq:EnKF-analysis-ensemble}. The entries of $\boldsymbol{\Pi}$ decrease with the distance between model components depending on the radius of influence $\ra$:
\begin{eqnarray}
\label{eq:EnKF-covariance-localisation}
\lle \boldsymbol{\Pi} \rle_{i,j} = \exp \lp -\frac{\pi \lp m_i,\,m_j \rp}{f(\ra)} \rp \,, \text{ for $1 \le i \le j \le \Nstate$}\,,
\end{eqnarray}
where $\pi \lp m_i,\,m_j \rp$ represents the physical distance between the model components $m_i$ and $m_j$ while, $f(\ra)$ is a function of $\ra$ (e.g., $f(\ra) = 2 \cdot \ra^2$). The exponential decay allows to reduce the impact of innovations between distant model components. The use of covariance matrix localization alleviates the impact of sampling errors. However, the explicit computation of $\boldsymbol{\Pi}$ (and even $\P^\textnormal{b}$) is prohibitive owing to numerical model dimensions. Thus, domain localization methods \cite{spatial_localization,local_analysis_1} are commonly used in the context of operational data assimilation. One of the best EnKF implementations based on domain localization is the local ensemble transform Kalman filter (LETKF) \cite{LETKFHunt}. In the LETKF the analysis increments are computed in  the space spanned by the ensemble perturbations $\U^\textnormal{b}$ defined in \eqref{eq:matrix-of-member-deviations}. An approximation of the analysis covariance matrix in this space reads:
\begin{subequations}
\label{eq:LETKF-method}
\begin{eqnarray}
\label{eq:LETKF-analysis-covariance-ensemble-space}
\displaystyle
\widehat{\P}^\textnormal{a} = \lb \lp\Nens-1\rp \cdot \I + \Q^T \cdot \R^{-1} \cdot \Q\rb^{-1} \in \Re^{\Nens \times \Nens},\,
\end{eqnarray}
where $\Q = \H \cdot \U^\textnormal{b} \in \Re^{\Nobs \times \Nens}$ and $\I$ is the identity matrix consistent with the dimension. The analysis increments in the subspace are:
\begin{eqnarray}
\W^\textnormal{a} = \widehat{\P}^\textnormal{a} \cdot \Q^T \cdot \R^{-1} \cdot \lb \y - \Ho (\xm^\textnormal{b}) \rb \in \Re^{\Nens \times 1} ,\,
\end{eqnarray}
from which an estimation of the analysis mean in the model space can be obtained:
\begin{eqnarray}
\label{eq:LETKF-analysis-mean}
\displaystyle
\xm^\textnormal{a} = \xm^\textnormal{b} + \U^\textnormal{b} \cdot \W^\textnormal{a} \in \Re^{\Nstate \times 1} .\,
\end{eqnarray}
Finally, the analysis ensemble reads:
\begin{eqnarray}
\label{eq:LETKF-analysis-ensemble}
\x^\textnormal{a} = \xm^\textnormal{a} \cdot \ones_{\Nens}^T + \U^\textnormal{b} \cdot \lb \lp \Nens-1\rp \cdot \widehat{\P}^\textnormal{a} \rb^{1/2} \in \Re^{\Nstate \times \Nens} .\,
\end{eqnarray}
\end{subequations}

The domain localization in the LETKF is performed as follows: each model component is surrounded by a local box of radius $\ra$. Within each local domain the analysis equations \eqref{eq:LETKF-method} are applied, and therefore a local analysis component is obtained. All local analysis components are mapped back onto the model space to obtain the global analysis state. Local boxes for different radii  are shown in Figure \ref{fig:LETKF-radius}. The local sample covariance matrix \eqref{eq:covariance-ensemble} is utilized as the covariance estimator of the local $\B$. This can perform well when small radii $\ra$ are considered during the assimilation step. However, for large values of $\ra$, the analysis corrections can be impacted by spurious correlations since the local sample covariance matrix can be rank deficient. Consequently, the local analysis increments can perform poorly.
\begin{figure}[H]
\centering
\begin{subfigure}{0.3\textwidth}
\centering
\includegraphics[width=0.9\textwidth,height=0.9\textwidth]{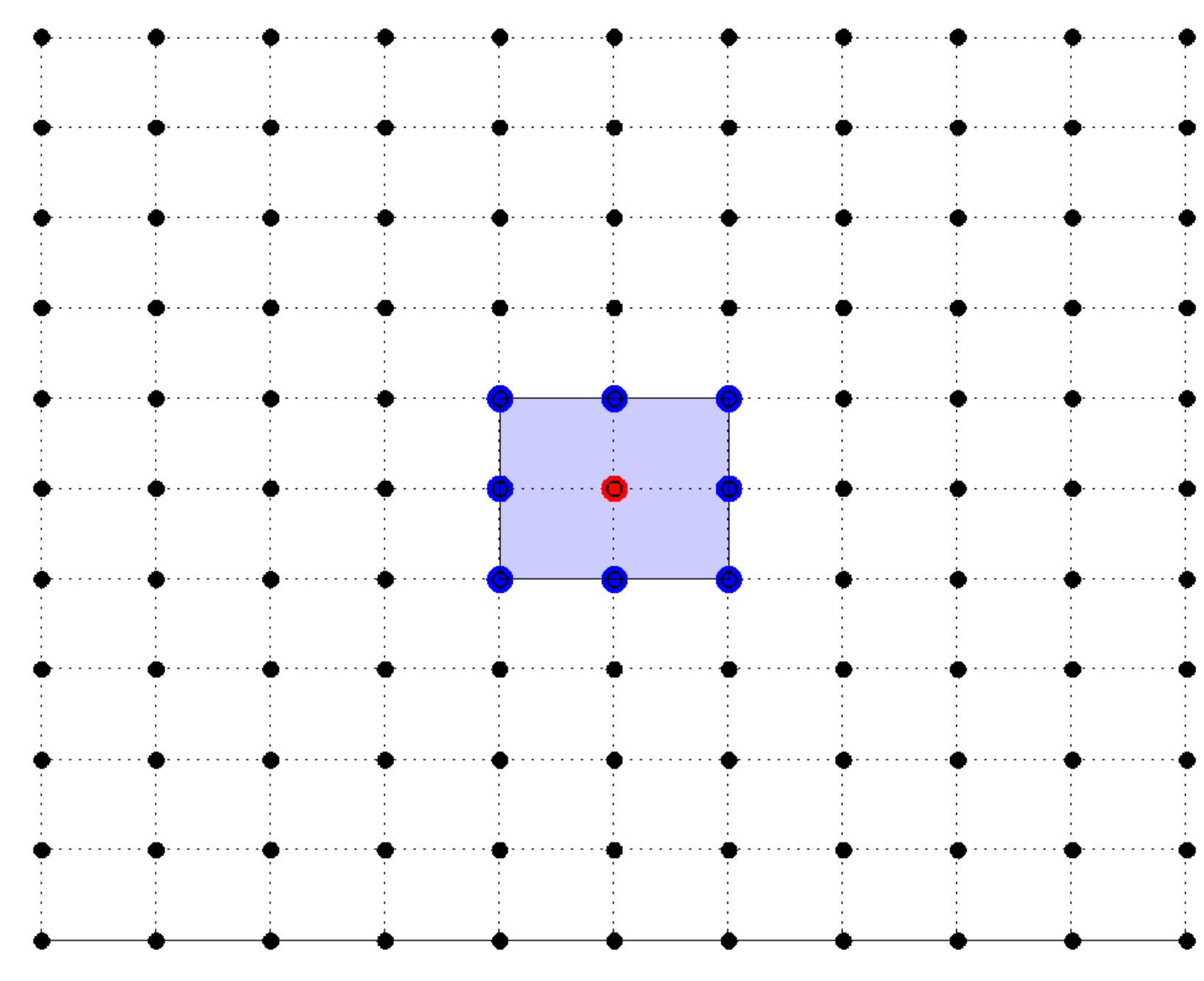}
\caption{$\ra=1$}
\end{subfigure}%
\begin{subfigure}{0.3\textwidth}
\centering
\includegraphics[width=0.9\textwidth,height=0.9\textwidth]{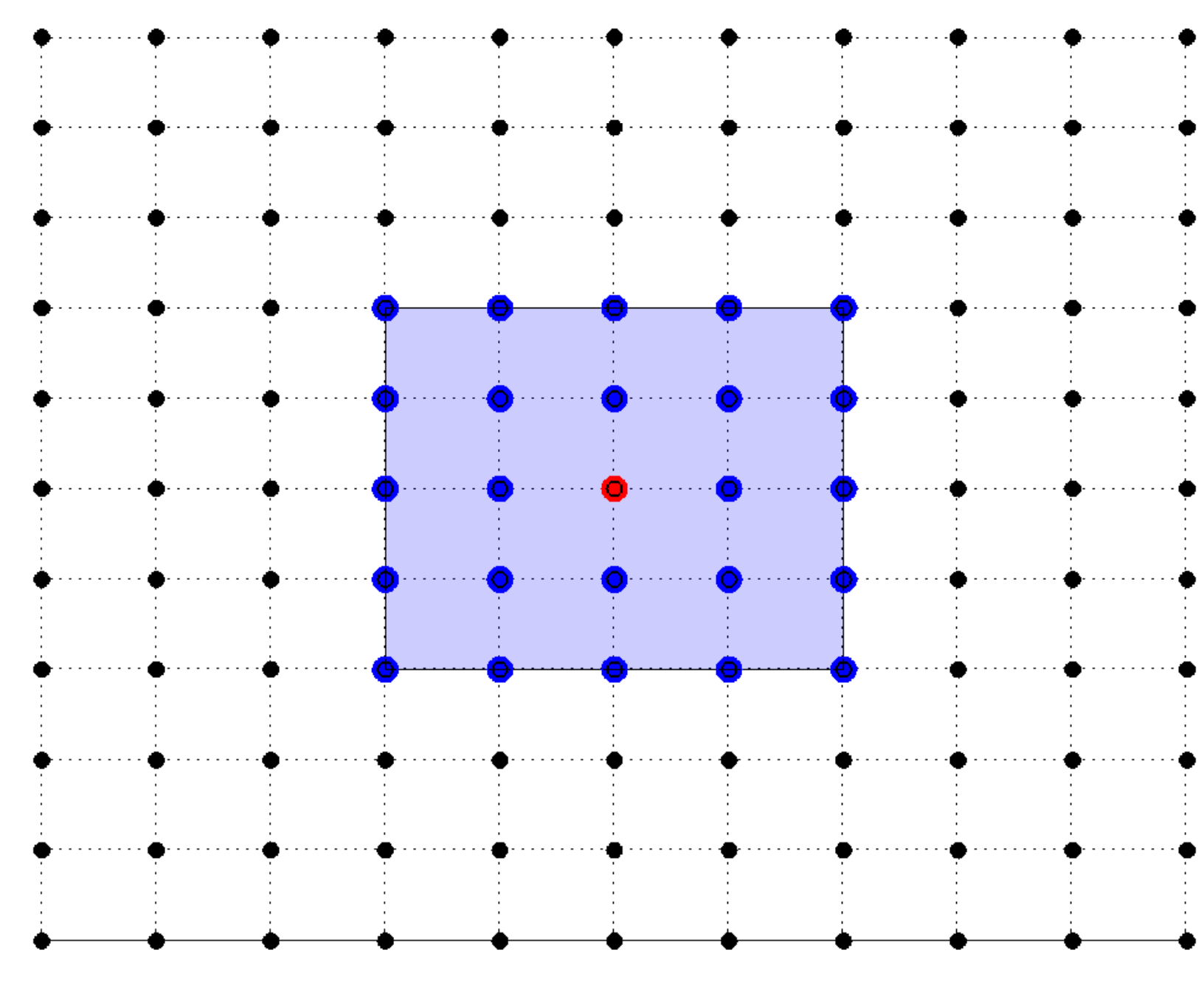}
\caption{$\ra=2$}
\label{subfig:grid-2}
\end{subfigure}
\begin{subfigure}{0.3\textwidth}
\centering
\includegraphics[width=0.9\textwidth,height=0.9\textwidth]{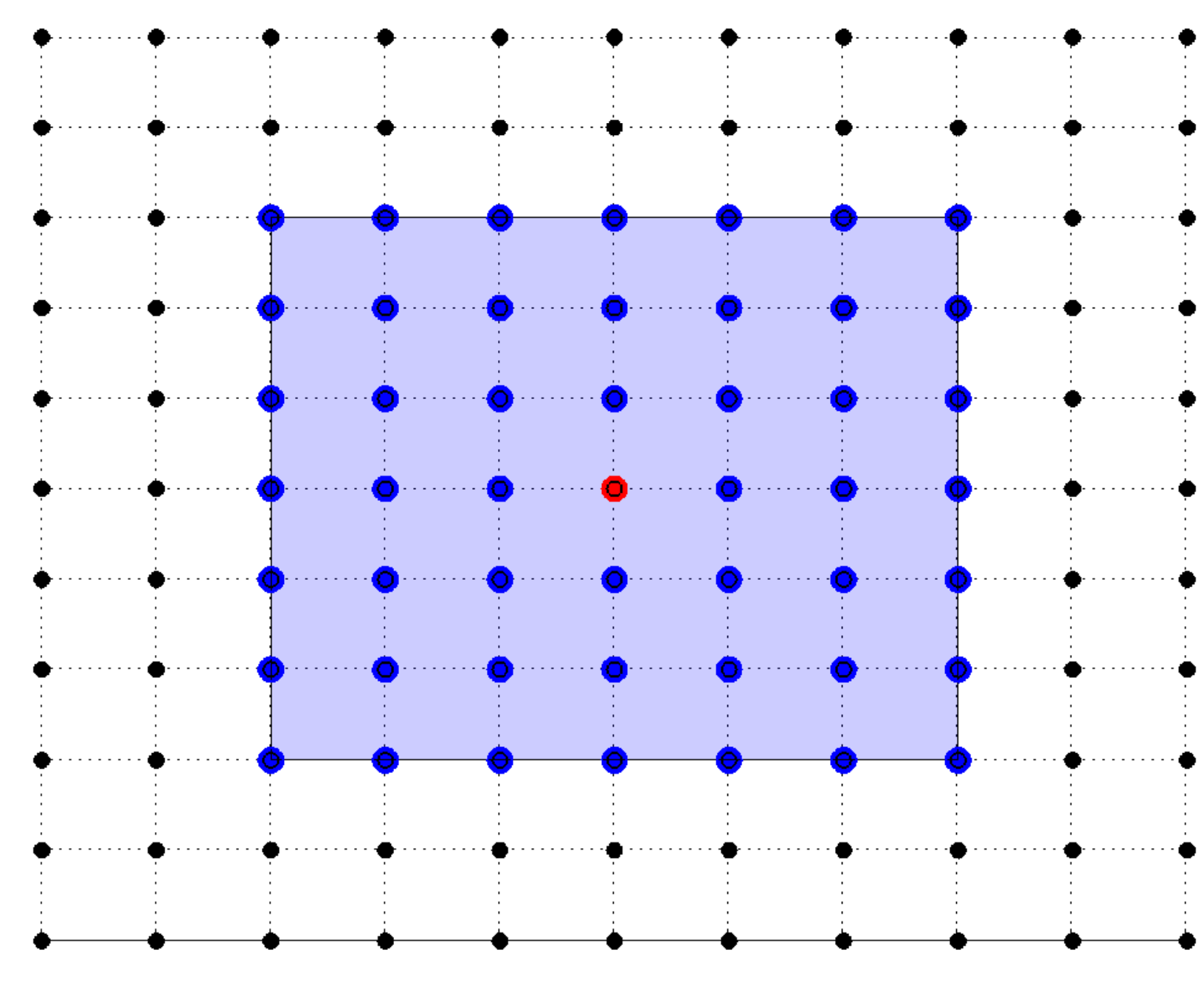}
\caption{$\ra=3$}
\label{subfig:grid-3}
\end{subfigure}
\caption{Local domains for different radii of influence $\ra$. The red dot is the model component to be assimilated, blue components are within the scope of $\ra$, and black model components are unused during the local assimilation process.}
\label{fig:LETKF-radius}
\end{figure}

There is an opportunity to reduce the impact of sampling errors by improving the background error covariance estimation. We achieve this by making use of the modified Cholesky decomposition for inverse covariance matrix estimation \cite{bickel2008}. Consider a sample of $\Nens$ Gaussian random vectors:
\begin{eqnarray*}
\label{eq:intro-samples}
\displaystyle
\S = \lb \s^{[1]},\, \s^{[2]},\, \ldots ,\, \s^{[\Nens]} \rb \in \Re^{\Nstate \times \Nens} ,\,
\end{eqnarray*}
with statistical moments:
\begin{eqnarray*}
\s^{[j]} \sim \Nor \lp \zero_{\Nstate} ,\, \Q \rp, \, \text{for $1 \le j \le \Nens$} ,\,
\end{eqnarray*}
where $\s^{[j]} \in \Re^{\Nstate \times 1}$ denotes the $j$-th sample. Denote by $\x_{[i]} \in \Re^{\Nens \times 1}$ the vector holding the $i$-th component across all the samples (the $i$-th row of $\S$, transposed). The modified Cholesky decomposition arises from regressing each component on his predecessors according to some component ordering:
\begin{eqnarray}
\label{eq:intro-modified-Chokesky-decomposition}
\displaystyle
\x_{[i]} = \sum_{j=1}^{i-1} \x_{[j]} \cdot \beta_{i,j} + \err_{[i]} \in \Re^{\Nens \times 1} ,\,\quad 2 \le i \le \Nstate,
\end{eqnarray}
where $\x_{[j]}$ is the $j$-th model component which precedes $\x_{[i]}$ for $1 \le j \le i-1$, $\err_{[1]} = \x_{[1]}$, and  $\err_{[i]} \in \Re^{\Nens \times 1}$ is the error in the $i$-th component regression for $i \ge 2$. Likewise, the coefficients $\beta_{i,j}$ in \eqref{eq:intro-modified-Chokesky-decomposition} can be computed by solving the optimization problem:
\begin{eqnarray}
\label{eq:intro-coefficient-computation}
\displaystyle
\be_{[i]} = \underset{\be}{\arg\,\min} \ln \x_{[i]} - \Z_{[i]} \cdot \be  \ \rn^2_2
\end{eqnarray}
where 
\begin{eqnarray*}
\displaystyle
\Z_{[i]} &=& \lb \x^{[1]},\, \x^{[2]},\, \ldots,\, \x^{[i-1]}\rb^T \in \Re^{(i-1) \times \Nens} ,\quad 2 \le i  \le \Nstate, \\
\displaystyle
\be_{[i]} &=& \lb \beta_{i,1},\, \beta_{i,2},\, \ldots ,\, \beta_{i,i-1} \rb^{T} \in \Re^{(i-1) \times 1} .\,
\end{eqnarray*}
The regression coefficients form the lower triangular matrix 
\begin{subequations}
\label{eq:intro-empirical-T-D}
\begin{eqnarray}
\label{eq:intro-lower-T}
\big\{ \TE \big\}_{i,j} = \left\{
\begin{aligned}
-\beta_{i,j} & \text{ for }1 \le j < i, \\
1 & \text{ for } j=i, \\
0 & \text{ for } j>i,
\end{aligned} \right.  \quad 1 \le i \le \Nstate,
\end{eqnarray} 
where $\big\{ \TE \big\}_{i,j}$ denotes the $(i,j)$-th component of matrix $\TE \in \mathbb{R}^{\Nstate \times \Nstate}$. The empirical variances $\evar$ of the residuals $\err_{[i]}$ form the diagonal matrix:
\begin{eqnarray}
\label{eq:intro-diagonal-D}
\DE = \underset{1 \le i \le \Nstate}{\textnormal{diag}}\left( \evar ( \err_{[i]}) \right) = \underset{1 \le i \le \Nstate}{\textnormal{diag}}\left( \frac{1}{\Nens-1} \sum_{j=1}^{\Nens} \big\{\err_{[i]} \big\}^2_{j} \right) \in \mathbb{R}^{\Nstate \times \Nstate}\,.
\end{eqnarray} 
\end{subequations}
where $\lle \DE \rle_{1,1} = \evar\lp \x_{[1]}\rp$. Then an estimate of $\Q^{-1}$ can be computed as follows:
\begin{subequations}
\label{eq:intro-Q-approximations}
\begin{eqnarray}
\label{eq:intro-Q-inverse}
\displaystyle
\QEST^{-1} = \TE^T \cdot \DE^{-1} \cdot \TE \in \Re^{\Nstate \times \Nstate} ,\,
\end{eqnarray}
or, by basic matrix algebra identities the estimate of $\Q$ reads:
\begin{eqnarray}
\displaystyle
\QEST = \TE^{-1} \cdot \DE \cdot \TE^{-T} \in \Re^{\Nstate \times \Nstate} \,.
\end{eqnarray}
\end{subequations}
Note that the structure of $\QEST^{-1}$ is strictly related to the structure of $\TE$. This can be exploited in order to obtain sparse estimators of $\Q^{-1}$ by imposing that some entries of $\TE$ are zero. This is important for high dimensional probability distributions where the explicit computation of $\QEST$ or $\QEST^{-1}$ is prohibitive. The zero components in $\TE$ can be justified as follows: when two components are {\it conditionally independent} their corresponding entry in $\QEST^{-1}$ is zero. In the context of data assimilation, the conditional independence of background errors between different model components can be achieved by making use of domain localization. We can consider zero correlations between background errors corresponding to model components located at distances that exceed a radius of influence $\ra$. In the next section we present an ensemble Kalman filter implementation based on modified Cholesky decomposition for inverse covariance matrix estimation.

%%%%%%%%%%%%%%%%%%%%%%%%%%%%%%%
\section{Ensemble Kalman Filter Based On Modified Cholesky Decomposition}
\label{sec:EnKF-MC}
%%%%%%%%%%%%%%%%%%%%%%%%%%%%%%%

In this section we discuss the new ensemble Kalman filter based on modified Cholesky decomposition for inverse covariance matrix estimation ( EnKF-MC).

%%%%%%%%%%%%%%%%%%%%%%%%%%%%%%%%%%%%%%%
\subsection{Estimation of the inverse background covariance}
%%%%%%%%%%%%%%%%%%%%%%%%%%%%%%%%%%%%%%%

The columns of matrix \eqref{eq:matrix-of-member-deviations}
\begin{eqnarray*}
\U^\textnormal{b} = \lb \ub^{\textnormal{b}[1]},\,\ub^{\textnormal{b}[2]},\, \ldots ,\, \ub^{\textnormal{b}[\Nens]}   \rb \in \Re^{\Nstate \times \Nens}
\end{eqnarray*}
can be  seen as samples of the (approximately normal) distribution:
\begin{eqnarray*}
\x^{\textnormal{b}[j]} - \xm^\textnormal{b} = \ub^{\textnormal{b}[j]} \sim \Nor \lp \zeros,\, \B \rp,\, \text{ for $1 \le j \le \Nens$} \,,
\end{eqnarray*}
and therefore, if we let $\x_{[i]} \in \Re^{\Nens \times 1}$ in \eqref{eq:intro-modified-Chokesky-decomposition}  to be the vector formed by the $i$-th row of matrix \eqref{eq:matrix-of-member-deviations}, for $1 \le i \le \Nstate$, according to equations \eqref{eq:intro-Q-approximations}, an estimate of the inverse background error covariance matrix reads:
\begin{subequations}
\label{eq:EnKF-MC-background-cov}
\begin{eqnarray}
\label{eq:EnKF-MC-inverse-background-cov}
\displaystyle \B^{-1} \approx \BE^{-1}  = \TE^T \cdot \DE^{-1} \cdot \TE \in \Re^{\Nstate \times \Nstate} ,\,
\end{eqnarray}
and therefore:
\begin{eqnarray}
\label{eq:EnKF-MC-background-cov}
\displaystyle
\B \approx \BE = \TE^{-1} \cdot \DE \cdot \TE^{-T} \in \Re^{\Nstate \times \Nstate} \,.
\end{eqnarray}
\end{subequations}
As we mentioned before, the structure of $\BE^{-1}$ depends on that of $\TE$. If we assume that the correlations between model components are local, and there are no correlations outside a radius of influence $\ra$, we obtain lower-triangular sparse estimators of $\TE$. Consequently, the resulting $\BE^{-1}$ will also be sparse, and $\BE$ will be localized. Since the regression \eqref{eq:intro-modified-Chokesky-decomposition} is performed only on the predecessors of each model component, an ordering (labeling) must be set on the model components prior the computation of $\TE$. Since we work with gridded models we consider column-major and row-major orders. They are illustrated in Figure \ref{fig:ordering} for a two-dimensional domain. Figure \ref{fig:predecessors} shows the local domain and the predecessors of the model component 6 when column-major order is utilized.
\begin{figure}[H]
\centering
\begin{subfigure}{0.4\textwidth}
\centering
\includegraphics[width=0.6\textwidth,height=0.6\textwidth]{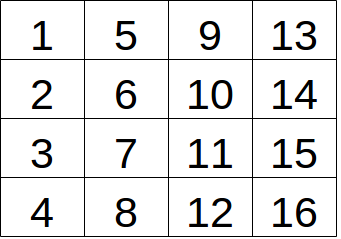}
\caption{Column-major order}
\end{subfigure}%
\begin{subfigure}{0.4\textwidth}
\centering
\includegraphics[width=0.6\textwidth,height=0.6\textwidth]{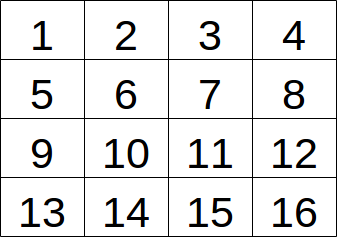}
\caption{Row-major order}
\label{subfig:row-major-order}
\end{subfigure}
\caption{Row-major and column-major ordering for a $4 \times 4$ domain. The total number of model components is 16. }
\label{fig:ordering}
\end{figure}
\begin{figure}[H]
\centering
\begin{subfigure}{0.4\textwidth}
\centering
\includegraphics[width=0.6\textwidth,height=0.6\textwidth]{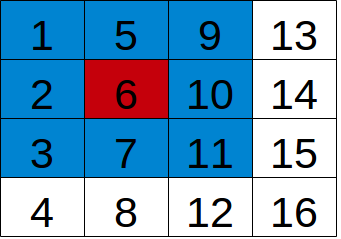}
\caption{In blue, local box for the model component 6 when $\ra=2$.}
\end{subfigure} \hspace{1em}
\begin{subfigure}{0.4\textwidth}
\centering
\includegraphics[width=0.6\textwidth,height=0.6\textwidth]{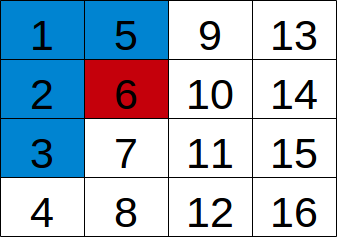}
\caption{In blue, predecessors of the model component 6 for $\ra=2$.}
\end{subfigure}
\caption{Local model components (local box) and local predecessors for the model component 6 when $\ra=2$. Column-major ordering is utilized to label the model components.}
\label{fig:predecessors}
\end{figure}

The estimation of $\BE^{-1}$ proceeds as follows:
\begin{enumerate}
\item Form the matrix $\Z_{[i]} \in \Re^{\pre \times \Nens}$ with the predecessors of the $i$-th model component:
\begin{eqnarray}
\label{eq:EnKF-MC-predecessors}
\Z_{[i]} = \lb \x^{[q(i,1)]},\,\x^{[q(i,2)]},\, \ldots ,\, \x^{[q(i,\pre)]} \rb^T \in \Re^{\pre \times \Nens} \,,
\end{eqnarray}
where $\x^{[e]}$ is the $e$-th row of matrix \eqref{eq:matrix-of-member-deviations}, $\pre$ is the number of predecessors of component $i$, and $1 \le q(i,j) \le \Nstate$ is the index (row of matrix \eqref{eq:matrix-of-member-deviations}) of the $j$-th predecessor of the $i$-th model component. 

\item For the $i$-th model components the regression coefficients are obtained as follows:
\begin{eqnarray*}
\x_{[i]} = \sum_{j=1}^{\pre} \beta_{i,j} \cdot \x^{[q(i,j)]} + \err_{[i]} \in \Re^{\Nens \times 1} \,.
\end{eqnarray*}
For $2 \le i \le \Nstate$, compute $\be_{[i]} = [\beta_{i,1},\,\beta_{i,2},\,\ldots,\,\beta_{i,\pre}] \in \Re^{\pre \times 1}$ by solving the optimization problem \eqref{eq:intro-coefficient-computation} with $\Z_{[i]}$ given by \eqref{eq:EnKF-MC-predecessors}.

\item Build the matrices 
\begin{eqnarray*}
\big\{ \TE \big\}_{i,q(i,j)}=  -\beta_{i,j} ~~ \text{ for}~~1 \le j \le \pre, ~~1 < i \le \Nstate \,; \quad \big\{ \TE \big\}_{i,i}=1,
\end{eqnarray*}
and $\DE$ according to equation \eqref{eq:intro-diagonal-D}. Note that the number of non-zero elements in the $i$-th row of $\TE$  equals the number of predecessors $\pre$. 

\end{enumerate}

Note that the solution of the optimization problem \eqref{eq:intro-coefficient-computation} can be obtained as follows:
\begin{eqnarray}
\label{eq:EnKF-MC-solution-of-optimization-problem}
\be_{[i]} = \lb \Z_{[i]} \cdot { \Z_{[i]} }^T \rb^{-1} \cdot \Z_{[i]} \cdot \x_{[i]} 
\end{eqnarray}
and since the ensemble size can be smaller than the number of model components, $\Z_{[i]} \cdot { \Z_{[i]} }^T \in \Re^{\pre \times \pre}  $ can be rank deficient. To overcome this situation, regularization of the zero singular values of $\Z_{[i]} \cdot  {\Z_{[i]}}^T$ can be used. One possibility is Tikhonov regularization \cite{Tik1,Tik2,Tik3}:
\begin{eqnarray}
\label{eq:EnKF-MC-optimization-tik}
\be_{[i]} = \underset{\be}{\arg\,\min} \lle \ln \x_{[i]} - \Z_{[i]} \cdot \be   \rn^2_2 + \lambda^2 \cdot \ln  \be \rn_2^2 \rle
\end{eqnarray}
where $\lambda \in \Re$. In our context the best choice for $\lambda$ relies on prior knowledge of the background and the observational errors \cite{Tik4}. Another approach to regularization is to use a truncated singular value decomposition (SVD) of $\Z_{[i]}$:
\begin{eqnarray*}
\Z_{[i]} = \U^{\Z_{[i]}} \cdot \Eig^{\Z_{[i]}} \cdot {\V^{\Z_{[i]}}}^T \in \Re^{\pre \times \Nens},
\end{eqnarray*}
where $\U^{\Z_{[i]}} \in \Re^{\pre \times \pre}$ and $\V^{\Z_{[i]}} \in \Re^{\Nens \times \Nens}$ are the right and the left singular vectors of $\Z_{[i]}$, respectively. Likewise, $\Eig^{\Z_{[i]}} \in \Re^{\pre \times \Nens}$ is a diagonal matrix whose diagonal entries are the singular values of $\Z_{[i]}$ in descending order. The solution of \eqref{eq:intro-coefficient-computation} can be computed as follows \cite{Jiang2000137,VANHUFFEL1991675,Per1}:
\begin{eqnarray}
\label{eq:EnKF-MC-truncated-SVD}
\be_{[i]} =\sum_{j=1}^{k_i} \frac{1}{\tau_j} \cdot \ub^{\Z_{[i]}}_j \cdot {\vv_j^{\Z_{[i]}}}^T \cdot \x_{[i]} \quad \text{with}\,\frac{\tau_{j}}{\tau_{\rm max}}  \ge \sigma_r,\,
\end{eqnarray}
where $\tau_j$ is the $j$-th singular value with corresponding right and left singular vectors $\ub^{\Z_{[i]}}_j \in \Re^{\pre \times 1}$ and $\vv^{\Z_{[i]}}_j \in \Re^{\pre \times 1}$, respectively, $\sigma_r \in (0,1)$ is a predefined threshold, and  $\tau_{\rm max} = \max\lle  \tau_1,\, \tau_2,\, \ldots,\, \tau_{\Nens-1}\rle$. Since small singular values are more sensitive to the noise in $\x_{[i]}$, the threshold $\tau_j > \tau_{\max} \cdot \sigma_r$ seeks to neglect their contributions.

%%%%%%%%%%%%%%%%%%%%%%%%%%%%%%%%%%%%%%%
\subsection{Formulation of EnKF-MC}
%%%%%%%%%%%%%%%%%%%%%%%%%%%%%%%%%%%%%%%

Once $\BE^{-1}$ is estimated, the EnKF based on modified Cholesky decomposition (EnKF-MC) computes the analysis using Kalman's formula:
\begin{subequations}
\label{eq:EnKF-MC-formulations}
\begin{eqnarray}
\label{eq:EnKF-MC-primal-incremental}
\x^\textnormal{a} &=& \x^\textnormal{b} + \AE \cdot \H^T \cdot \R^{-1} \cdot \inno \in \Re^{\Nstate \times \Nens},
\end{eqnarray}
where $\AE \in \Re^{\Nstate \times \Nstate}$ is the estimated analysis covariance matrix
\begin{eqnarray*}
\displaystyle
\AE = \lb \BE^{-1} +\H^T \cdot \R^{-1} \cdot \H \rb^{-1}\,,
\end{eqnarray*}
and $\inno \in \Re^{\Nobs \times \Nens}$ is the innovation matrix on the perturbed observations given in \eqref{eq:EnKF-innvo-observations}. 

Computationally-friendlier alternatives to \eqref{eq:EnKF-MC-dual} can be obtained by making use of elementary matrix identities:
\begin{eqnarray}
\label{eq:EnKF-MC-primal}
\displaystyle
\x^\textnormal{a} &=& \AE \cdot \lb \BEST^{-1} \cdot \x^\textnormal{b} + \H^T \cdot \R^{-1} \cdot \Y^\textnormal{s} \rb \in \Re^{\Nstate \times \Nens},\, \\
\label{eq:EnKF-MC-dual}
\displaystyle
\x^\textnormal{a} &=& \x^\textnormal{b} + \TE^{-1} \cdot \DE^{1/2} \cdot \V_{\BE}^T \cdot \lb \R + \V_{\BE} \cdot \V_{\BE}^T \rb^{-1} \cdot \inno,\,  \\
\notag
\V_{\BE} &=& \H  \cdot \TE^{-1} \cdot \DE^{1/2} \in \Re^{\Nstate \times \Nobs},
\end{eqnarray}
\end{subequations}
where $\Y^\textnormal{s}$ are the perturbed observations. The formulation \eqref{eq:EnKF-MC-dual} is well-known as the EnKF dual formulation, \eqref{eq:EnKF-MC-primal} is known as the EnKF primal formulation, and the equation \eqref{eq:EnKF-MC-primal-incremental} is the incremental form of the primal formulation. In the next subsection, we discuss the computational effort of the EnKF-MC implementations \eqref{eq:EnKF-MC-formulations}.

%%%%%%%%%%%%%%%%%%%%%%%%%%%%%%%%%%%%%%%
\subsection{Computational effort of EnKF-MC implementations}
%%%%%%%%%%%%%%%%%%%%%%%%%%%%%%%%%%%%%%%

The computational cost of the different EnKF-MC implementations depend, in general, on the model state dimension $\Nstate$, the number of observed components $\Nobs$, the radius of influence $\ra$, and the ensemble size $\Nens$. Typically \cite{Tippett2003} the data error covariance matrix $\R$ has a simple structure (e.g., block diagonal), the ensemble size is much smaller than the model dimension ($\Nstate \gg \Nens$), and the observation operator $\H$ is sparse or can be applied efficiently. We analyze the computational effort of the formulation \eqref{eq:EnKF-MC-primal-incremental}; similar analyses can be carried out for the other formulations. The incremental formulation can be written as follows:
\begin{eqnarray*}
\x^\textnormal{a} = \x^\textnormal{b} + \DX^\textnormal{a} \,,
\end{eqnarray*}
where the analysis increments $\DX^\textnormal{a} \in \Re^{\Nstate \times \Nens}$ are given by the solution of the linear system:
\begin{eqnarray*}
\lb \BE^{-1} + \R_{\H} \cdot \R_{\H}^T\rb \cdot \DX^\textnormal{a} = \inno_{\H} \,.
\end{eqnarray*}
with $\R_{\H} = \H^T \cdot \R^{-1/2} \in \Re^{\Nstate \times \Nobs}$, $\inno_{\H} = \H^T \cdot \R^{-1} \cdot \inno \in \Re^{\Nstate \times \Nens}$, and $\inno$ is given in \eqref{eq:EnKF-innvo-observations}. This linear system can be solved making use of the iterative Sherman Morrison formula \cite{Nino2014} as follows:
\begin{enumerate}
\item Compute:
\begin{subequations}
\label{eq:EnKF-MC-linear-system-primal}
\begin{eqnarray}
 \WV^{(0)[i]}_{\Z} &=& \lb \TE^T \cdot \D^{-1} \cdot \TE \rb^{-1} \cdot \inno_{\H}^{[i]} ,\, \text{ for $1 \le i \le \Nens$},\, \\
  \WV^{(0)[j]}_{\U} &=& \lb \TE^T \cdot \D^{-1} \cdot \TE \rb^{-1} \cdot \R_{\H}^{[j]} ,\, \text{ for $1 \le j \le \Nobs$} \,.
\end{eqnarray}
\end{subequations}
where $\inno_{\H}^{[i]} \in \Re^{\Nstate \times 1}$ and $\R_{\H}^{[j]} \in \Re^{\Nstate \times 1}$ denote the $i$ and $j$ columns of matrices $\inno_{\H}$ and $\R_{\H}$, respectively. Since $\TE$ is a sparse unitary lower triangular matrix, the direct solution of the linear system \eqref{eq:EnKF-MC-linear-system-primal} can be obtained by making use of forward and backward substitutions. Hence, this step can be performed with:
\begin{eqnarray}
\label{eq:EnKF-MC-computational-effort-step-1}
\displaystyle
\BO{\Nstate_{nz} \cdot \Nstate \cdot \Nens + \Nstate_{nz} \cdot \Nstate \cdot \Nobs }
\end{eqnarray}
long computations, where $\Nstate_{nz}$ denotes the maximum number of non-zero elements across all rows of $\TE$, this is
\begin{eqnarray*}
\Nstate_{nz} = \max \lle p_1,\,p_2,\,\ldots,\,p_{\Nstate} \rle
\end{eqnarray*}
where $p_i$ is the number of predecessors of model component $i$, for $1 \le i \le \Nstate$.
\item For $1 \le i \le \Nobs$ compute:
\begin{eqnarray*}
\h^{(i)} &=& \frac{1}{\gamma^{(i)}} \cdot \WV_{\U}^{(i-1)[i]},\, \text{with } \gamma^{(i)} =\lb 1+{\R_{\H}^{[i]}}^T \cdot \WV_{\U}^{(i-1)[i]} \rb^{-1} \,, \\
\WV_{\Z}^{(i)[j]} &=& \WV_{\Z}^{(i-1)[j]} - \h^{(i)} \cdot \lb {\R_{\H}^{[i]}}^T \cdot \WV_{\Z}^{(i-1)[j]} \rb ,\, \text{ for $1 \le j \le \Nens$}\,, \\
\WV_{\U}^{(i)[k]} &=& \WV_{\U}^{(i-1)[k]} - \h^{(i)} \cdot \lb {\R_{\H}^{[i]}}^T \cdot \WV_{\U}^{(i-1)[k]} \rb ,\, \text{ for $i+1 \le k \le \Nobs$}\,. \\
\end{eqnarray*}
Note that, at each step, $\h^{(i)}$ can be computed with $\Nstate$ long computations, while $\WV_{\Z}$ and $\WV_{\U}$ can be obtained with $\Nstate \cdot \Nens$ and $\Nstate \cdot \Nobs$ long computations, respectively. This leads to the next bound for the number of long computations:
\begin{eqnarray*}
\BO{\Nobs \cdot \Nstate+\Nobs \cdot \Nstate \cdot \Nens + \Nobs^2 \cdot \Nstate} \,.
\end{eqnarray*}
\end{enumerate}
Hence, the computational effort involved during the assimilation step of formulation \eqref{eq:EnKF-MC-primal-incremental} can be bounded by:
\begin{eqnarray*}
\BO{\Nobs \cdot \Nstate+\Nobs \cdot \Nstate \cdot \Nens + \Nobs^2 \cdot \Nstate+\Nstate_{nz} \cdot \Nstate \cdot \Nens + \Nstate_{nz} \cdot \Nstate \cdot \Nobs},\,
\end{eqnarray*}
which is linear with respect to the number of model components. For dense observational networks, when local observational operators can be approximated, domain decomposition can be exploited in order to reduce the computational effort during the assimilation cycle. This can be done as follows:
\begin{enumerate}
\item The domain is split in certain number of sub-domains (typically matching a given number of processors).
\item Background error correlations are estimated locally.
\item The assimilation is performed on each local domain.
\item The analysis sub-domains are mapped back onto the model domain from which the global analysis state is obtained.
\end{enumerate}
Figure \ref{fig:sub-domain-splitting} shows the global domain splitting for different sub-domain sizes. In Figure \ref{subfig:boundary-information} the boundary information needed during the assimilation step for two particular sub-domains is shown in dashed blue lines. Note that each sub-domain can be assimilated independently. Note that we only use domain decomposition in order to reduce the computational effort of the proposed implementation (and its derivations) and not in order to reduce the impact of spurious correlations. 
 \begin{figure}[H]
\centering
\begin{subfigure}{0.5\textwidth}
\centering
\includegraphics[width=0.9\textwidth,height=0.625\textwidth]{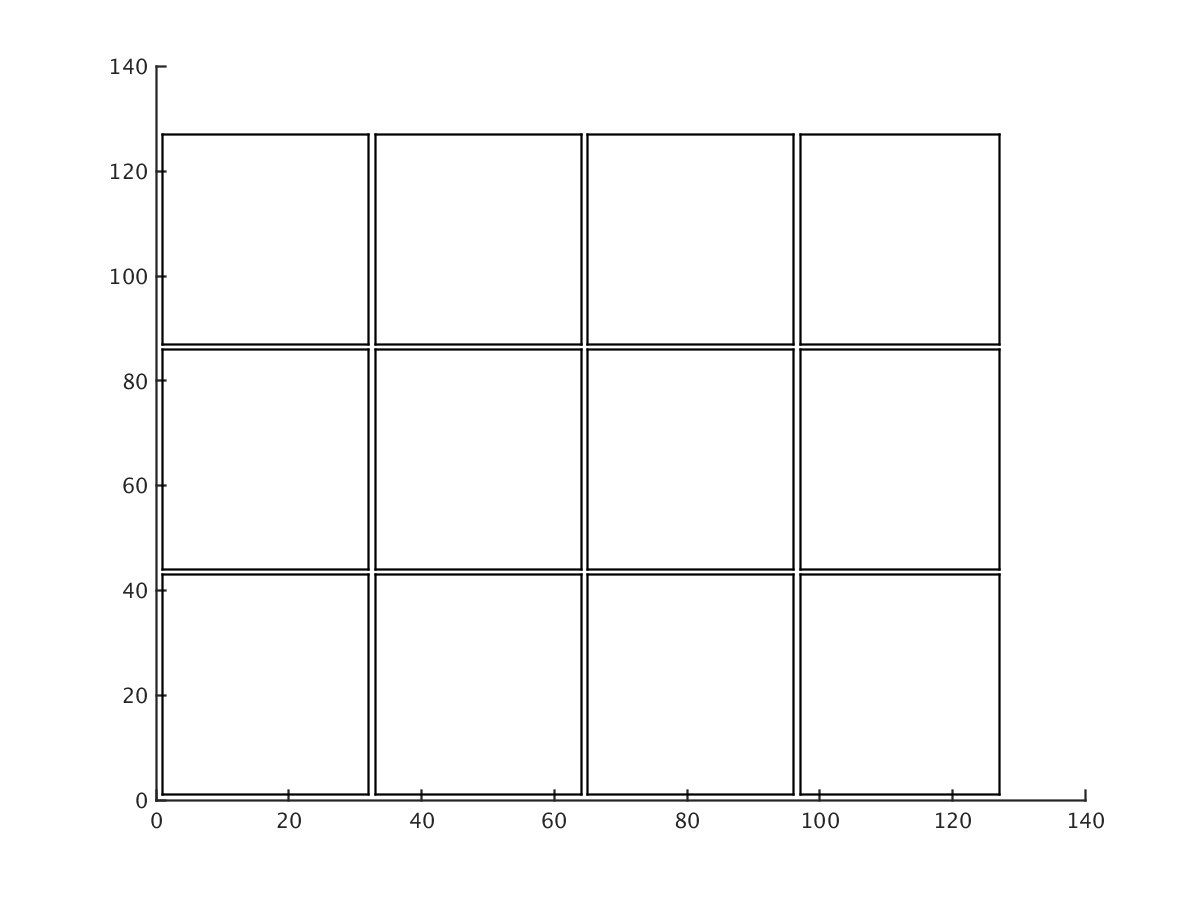}
\caption{Number of sub-domains 12}
\end{subfigure}%
\begin{subfigure}{0.5\textwidth}
\centering
\includegraphics[width=0.9\textwidth,height=0.625\textwidth]{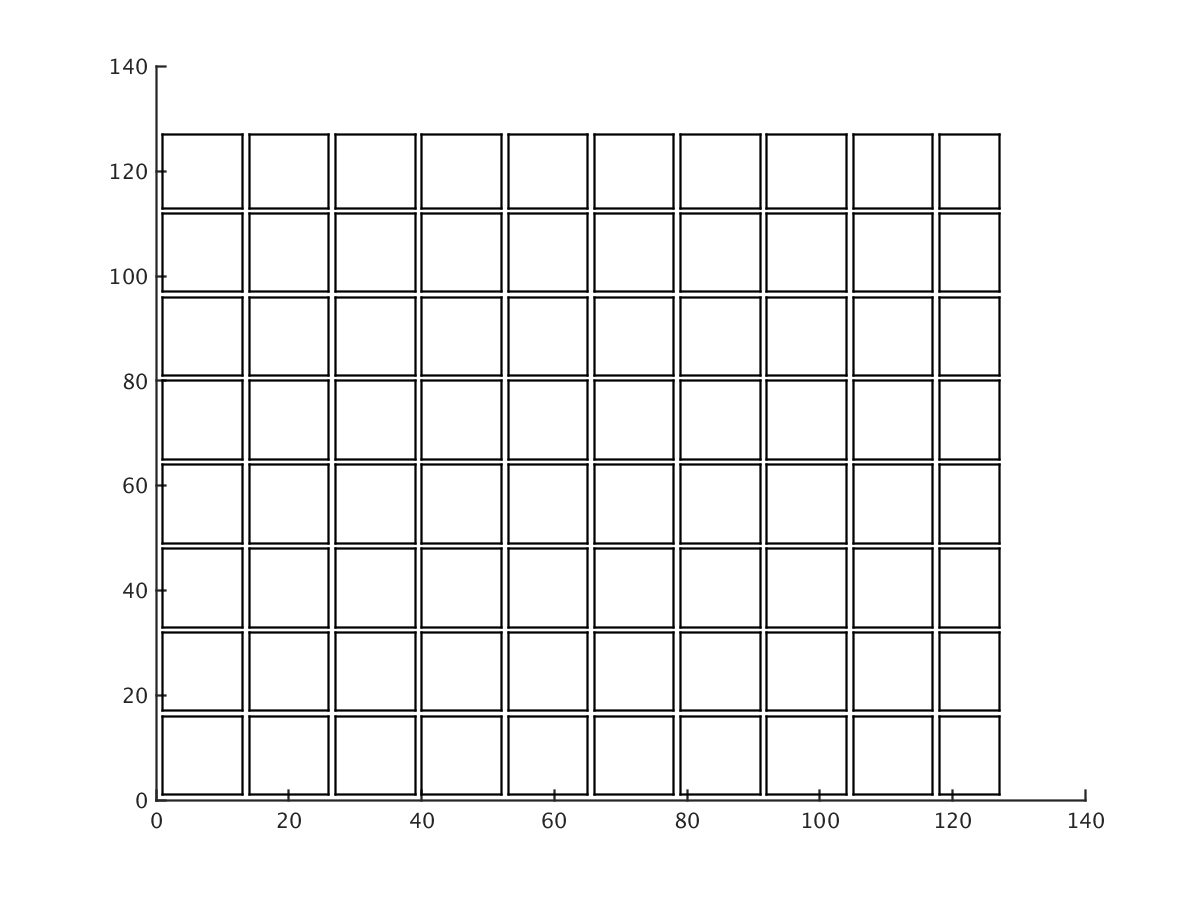}
\caption{Number of sub-domains 80}
\end{subfigure}

\begin{subfigure}{0.5\textwidth}
\centering
\includegraphics[width=0.9\textwidth,height=0.9\textwidth]{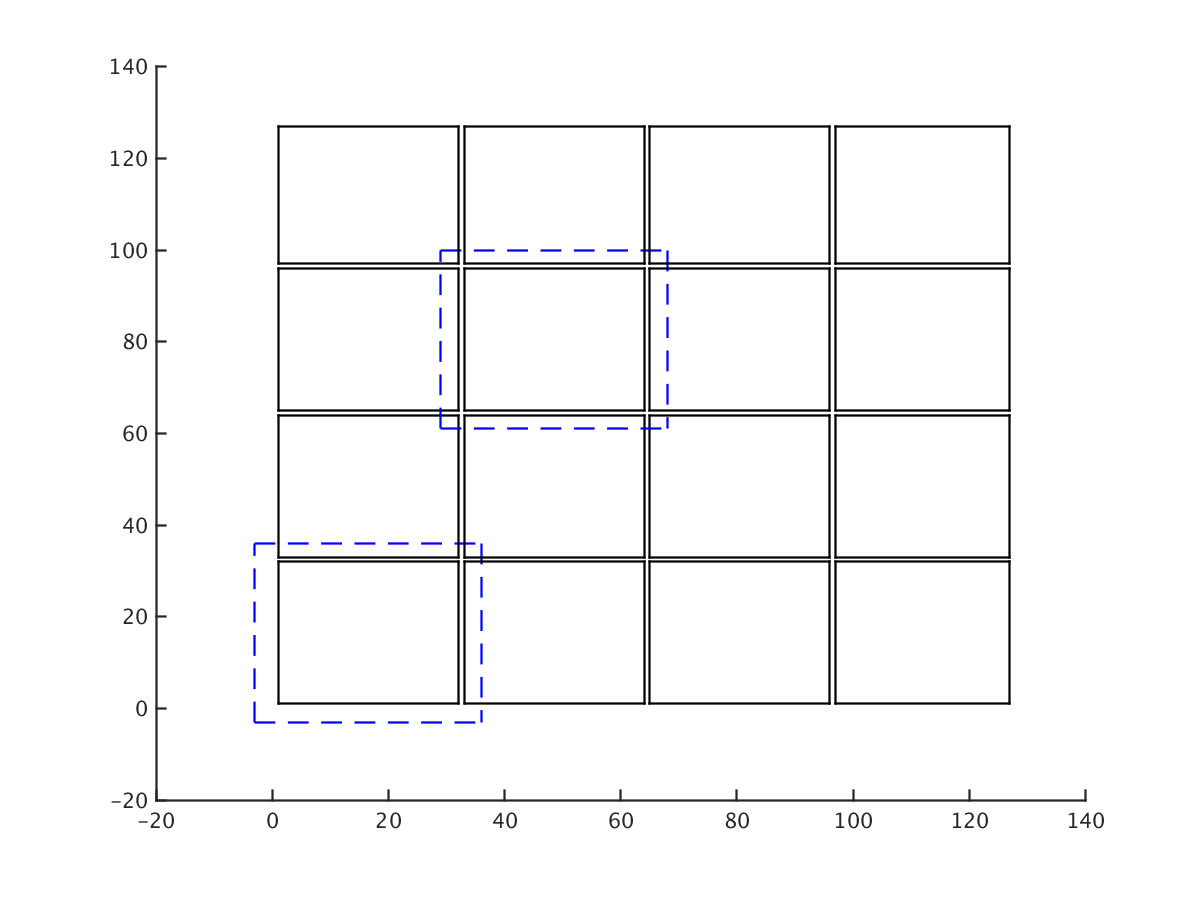}
\caption{Number of sub-domains 16.}
\label{subfig:boundary-information}
\end{subfigure}%
\caption{Global domain splitting in different sub-domain sizes. Blue local boxes reflects the boundary information utilized in order to perform local data assimilation.}
\label{fig:sub-domain-splitting}
\end{figure}

%%%%%%%%%%%%%%%%%%%%%%%%%%%%%%%%%%%%%%%
\subsection{Convergence of the covariance inverse estimator}
\label{subsec:proof-of-convergence}
%%%%%%%%%%%%%%%%%%%%%%%%%%%%%%%%%%%%%%%

In this section we prove the convergence of the $\BE^{-1}$ estimator in the context of data assimilation. 

\begin{comment}[Sparse Cholesky factors and localization]
The modified Cholesky decomposition for inverse covariance matrix estimation can be seen as a form of covariance matrix localization method in which the resulting matrix approximates the inverse of a localized ensemble covariance matrix. This process is implicit in the resulting estimator when only a local neighborhood for each model component is utilized in order to perform the local regression and to estimate $\TE$ and $\DE$. Figure \ref{fig:theo-Cholesky-factors} shows an example for the Lorenz 96 \cite{LorenzModel}:
\begin{eqnarray}
\label{eqch4:Lorenz-model}
\displaystyle
\frac{dx_j}{dt} = \begin{cases}
\lp x_2-x_{\Nstate-1}\rp \cdot x_{\Nstate} -x_1 +F & \text{ for $j=1$} \\
\lp x_{j+1}-x_{j-2}\rp \cdot x_{j-1} -x_j +F & \text{ for $2 \le j \le \Nstate-1$} \\
\lp x_1-x_{\Nstate-2}\rp \cdot x_{\Nstate-1} -x_{\Nstate} +F & \text{ for $j=\Nstate$} 
\end{cases}
\end{eqnarray}
where $F$ is usually set to $8$ to exhibit chaotic behavior and the number of model components is $\Nstate=40$. We assume $\B$ to be a sample covariance matrix based on $10^5$ samples, the localized ensemble covariance matrix $\P^\textnormal{b}$ and the estimator $\BE^{-1}$ are based on 80 samples. The radius of influence is $\ra = 7$. The similarities among the different Cholesky factors is evident. Even more, along the main diagonal, the correlations decay with respect to the distance of the model components. This is reflected in the resulting estimator of $\B^{-1}$ for each case. Definition \ref{theo-def} of covariance matrices relies on this assumpation for the Cholesky factors $\T$ and $\TE$.
\end{comment}

\begin{figure}[H]
\centering
\begin{subfigure}{0.5\textwidth}
\centering
\includegraphics[width=1\textwidth,height=0.8\textwidth]{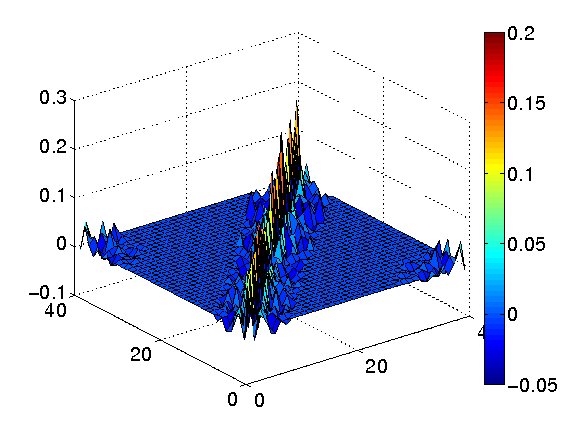}
\caption{Exact $\B^{-1} \approx {\P^b}^{-1}$ for $\Nens={10^5}$}
\end{subfigure}%
\begin{subfigure}{0.5\textwidth}
\centering
\includegraphics[width=1\textwidth,height=0.8\textwidth]{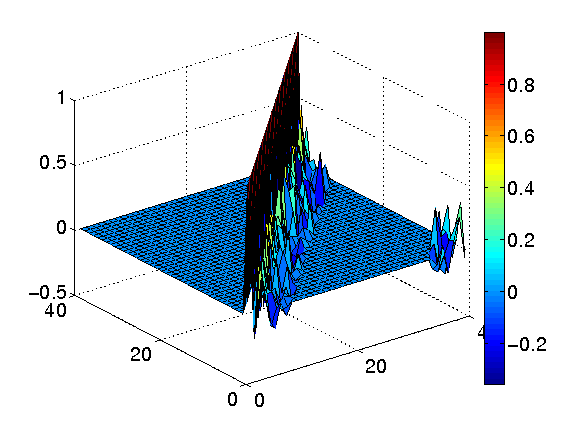}
\caption{$\T$, $\B^{-1} = \T^T \cdot \D^{-1} \cdot \T$}
\end{subfigure}

\begin{subfigure}{0.5\textwidth}
\centering
\includegraphics[width=1\textwidth,height=0.8\textwidth]{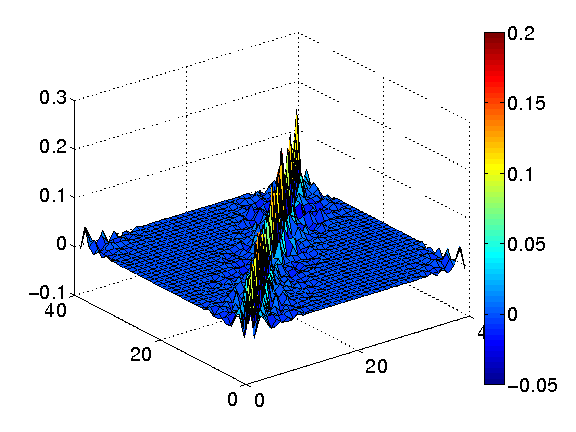}
\caption{Localized ensemble estimate $\widehat{\P^\textnormal{b}}^{-1}$}
\end{subfigure}%
\begin{subfigure}{0.5\textwidth}
\centering
\includegraphics[width=1\textwidth,height=0.8\textwidth]{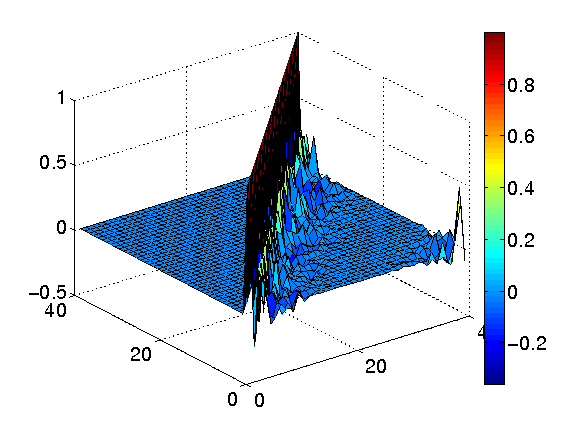}
\caption{$\T_{\bf L}$, $\widehat{\P^\textnormal{b}}^{-1} = \T_{\bf L}^T \cdot \D_{\bf L} \cdot \T_{\bf L}$}
\end{subfigure}%

\begin{subfigure}{0.5\textwidth}
\centering
\includegraphics[width=1\textwidth,height=0.8\textwidth]{figures/theo_PBLI.png}
\caption{Cholesky estimate $\BE^{-1}$}
\end{subfigure}%
\begin{subfigure}{0.5\textwidth}
\centering
\includegraphics[width=1\textwidth,height=0.8\textwidth]{figures/theo_PBLI_T.png}
\caption{$\TE$,  $\BE^{-1} = \TE^T \cdot \DE^{-1} \cdot \TE$}
\end{subfigure}
\caption{Decay of correlations in the Cholesky factors for different approximations of $\B^{-1}$. }
\label{fig:theo-Cholesky-factors}
\end{figure}
We consider a two-dimensional square domain with $s \times s$ grid points. Our proof below can be extended immediately to non-square domains, as well as to three-dimensional domains. In our domain each space point is described by two indices $(i,\,j)$, a zonal component $i$ and a meridional component $j$, for $1 \le i,j \le s$. A particular case for $s = 4$ is shown in Figure \ref{subfig:grid-distribution}. We make use of row-major order in order to map model grid components to the one dimensional ``index space'':
\begin{eqnarray*}
k = f(i,j) = (j-1) \cdot s+i,\,\quad \text{for $1 \le k \le \Nstate$}.\,
\end{eqnarray*}
where here, $\Nstate = s^2$. For a particular grid component $(i,\,j)$, the resulting $k = f(i,j)$ denotes the row index in $\BE^{-1}$. The results of labeling each model component in this manner can be seen in Figure \ref{subfig:row-major-order}.
\begin{figure}[H]
\centering
\begin{subfigure}{0.5\textwidth}
\centering
\fbox{\includegraphics[width=0.8\textwidth]{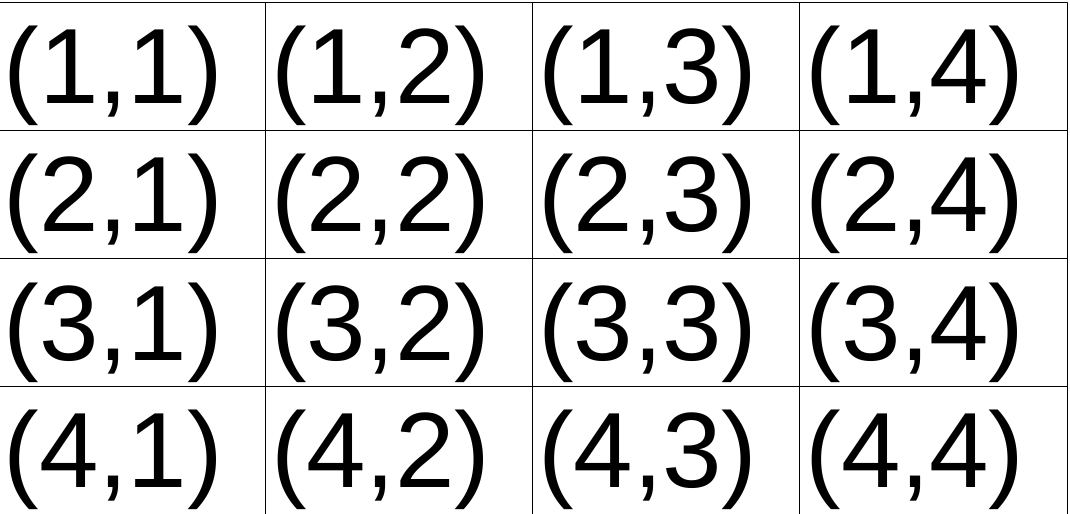}}
\caption{Grid components $(i,j)$}
\label{subfig:grid-distribution}
\end{subfigure}%
\begin{subfigure}{0.5\textwidth}
\centering
\fbox{\includegraphics[width=0.45\textwidth]{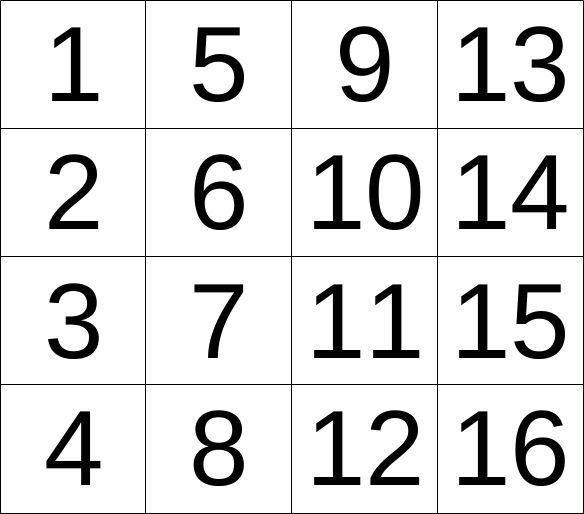}}
\caption{Index space $f(i,j)$}
\label{fig:corresponding-index-B}
\end{subfigure}
\caption{Grid distribution of model components and corresponding index terms in $\BE^{-1}$. }
\label{fig:ordering-example}
\end{figure}

To start our proof, the inverse of the (exact) background error covariance matrix $\B^{-1}$ and of the its estimator $\BE^{-1}$ can be written as 
\begin{subequations}
\begin{eqnarray}
\label{eq:proof-estimator-written}
\displaystyle
\BE^{-1} = \lb \I- \CE \rb^T \cdot \DE^{-1} \cdot \lb \I - \CE\rb \in \Re^{\Nstate \times \Nstate}
\end{eqnarray}
and
\begin{eqnarray}
\label{eq:proof-estimator-written}
\displaystyle
\B^{-1} = \lb \I - \CT\rb^T \cdot \D^{-1} \cdot \lb \I - \CT\rb \in \Re^{\Nstate \times \Nstate} ,\,
\end{eqnarray}
\end{subequations}
respectively, where $\CE = \I - \TE \in \Re^{\Nstate \times \Nstate}$ and $\CT = \I - \T \in \Re^{\Nstate \times \Nstate}$. Moreover,  $\D$ and $\DE$ are diagonal matrices:
\begin{eqnarray*}
\D &=& \diag \lle {d}_{1}^2,\, {d}_{2}^2,\, \ldots,\, {d}_{\Nstate}^2\rle  \\
\DE &=& \diag \lle \widehat{d}_{1}^2,\, \widehat{d}_{2}^2,\, \ldots,\, \widehat{d}_{\Nstate}^2\rle 
\end{eqnarray*}
where $\lle \D \rle_{i,i} = d_i^2$ and  $\lle \DE \rle_{i,i} = \widehat{d}_i^2$, for $1 \le i \le \Nstate$. In what follows we denote by $\CTc^{\{j\}} \in \Re^{\Nstate \times 1}$ and $\CTt^{\{j\}} \in \Re^{\Nstate \times 1}$ the $j$-th columns of matrices $\CE$ and $\CT$, respectively, for $1 \le j \le \Nstate$.

\begin{subequations}
\label{eq:Proof-theorem-1}

\begin{definition}[Class of matrices under consideration.] 
\label{theo-def}
We consider the class of covariance matrices matrices with correlations decreasing quickly:
\begin{eqnarray}
\label{eq:Proof-class-of-matrices}
\displaystyle
\classU^{-1} \lp \varepsilon_0, C, \alpha \rp &=&  \Bigg \{ \B: 0<\varepsilon_0 \le \lambda_{min} \lp \B \rp \le \lambda_{max} \lp \B \rp  \le \varepsilon_0^{-1},\,  \\ \nonumber
& & \underset{k}{\max} \sum_{\ell=1}^{\Nstate} \left | \gamma_{k,\ell} \cdot \lle \T \rle_{k,\ell} \right |  \le C \cdot {\ra}^{-\alpha},\, \text{ for $\ra \le s-1$}  \Bigg \}
\end{eqnarray}
where $\B^{-1} = \T^T \, \D^{-1} \, \T$, $\alpha$ is the decay rate (related to the dynamics of the numerical model),
\begin{eqnarray*} 
\gamma_{k_{(i,j)},\ell_{(p,q)}} &=& \begin{cases}
0 &  j-\ra \le q \le j-1 \text{ and } i-\ra \le p \le i+\ra \\
0 & q = j \text{ and } i-\ra \le p \le i \\
1 & \text{otherwise}
\end{cases} \,,
\end{eqnarray*}
and the grid components $(i,j)$ and $(p,q)$, for $1 \le i,j,p,q \le s$ are related to the $(k_{(i,j)},\ell_{(p,q)})$ matrix entry by $k_{(i,j)} = f(i,j)$ and $\ell = f(p,q)$.
\end{definition}

\begin{comment}
The factors $\gamma_{k,\ell}$ for the grid component $(i,j)$ in Definition \eqref{theo-def} are zero  inside the scope of $\ra$. 
\end{comment}

\begin{theorem}[Error in the covariance inverse estimation]
\label{theo-main}
Uniformly for $\B \in \classU^{-1} \lp \varepsilon_0, C, \alpha \rp$, if $\ra \approx  \lb \Nens^{-1} \cdot \log \Nstate \rb^{-1/2(\alpha+1)}$ and $\Nens^{-1} \cdot \log \Nstate = o(1)$,
\begin{eqnarray}
\label{eq:Proof-theorem}
\displaystyle
\ln \BEST^{-1} - \B^{-1}\rn_{\infty} = \mathcal{O} \lp \lb \frac{\log(\Nstate)}{\Nens}\rb^{\alpha(\alpha+1)/2}\rp 
\end{eqnarray}
where $\ln \cdot \rn_{\infty}$ denotes the infinity norm (matrix or vector)
\end{theorem}
\end{subequations}

\begin{comment}
The factors $\gamma_{k,\ell}$ in Theorem \eqref{theo-main} are zero for the predecessors of the grid component $(i,j)$ inside the scope of $\ra$. 
\end{comment}

In order to prove Theorem \ref{theo-main}, we need the following result. 
\begin{lemma}
\label{lemma:differences-emp-true}
Under the conditions of Theorem \ref{theo-main}, uniformly on $\classU^{-1}$
\begin{subequations}
\label{eq:proof-differences-empirical}
\begin{eqnarray}
\label{eq:proof-max-T}
\displaystyle
&&\max \lle \ln \CTc^{\{j\}}-\CTt^{\{j\}}  \rn_{\infty} : 1\le j \le \Nstate \rle = \BO{\Nens^{-1/2} \log^{1/2} \Nstate},\, \\ 
\label{eq:proof-max-D}
&&\max \lle \left| \widehat{d}_{j}^2 - d_{j}^2\right |: 1\le j \le \Nstate \rle = \BO{ \lb \Nens^{-1} \log \Nstate \rb^{\alpha/(2(\alpha+1))}} ,\,
\end{eqnarray}
and
\begin{eqnarray}
\label{eq:proof-norm1}
\displaystyle
\ln \CT \rn_{\infty} = \ln \D^{-1} \rn_{\infty} = \BO{1}.
\end{eqnarray}
\end{subequations}
\end{lemma}
The proof of Lemma \ref{lemma:differences-emp-true} is based on the following results of Bickel and Levina in \cite{bickel2008}.
\begin{lemma}
\label{lemma:lemma-Bickel}[\cite[Lemma A.2]{bickel2008}]
Let $\errbac^{[k]}  \sim \Nor \lp \zero,\, \B\rp$ and $\lambda_{\max} \lp \B \rp \le \varepsilon_0^{-1} < \infty$, for $1 \le k \le \Nens$. Then, if $\lle \B \rle_{i,j}$ denotes the $(i,\,j)$-th component of $\B$, for $1 \le i \le j \le \Nstate$, 
\begin{eqnarray}
\label{eq:lemma-Bickel}
\displaystyle
&& \textnormal{Prob} \lb \sum_{k=1}^{\Nens} \lb \lle \errbac^{[k]}  \rle_i \cdot \lle \errbac^{[k]} \rle_j - \lle \B\rle_{i,j} \rb \ge \Nens \cdot \nu \rb \\
\nonumber
&& \qquad \le C_1 \cdot \exp \lp -C_2 \cdot \Nens \cdot \nu^2 \rp,\,
\end{eqnarray}
for $|\nu| \le \delta$, where $\lle \errbac^{[k]} \rle_i$ is the $i$-th component of the sample $\errbac^{[k]}$, for $1 \le k \le \Nens$, and $1 \le i \le \Nstate$. Likewise, $C_1$, $C_2$ and $\delta$ depend on $\varepsilon_0$ only.
\end{lemma}

\begin{proof}[Proof of Lemma \ref{lemma:differences-emp-true}] 

In what follows we denote by $\var$ and $\evar$ denote the true and the empirical covariances, respectively. In the context of EnKF we have that $\var \lp \U^\textnormal{b} \rp = \B$.

Recall that
\begin{eqnarray*}
\evar \lp \U^\textnormal{b} \rp = \P^\textnormal{b} = \frac{1}{\Nens-1} \cdot \U^\textnormal{b} \cdot {\U^\textnormal{b}}^T = \frac{1}{\Nens-1} \cdot \sum_{k=1}^{\Nens} \ub^{\textnormal{b}[k]} \cdot {\ub^{\textnormal{b}[k]}}^T ,\,
\end{eqnarray*}
and therefore
\begin{eqnarray*}
\lle \evar \lp \U^\textnormal{b} \rp \rle_{i,j} = \frac{1}{\Nens-1} \cdot \sum_{k=1}^{\Nens} \lle \ub^{\textnormal{b}[k]} \rle_{i} \cdot \lle \ub^{\textnormal{b}[k]} \rle_{j}.
\end{eqnarray*}
For $\nu>0$, $\lle \errbac^{[k]}  \rle_i \cdot \lle \errbac^{[k]} \rle_j - \lle \B\rle_{i,j}  \ge \Nens \cdot \nu$ implies $\lle \errbac^{[k]}  \rle_i \cdot \lle \errbac^{[k]} \rle_j - \lle \B\rle_{i,j}  \ge (\Nens-1 )\cdot \nu$, and therefore by Lemma \ref{lemma:lemma-Bickel} we have:
\begin{subequations}
\begin{eqnarray}
\label{eq:proof-bound-U}
\ln \var \lp \U^\textnormal{b} \rp - \evar \lp  \U^\textnormal{b}  \rp \rn_{\infty} = \BO {\Nens^{-1/2} \cdot \log^{1/2} \Nstate },\,
\end{eqnarray}
since the entries of $\var \lp \U^\textnormal{b} \rp - \evar \lp  \U^\textnormal{b}  \rp$ can be bounded by:
\begin{eqnarray*}
\left | \lle \var \lp \U^\textnormal{b} \rp - \evar \lp  \U^\textnormal{b}  \rp \rle_{i,j} \right | & \le & \Nens^{-1} \cdot \sum_{k=1}^\Nens \left |  \lle \ub^{\textnormal{b}[k]} \rle_i \cdot \lle \ub^{\textnormal{b}[k]} \rle_j - \lle \B \rle_{i,j} \right |.
%&+& \Nens^{-2} \cdot \left | \sum_{k=1}^{\Nens} \lle \ub^{\textnormal{b}[k]} \rle_i \right | \cdot  \left |  \sum_{k=1}^{\Nens} \lle \ub^{\textnormal{b}[k]} \rle_j \right |  \,.
\end{eqnarray*}
Lemma \ref{lemma:lemma-Bickel} ensures that:
\begin{eqnarray*}
&& \textnormal{Prob} \lb \underset{i,j}{\max}  \left | \Nens^{-1} \cdot \sum_{k=1}^\Nens  \lle \ub^{\textnormal{b}[k]} \rle_i \cdot \lle \ub^{\textnormal{b}[k]} \rle_j - \lle \B \rle_{i,j}  \right | \ge \nu \rb \\
&& \qquad \le C_1 \cdot \Nstate^2 \cdot \exp \lp  -C_2 \cdot \Nens \cdot \nu^2 \rp,\,
\end{eqnarray*} 
for $|\nu| \le \delta$. Let $\nu = \lp \frac{\log \Nstate^2}{ \Nens \cdot C_2} \rp^{1/2}\cdot M$, for $M$ arbitrary. 

Since $\Z_{[i]}$ stores the columns of $\U^\textnormal{b}$ corresponding to the predecessors of model component $i$, an immediate consequence of \eqref{eq:proof-bound-U} is
\begin{eqnarray}
\label{eq:proof-bound-Z}
\underset{i}{\max} \ln \var \lp \Z_{[i]} \rp -\evar \lp \Z_{[i]} \rp \rn_{\infty} = \BO{ \Nens^{-1/2} \cdot \log^{1/2} \Nstate} \,.
\end{eqnarray}
\end{subequations}
%\end{eqnarray*}
%
Also,
\begin{eqnarray*}
\ln \B^{-1} \rn_{\infty} = \ln \var \lp \U^\textnormal{b} \rp^{-1} \rn_{\infty} \le \varepsilon_{0}^{-1} \,.
\end{eqnarray*}
According to equation \eqref{eq:EnKF-MC-solution-of-optimization-problem}, 
\begin{eqnarray*}
\lle \CTt^{[i]} \rle_{j} &=& \lle \var \lp \Z_{[i]} \rp^{-1} \cdot \Z_{[i]} \cdot \x_{[i]} \rle_j \,,\\
\lle \CTc^{[i]} \rle_{j} &=& \lle \evar \lp \Z_{[i]} \rp^{-1} \cdot \Z_{[i]} \cdot \x_{[i]} \rle_j  \,,
\end{eqnarray*}
therefore:
\begin{eqnarray} \nonumber
&&\underset{k}{\max} \lab \lle \CTt^{[i]} \rle_{k}-\lle \CTc^{[i]} \rle_{k} \rab \\
&=& \underset{k}{\max} \lab \lle \var \lp \Z_{[i]} \rp^{-1} \cdot \Z_{[i]} \cdot \x_{[i]} \rle_k -  \lle \evar \lp \Z_{[i]} \rp^{-1} \cdot \Z_{[i]} \cdot \x_{[i]} \rle_k  \rab \\ \nonumber
&=& \underset{k}{\max} \lab \lle \lb \var \lp \Z_{[i]} \rp^{-1}  - \evar \lp \Z_{[i]} \rp^{-1} \rb \cdot \Z_{[i]} \cdot \x_{[i]} \rle_k  \rab \\ \label{eq:proof-parta}
&=& \mathcal{O} \lp \Nens^{-1/2} \cdot \log^{1/2} \Nstate \rp 
\end{eqnarray}
from which \eqref{eq:proof-max-T} follows. Note that:
\begin{eqnarray*}
&& \x_{[i]} = \sum_{j=1}^{\Nstate} \tilde{\gamma}_{i,j} \cdot \lle \CTc^{[i]} \rle_j \cdot \x_{[j]} + \widehat{\err}^{[i]} \\
&\Leftrightarrow & \evar \lp \x_{[i]} \rp =  \evar \lp \sum_{j=1}^{\Nstate} \tilde{\gamma}_{i,j} \cdot \lle \CTc^{[i]} \rle_j \cdot \x_{[j]} + \widehat{\err}^{[i]} \rp \\
&\Leftrightarrow & \evar \lp \x_{[i]} \rp =  \evar \lp \sum_{j=1}^{\Nstate} \tilde{\gamma}_{i,j} \cdot \lle \CTc^{[i]} \rle_j \cdot \x_{[j]} \rp + \evar \lp \widehat{\err}^{[i]} \rp \\
&\Leftrightarrow & \widehat{d}^2_i = \evar \lp \x_{[i]}\rp -  \evar \lp \sum_{j=1}^{\Nstate} \tilde{\gamma}_{i,j} \cdot \lle \CTc^{[i]} \rle_j \cdot \x_{[j]} \rp ,\,
\end{eqnarray*}
and similarly 
\begin{eqnarray*}
d^2_i = \var \lp \x_{[i]}\rp -  \var \lp \sum_{j=1}^{\Nstate} \tilde{\gamma}_{i,j} \cdot \lle \CTt^{[i]} \rle_j \cdot \x_{[j]} \rp \,.
\end{eqnarray*}

The claim \eqref{eq:proof-max-D} and the first part of \eqref{eq:proof-norm1} follow from \eqref{eq:proof-bound-U}, \eqref{eq:proof-bound-Z} and \eqref{eq:proof-parta}. Since
%\begin{eqnarray*}
%\de^2_{i} &=& \evar \lp \x_{[i]} \rp - \evar \lp \sum_{j=1}^n \tilde{\gamma}_{i,j} \cdot \CTc^{[i]}_j \cdot \x_{[j]} \rp \,, \\
%{d}^2_{i} &=& \var \lp \x_{[i]} \rp - \var \lp \sum_{j=1}^n \tilde{\gamma}_{i,j} \cdot {\CTt}^{[i]}_j \cdot \x_{[j]} \rp \,, \\
%\end{eqnarray*}
%
%and the covariance operator is linear,
\begin{eqnarray} \nonumber
\lab \de_{i}^2 - d_{i}^2 \rab &\le & \lab \var \lp \x_{[i]}\rp - \evar \lp \x_{[i]}\rp \rab \\ 
&+& \lab \evar \lp \sum_{j=1}^{\Nstate} \tilde{\gamma}_{i,j} \cdot \lb \lle {\CTc}^{[i]} \rle_j - \lle {\CTt}^{[i]} \rle_j \rb \cdot \x_{[j]} \rp \rab \\ \nonumber
&+& \lab \evar \lp \sum_{j=1}^{\Nstate} \tilde{\gamma}_{i,j} \cdot \lle \CTc^{[i]} \rle_j \cdot \x_{[j]} \rp - \var \lp \sum_{j=1}^{\Nstate} \tilde{\gamma}_{i,j} \cdot \lle \CTc^{[i]}\rle_j \cdot \x_{[j]} \rp \rab
\end{eqnarray}
where $\tilde{\gamma}_{i,j} =  1-{\gamma}_{i,j}$. By Lemma \ref{lemma:lemma-Bickel} the maximum over $i$ of the first term is:
\begin{eqnarray*}
\displaystyle 
\underset{i}{\max} \lab \var \lp \x_{[i]}\rp - \evar \lp \x_{[i]}\rp \rab = \BO{\Nens^{-1/2} \cdot \log^{1/2} \Nstate }.\,
\end{eqnarray*}
The second term can be bounded as follows:
\begin{eqnarray*} 
%{\tiny \lab \evar \lp \sum_{j=1}^{\Nstate} \tilde{\gamma}_{i,j} \cdot \lb \lle {\CTc}^{[i]} \rle_j - \lle %{\CTt}^{[i]} \rle_j \rb \cdot \x_{[j]} \rp \rab} &=& \\
&& {\tiny \lab \sum_{j=1}^{\Nstate} \tilde{\gamma}_{i,j}^2 \cdot \lb \lle {\CTc}^{[i]} \rle_j - \lle {\CTt}^{[i]} \rle_j \rb^2 \cdot \evar\lp \x_{[j]} \rp \rab}   \\
& \le &  \sum_{j=1}^{\Nstate} \tilde{\gamma}_{i,j}^2 \cdot \lb \lle {\CTc}^{[i]} \rle_j - \lle {\CTt}^{[i]} \rle_j \rb^2 \cdot \lab \evar\lp \x_{[j]} \rp \rab  \\
& \le &  \underset{k}{\max} \lb \lle {\CTc}^{[i]} \rle_k - \lle {\CTt}^{[i]} \rle_k \rb^2 \cdot \underset{i}{\max}\lab \evar\lp \x_{[i]} \rp \rab \cdot \sum_{j=1}^{\Nstate} \tilde{\gamma}_{i,j}^2   \\
&=&\BO{ \ra^2 \cdot \Nens^{-1} \cdot \log \Nstate } \\  \nonumber
&=& \BO{ \lb \Nens^{-1} \cdot \log \Nstate \rb^{\alpha/2 \cdot (\alpha+1)} }
%
%& &\sum_{i=1}^{\Nstate} \lb  \sum_{j=1}^{\Nstate}  \lb {\CTc}^{[i]}_j -\CTt^{[i]}_j \rb \cdot \lb {\CTc}^{[i]}_j -\CTt^{[i]}_j \rb \cdot \widehat{\bf cov} \lp \x_{[i]},\, \x_{[j]} \rp \rb \\ 
%& \le & \lb \sum_{t=1}^{\Nstate} \tilde{\gamma}_{i,j} \cdot \lab {\CTc}^{[i]}_t -{\CTt}^{[i]}_t \rab \cdot \evar \lp \x^{[t]} \rp \rb^2 \\ \nonumber
%& \le & \ra^2 \cdot \lb \underset{t}{\max} \lle \lp {\CTc^{[i]}_t} - {\CTt^{[i]}_t} \rp^2 \rle \rb \cdot \lb \underset{t}{\max}\, \lle \evar \lp \x^{[t]} \rp \rle \rb = O \lp \ra^2 \cdot \Nens^{-1} \cdot \log \Nstate^2 \rp \\  \nonumber
%&=& O \lp \lb \Nens^{-1} \cdot \Nstate \rb^{\alpha/2 \cdot (\alpha+1)} \rp
\end{eqnarray*}
by \eqref{eq:proof-max-T} and $\ln \B\rn \le \varepsilon_0^{-1}$. Recall that $\ra = \lb \Nens^{-1} \cdot \log \Nstate \rb^{1/2 \cdot(\alpha+1)}$ and even more, note that:
\begin{eqnarray*}
\sum_{j=1}^{\Nstate} \tilde{\gamma}^2_{i,j} = \frac{\lb \ra+1 \rb^2}{2} = \frac{\ra^2}{2}+\ra+\frac{1}{2} = \BO{\ra^2} \,.
\end{eqnarray*}
The third term can be bounded similarly. Thus \eqref{eq:proof-max-D} follows. Furthermore, 
\begin{eqnarray*}
\displaystyle
d^2_{i} = \var \lp \x_{[i]} - \sum_{j=1}^{\Nstate} \tilde{\gamma}_{i,j} \cdot \lle \CTc^{[i]}\rle_j \cdot \x_{[j]} \rp \ge \varepsilon_0 \cdot \lp 1 + \sum_{i=1}^{\Nstate} \lb \CTc^{[i]}_j \rb^2 \rp \ge \varepsilon_0 \,,
\end{eqnarray*}
and the lemma follows.
\end{proof}
We now are ready to prove Theorem \ref{theo-main}.
\begin{proof}[Proof of Theorem \ref{theo-main}]
We need only check that:
\begin{subequations}
\label{eq:Proof-to-check}
\begin{eqnarray}
\label{eq:Proof-difference-of-inverses}
\displaystyle
\ln \BEST^{-1} - \B^{-1} \rn_{\infty} = \BO{ \Nens^{-1/2} \cdot \log^{1/2} \lp \Nstate \rp}
\end{eqnarray}
and
\begin{eqnarray}
\label{eq:Proof-rate-radius}
\displaystyle
\ln \B^{-1} - \Phi_{\ra} \lp \B^{-1} \rp\rn_{\infty} = \BO{\ra^{-\alpha}}
\end{eqnarray}
where the entries of $\Phi_{\ra} \lp \B^{-1} \rp$ are given by:
\begin{eqnarray}
\lle \Phi_{\ra} \lp \B^{-1} \rp \rle_{k,\ell} = \delta_{k,\ell} \cdot \lle \B^{-1} \rle_{k,\ell} ,\, \text{ for $1 \le k,\ell \le \Nstate$}
\end{eqnarray}
where $k = f(i,j)$  and $\ell = f(q,p)$ for $1 \le i,j,p,q \le s$, and
\begin{eqnarray*} 
\delta_{k,\ell} &=& \begin{cases}
1 &  j-\ra \le q \le j+\ra \text{ and } i-\ra \le p \le i+\ra \\
0 & \text{otherwise}
\end{cases}
\end{eqnarray*}
\end{subequations}
We first prove \eqref{eq:Proof-difference-of-inverses}. By definition,
\begin{eqnarray}
\BEST^{-1} - \B^{-1} = \BESTT^T \cdot \BESTD^{-1} \cdot \BESTT - \T^T \cdot \D^{-1} \cdot \T.
\end{eqnarray}
Applying the standard inequality:
\begin{eqnarray*}
\ln \T^T \cdot \D^{-1} \cdot \T -\TE^T \cdot \DE^{-1} \cdot \TE^T\rn &\le & \ln \T^T- \TE^T \rn \cdot \ln \DE \rn \cdot \ln \TE \rn \\
&+& \ln \D- \DE \rn \cdot \ln \TE^T \rn \cdot \ln \TE \rn \\
&+& \ln \T- \TE \rn \cdot \ln \TE \rn \cdot \ln \DE \rn \\
&+& \ln \TE \rn  \cdot \ln \D-\DE \rn \cdot \ln \TE^T-\T^T \rn \\
&+& \ln \DE \rn  \cdot \ln \T-\TE \rn \cdot \ln \TE^T-\T^T \rn \\
&+& \ln \TE^T \rn  \cdot \ln \D-\DE \rn \cdot \ln \TE-\T \rn \\
&+& \ln \D-\DE \rn \cdot \ln \T-\TE \rn \cdot \ln \T^T-\TE^T \rn
\end{eqnarray*}
%\begin{eqnarray*}
%\ln \prod_{j=1}^3 \YE^{(j)}- \prod_{j=1}^3\P^{(j)} \rn_{\infty} & \le &  \sum_{j=1}^3 \ln \YE^{(j)} - \P^{(j)} \rn_{\infty} \cdot \prod_{i \neq j} \ln \P^{(i)} \rn_{\infty} \\
%&+& \sum_{j=1}^3 \ln \P^{(j)} \rn_{\infty} \cdot \prod_{i \neq j} \ln \YE^{(i)} - \P^{(i)}\rn_{\infty}  \\ 
%&+& \prod_{j=1}^3 \ln \YE^{(j)}-\P^{(j)} \rn_{\infty}
%\end{eqnarray*}
%
all previous terms can be bounded making use of Lemma \ref{lemma:differences-emp-true} and therefore, \eqref{eq:Proof-difference-of-inverses} follows. Likewise, for \eqref{eq:Proof-rate-radius}, we need to note that for any matrix $\MV$,
\begin{eqnarray*}
\ln \MV \cdot \MV^T  - \Phi_{\ra} \lp \MV \rp \cdot \Phi_{\ra} \lp \MV \rp^T \rn_{\infty} & \le & 2 \cdot \ln \MV \rn_{\infty} \cdot \ln \Phi_{\ra} \lp \MV \rp-\MV^{-1} \rn_{\infty} \\
&+&\ln \Phi_{\ra} \lp \MV \rp-\MV \rn_{\infty}^2 
\end{eqnarray*}
and by letting $\MV = \T^T \cdot \D^{-1/2}$, the theorem follows from Definition \ref{theo-def}.
\end{proof}

%%%%%%%%%%%%%%%%%%%%%%%%%%%%%%%
\section{Numerical Experiments}
\label{sec:experimental-settings}
%%%%%%%%%%%%%%%%%%%%%%%%%%%%%%%

In this section we study the performance of the proposed EnKF-MC implementation. The experiments are performed using the atmospheric general circulation model SPEEDY \cite{Speedy1,Speedy2}. SPEEDY is a hydrostatic, spectral coordinate, spectral transform model in the vorticity-divergence form, with semi-implicit treatment of gravity waves. The number of layers in the SPEEDY model is 8 and the T-63 model resolution ($192 \times 96$ grids)  is used for the horizontal space discretization of each layer. Four model variables are part of the assimilation process: the temperature ($K$), the zonal and the meridional wind components ($m/s$), and the specific humidity ($g/kg$). The total number of model components is $\Nstate = 589,824$. The number of ensemble members is $\Nens=94$ for all the scenarios. The model state space is approximately 6,274 times larger than the number of ensemble members ($\Nstate \gg \Nens$). 

Starting with the state of the system  $\x^\textnormal{ref}_{-3}$ at time $t_{-3}$, the model solution $\x^\textnormal{ref}_{-3}$ is propagated in time over one year:
\begin{eqnarray*}
\x^\textnormal{ref}_{-2} = \M_{t_{-3} \rightarrow t_{-2}} \lp \x^\textnormal{ref}_{-3}\rp.
\end{eqnarray*}
The reference solution $\x^\textnormal{ref}_{-2}$ is used to build a perturbed background solution:
\begin{eqnarray}
\label{eq:exp-perturbed-background}
\displaystyle
\widehat{\x}^\textnormal{b}_{-2} = \x^\textnormal{ref}_{-2} + \errobs^\textnormal{b}_{-2}, \quad  \errobs^\textnormal{b}_{-2} \sim \Nor \lp \zero_{\Nstate} ,\, \underset{i}{\textnormal{diag}} \left\{ (0.05\, \{\x^\textnormal{ref}_{-2}\}_i)^2 \right\} \rp.
\end{eqnarray}
The perturbed background solution is propagated over another year to obtain the background solution at time $t_{-1}$:
\begin{eqnarray}
\label{eq:exp-background-state-1}
\x^\textnormal{b}_{-1} = \M_{t_{-2} \rightarrow t_{-1}} \lp \widehat{\x}^\textnormal{b}_{-2}\rp.
\end{eqnarray}
This model propagation attenuates the random noise introduced in \eqref{eq:exp-perturbed-background} and makes the background state \eqref{eq:exp-background-state-1} consistent with the physics of the SPEEDY model. Then, the background state \eqref{eq:exp-background-state-1} is utilized in order to build an ensemble of perturbed background states:
\begin{eqnarray}
\label{eq:exp-perturbed-ensemble}
\displaystyle
\widehat{\x}^{\textnormal{b}[i]}_{-1}  = \x^\textnormal{b}_{-1} + \errobs^\textnormal{b}_{-1},\quad \errobs^\textnormal{b}_{-1} \sim \Nor \lp \zero_{\Nstate} ,\, \underset{i}{\textnormal{diag}} \left\{ (0.05\, \{\x^\textnormal{b}_{-1}\}_i)^2 \right\} \rp,
\quad 1 \le i \le \Nens,
\end{eqnarray}
from which, after three months of model propagation, the initial ensemble is obtained at time $t_0$:
\begin{eqnarray*}
\x^{\textnormal{b}[i]}_0 = \M_{t_{-1} \rightarrow t_0} \lp \widehat{\x}^{\textnormal{b}[i]}_{-1}\rp \,.
\end{eqnarray*}
Again, the model propagation of the perturbed ensemble ensures that the ensemble members are consistent with the physics of the numerical model. 

The experiments are performed over a period of 24 days, where observations are taken every 2 days ($\N=12$). At time $k$ synthetic observations are built as follows:
\begin{eqnarray*}
\y_k = \H_k \cdot \x^\textnormal{ref}_k + \errobs_k, \quad \errobs_k \sim \Nor \lp \zeros_{\Nobs},\, \R_k \rp,\,
\quad \R_k = \textnormal{diag}_{i}\left\{ (0.01\, \{\H_k \, \x^\textnormal{ref}_k\}_i )^2  \right\}.
\end{eqnarray*}
The observation operators $\H_k$ are fixed throughout the time interval. We perform experiments with several operators characterized by different  proportions $p$ of observed components from the model state $\x^\textnormal{ref}_k$ ($\Nobs \approx p \cdot \Nstate$). We consider four different values for $p$: 0.50, 0.12, 0.06 and 0.04 which represent 50\%, 12 \%, 6 \% and 4 \% of the total number of model components, respectively. Some of the observational networks used during the experiments are shown in Figure \ref{fig:exp-observational-grids} with their corresponding percentage of observed components from the model state.

The analyses of the EnKF-MC are compared against those obtained making use of the LETKF implementation proposed by Hunt et al in \cite{LETKFHunt,TELA:TELA076,application_letkf_1} . The analysis accuracy is measured by the root mean square error (RMSE)
\begin{eqnarray}
\label{eq:ER-RMSE-formula}
\displaystyle
\text{RMSE}  = \sqrt{\frac{1}{\N} \cdot \sum_{k=1}^\N \lb \x^\textnormal{ref}_k -\x^\textnormal{a}_k \rb^T \cdot \lb \x^\textnormal{ref}_k -\x^\textnormal{a}_k \rb }
\end{eqnarray}
where $\x^\textnormal{ref} \in \Re^{\Nstate \times 1}$ and $\x^\textnormal{a}_{k} \in \Re^{\Nstate \times 1}$ are the reference and the analysis solutions at time $k$, respectively, and $\N$ is the number of assimilation times.

The threshold used in \eqref{eq:EnKF-MC-truncated-SVD} during the computation of $\BE^{-1}$ is $\sigma_{r} = 0.10$. During the assimilation steps, the data error covariance matrices $\R_k$ are used (no representativeness errors are involved during the assimilations) and therefore. The different EnKF implementations are performed making use of FORTRAN and specialized libraries such as BLAS and LAPACK are used in order to perform the algebraic computations. 
\begin{figure}[H]
\centering
\begin{subfigure}{0.5\textwidth}
\centering
\includegraphics[width=0.9\textwidth,height=0.5\textwidth]{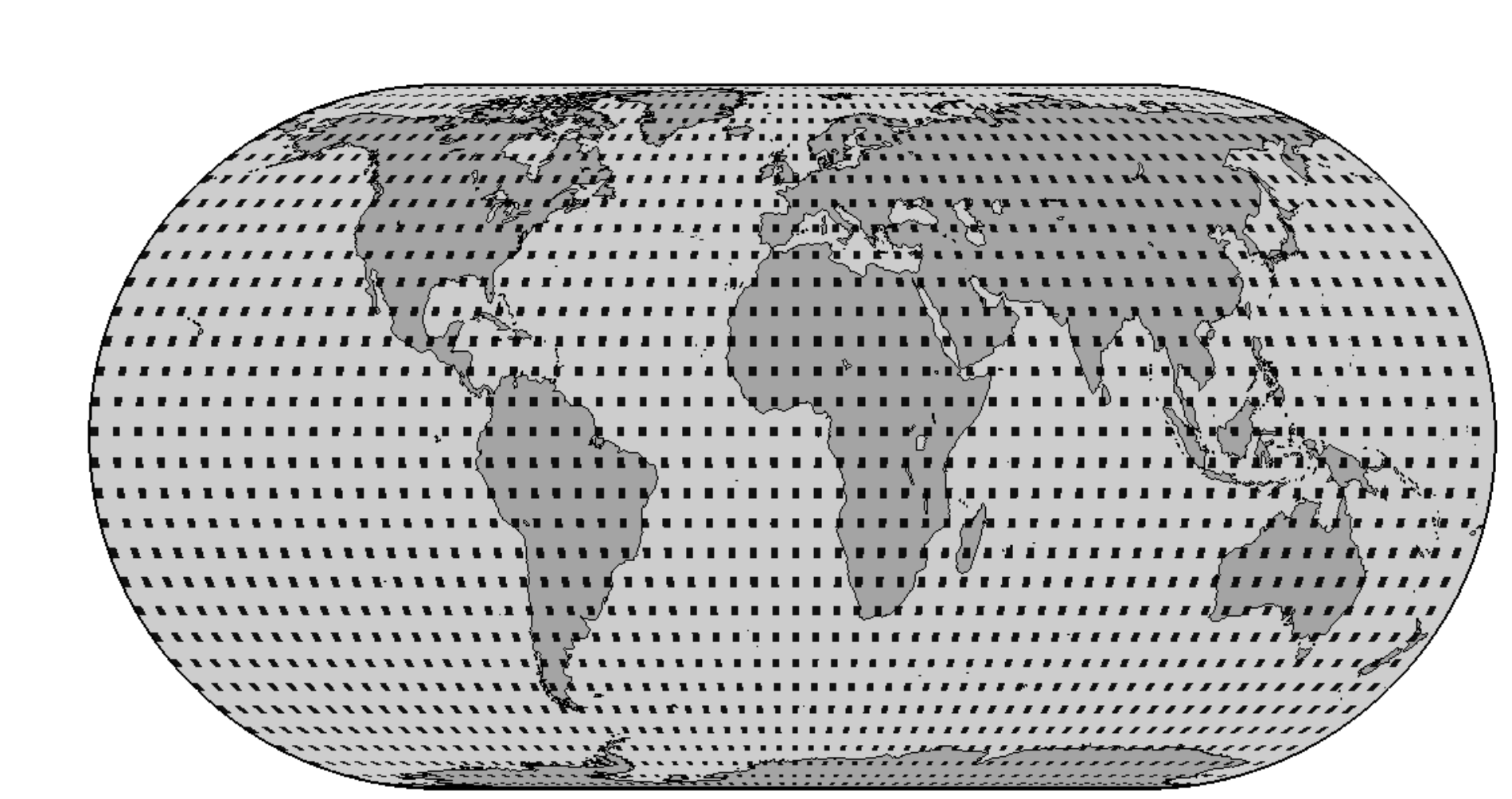}
\caption{$p=12\%$ }
\end{subfigure}%
\begin{subfigure}{0.5\textwidth}
\centering
\includegraphics[width=0.9\textwidth,height=0.5\textwidth]{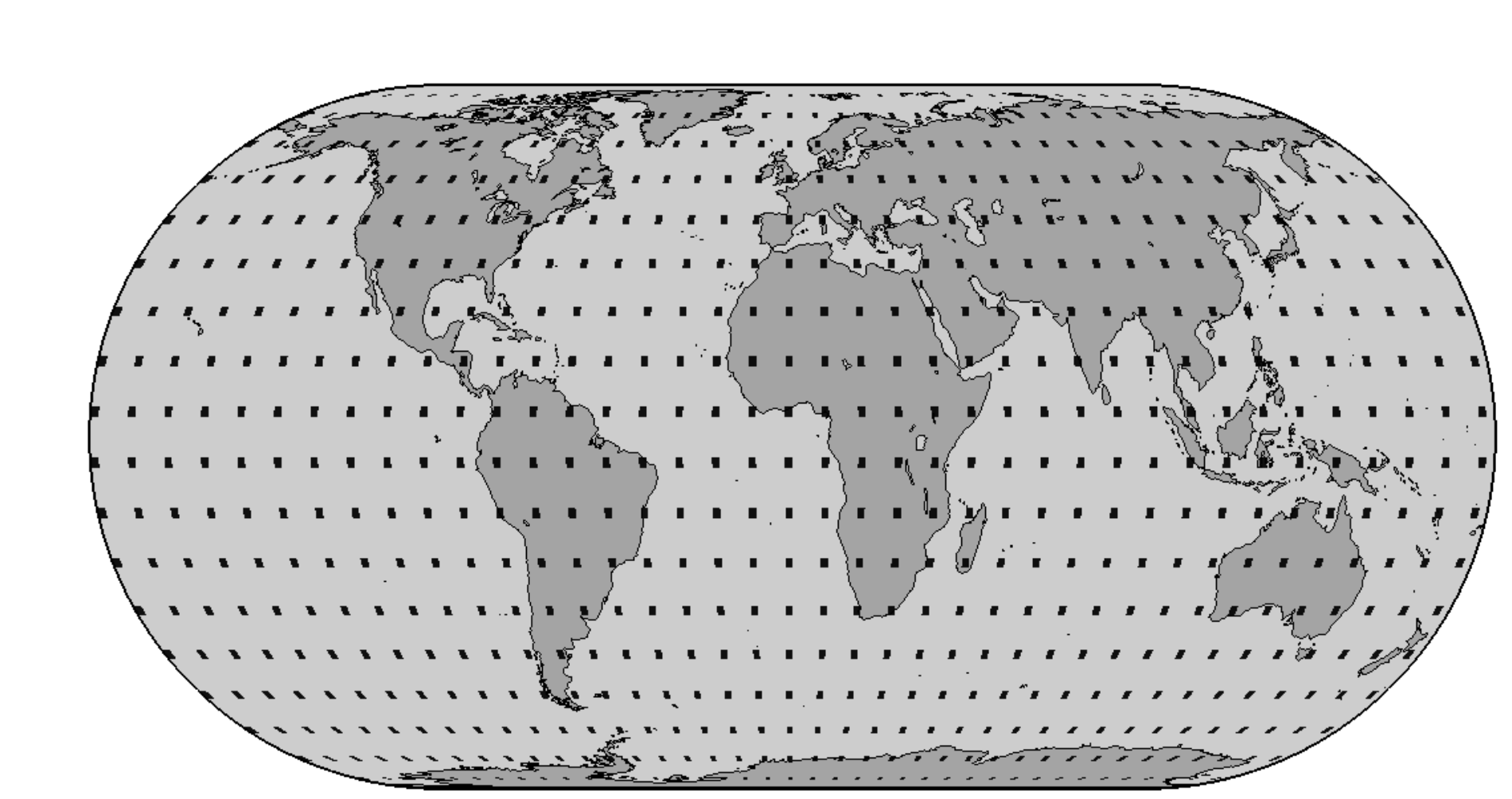}
\caption{$p=4\%$ }
\end{subfigure}%
\caption{Observational networks for different values of $p$. Dark dots denote the location of the observed components. The observed model variables are the zonal and the meridional wind components, the specific humidity, and the temperature.}
\label{fig:exp-observational-grids}
\end{figure}
%
%%%%%%%%%%%%%%%%%%%%%%%%%%%%%%%%%%
\subsection{Results with dense observation networks}
%%%%%%%%%%%%%%%%%%%%%%%%%%%%%%%%%%

We first consider dense observational networks in which 100\% and 50\% of the model components are observed. We vary the radius of influence $\ra$ from 1 to 5 grid points.

Figure \ref{fig:exp-LETKF-RMSE-different-radii-dense-network} shows the RMSE values for the LETKF and EnKF-MC analyses for different values of $\ra$ for the specific humidity when $50\%$ of model components are observed. When the radius of influence is increased the quality of the LETKF results degrades due to spurious correlations. This is expected since the local estimation of correlations in the context of LETKF is the sample covariance matrix. For instance, for a radius of influence of 1, the total number of local components for each local box is 36 which matches the dimension of the local background error distribution. Now, when we compare it against the ensemble size (96 ensemble members), sufficient degrees of freedom (95 degrees of freedom) are available in order to estimate the local background error distribution onto the ensemble space, and consequently all  directions of the local probability error distribution are accounted during the estimation and posterior assimilation. On the other hand, when the radius of influence is 5, the local box sizes have dimension 484 (model components) which is approximately 5 times larger than the ensemble size. Thus, when the analysis increments are computed onto the ensemble space, just part of the local background error distribution is accounted during the assimilation. Consequently, the larger the local box, the more local background error information cannot be represented in the ensemble space. 

Figure \ref{fig:exp-LETKF-RMSE-different-radii-dense-network} shows that EnKF-MC analyses improve with increasing radius of influence $\ra$. Since a dense observational network is considered during the assimilation, when the radius of influence is increased, a better estimation of the state of the system is obtained by the EnKF-MC. This can be seen clearly in Figure \ref{fig:exp-LETKF-RMSE-different-radii-dense}, where the RMSE values within the assimilation window are shown for the LETKF and the EnKF-MC solutions for the specific humidity variable and different values of $\ra$ and $p$. The quality of the EnKF-MC analysis for $\ra=5$ is better than that of the LETKF with $\ra=1$. Likewise, when a full observational network is considered ($p=100\%$), the proposed implementation outperforms the LETKF implementation. EnKF-MC  is able to exploit the large amount of information contained in dense observational networks by properly estimating the local background error correlations.  The RMSE values for all model variables and different values for $\ra$ and $p$ are summarized in Table \ref{tab:exp-RMSE-values-all-dense}.
\begin{figure}[H]
\centering
\begin{subfigure}{0.5\textwidth}
\centering
\includegraphics[width=0.9\textwidth,height=0.75\textwidth]{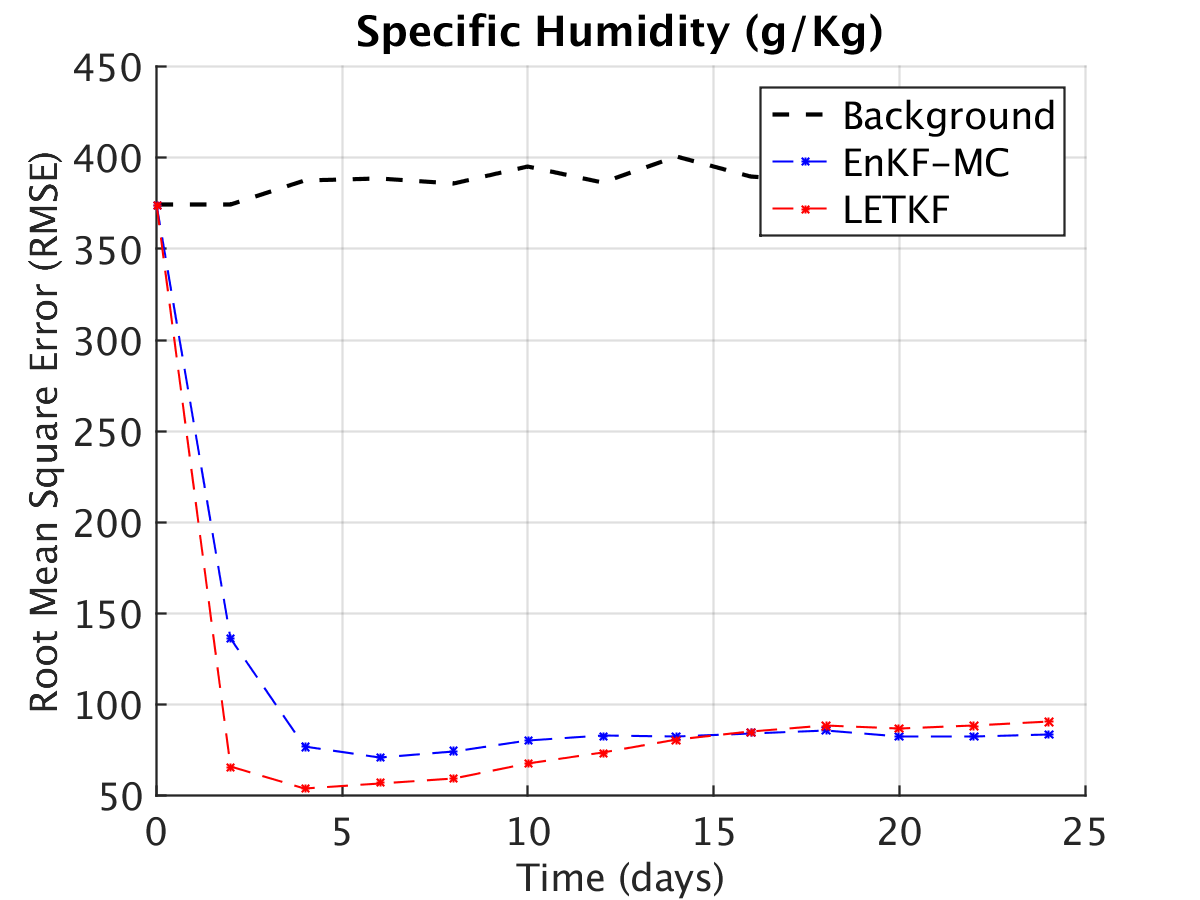}
\caption{$\ra  = 1$.}
\end{subfigure}%
\begin{subfigure}{0.5\textwidth}
\centering
\includegraphics[width=0.9\textwidth,height=0.75\textwidth]{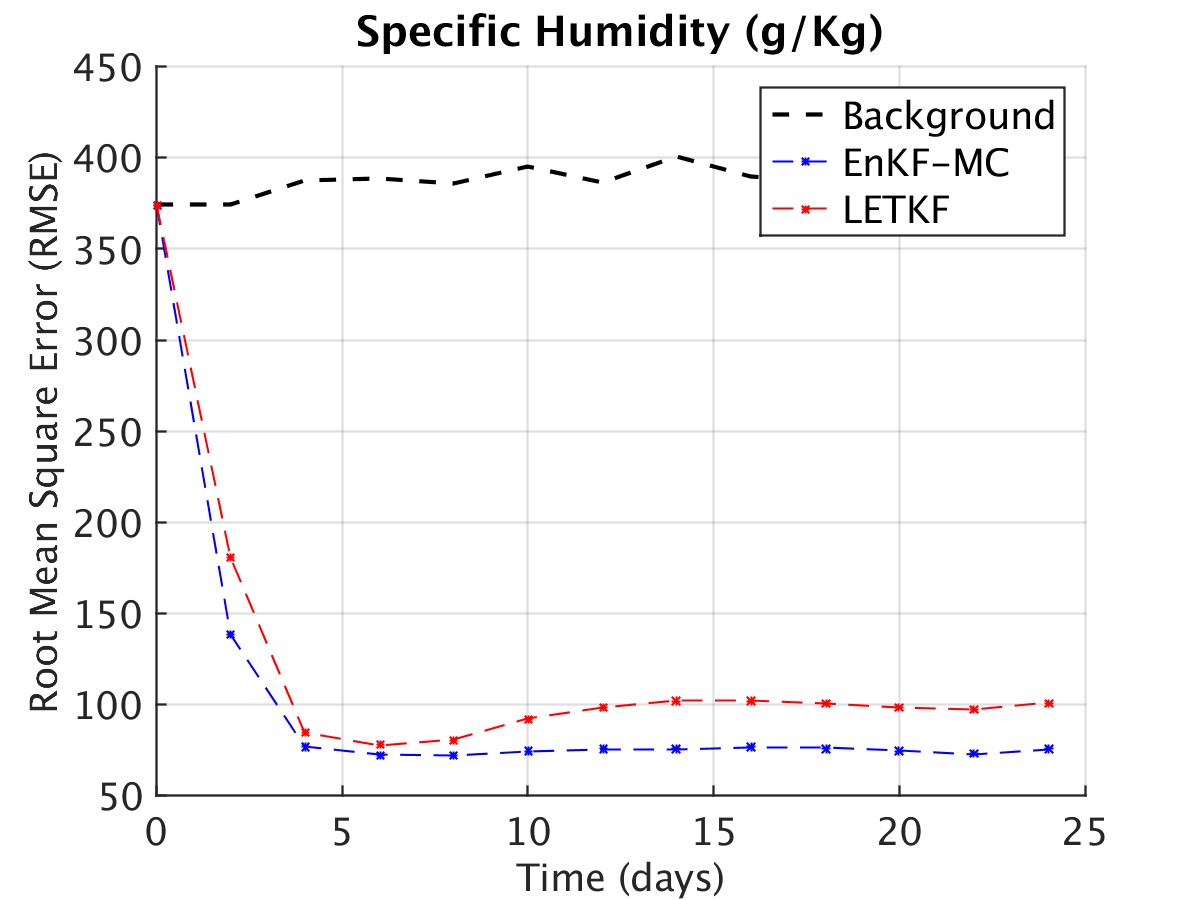}
\caption{$\ra  = 2$.}
\end{subfigure}
\begin{subfigure}{0.5\textwidth}
\centering
\includegraphics[width=0.9\textwidth,height=0.75\textwidth]{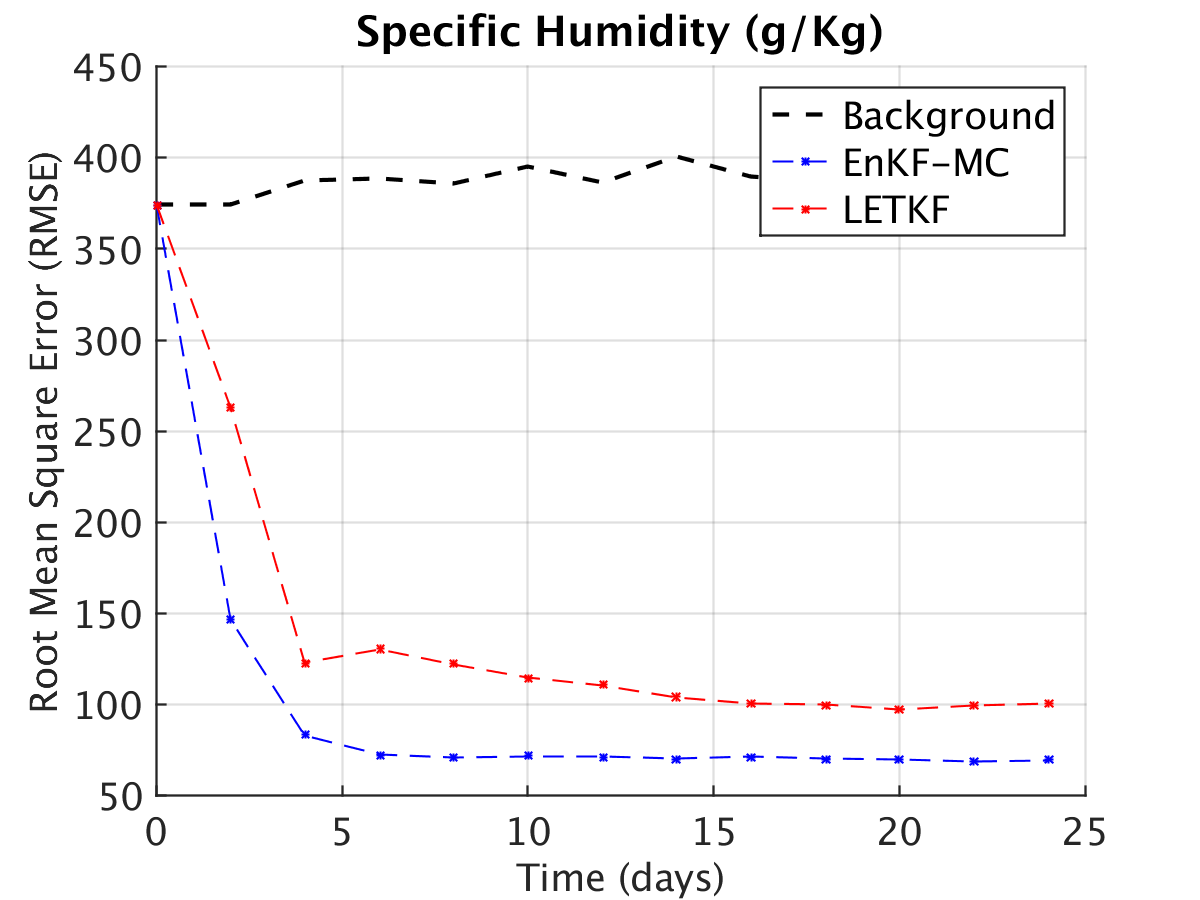}
\caption{$\ra  = 3$.}
\end{subfigure}%
\begin{subfigure}{0.5\textwidth}
\centering
\includegraphics[width=0.9\textwidth,height=0.75\textwidth]{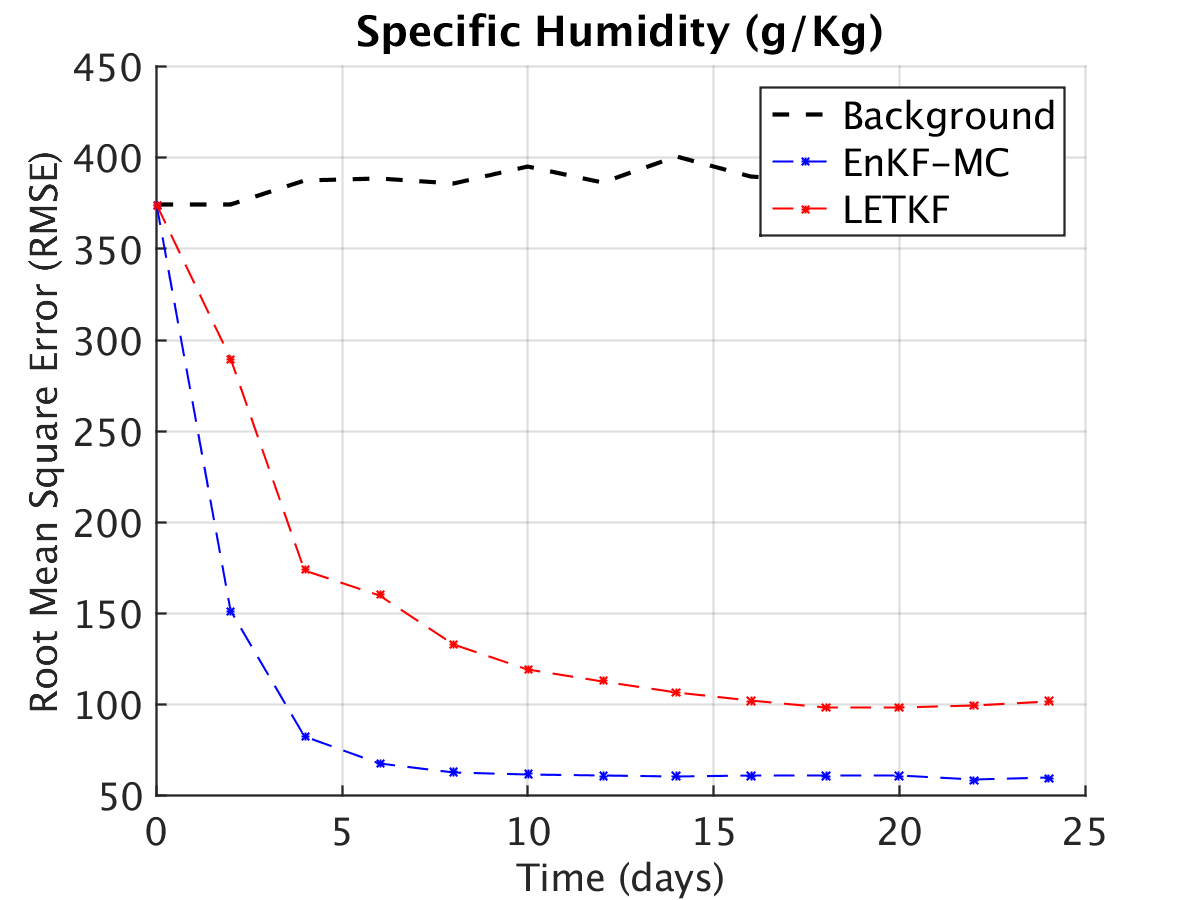}
\caption{$\ra  = 4$.}
\end{subfigure}
\begin{subfigure}{0.5\textwidth}
\centering
\includegraphics[width=0.9\textwidth,height=0.75\textwidth]{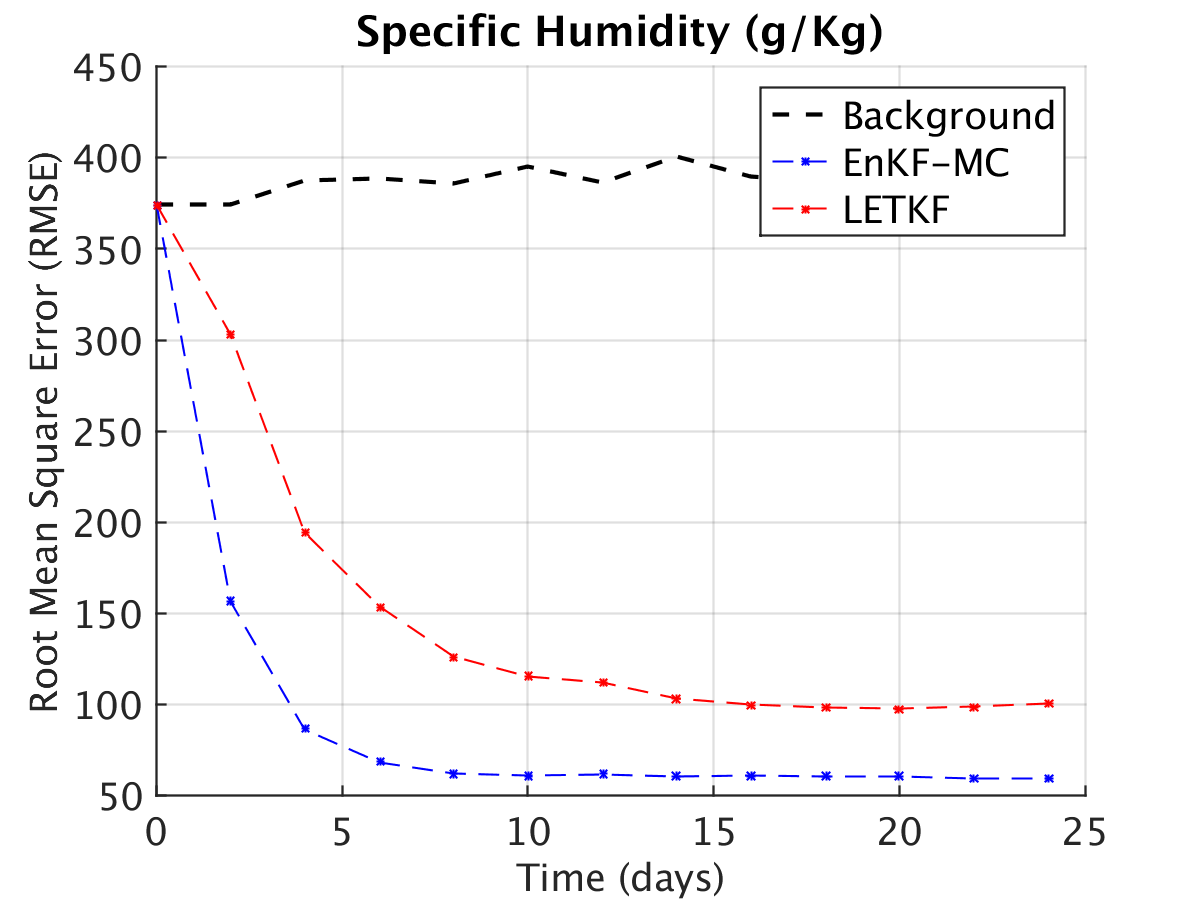}
\caption{$\ra  = 5$.}
\end{subfigure}
\caption{RMSE of specific humidity analyses with a dense observational network. When the radius of influence $\ra$ is increased the performance of LETKF degrades.}
\label{fig:exp-LETKF-RMSE-different-radii-dense-network}
\end{figure}

\begin{figure}[H]
\centering
\begin{subfigure}{0.5\textwidth}
\centering
\includegraphics[width=0.9\textwidth,height=0.75\textwidth]{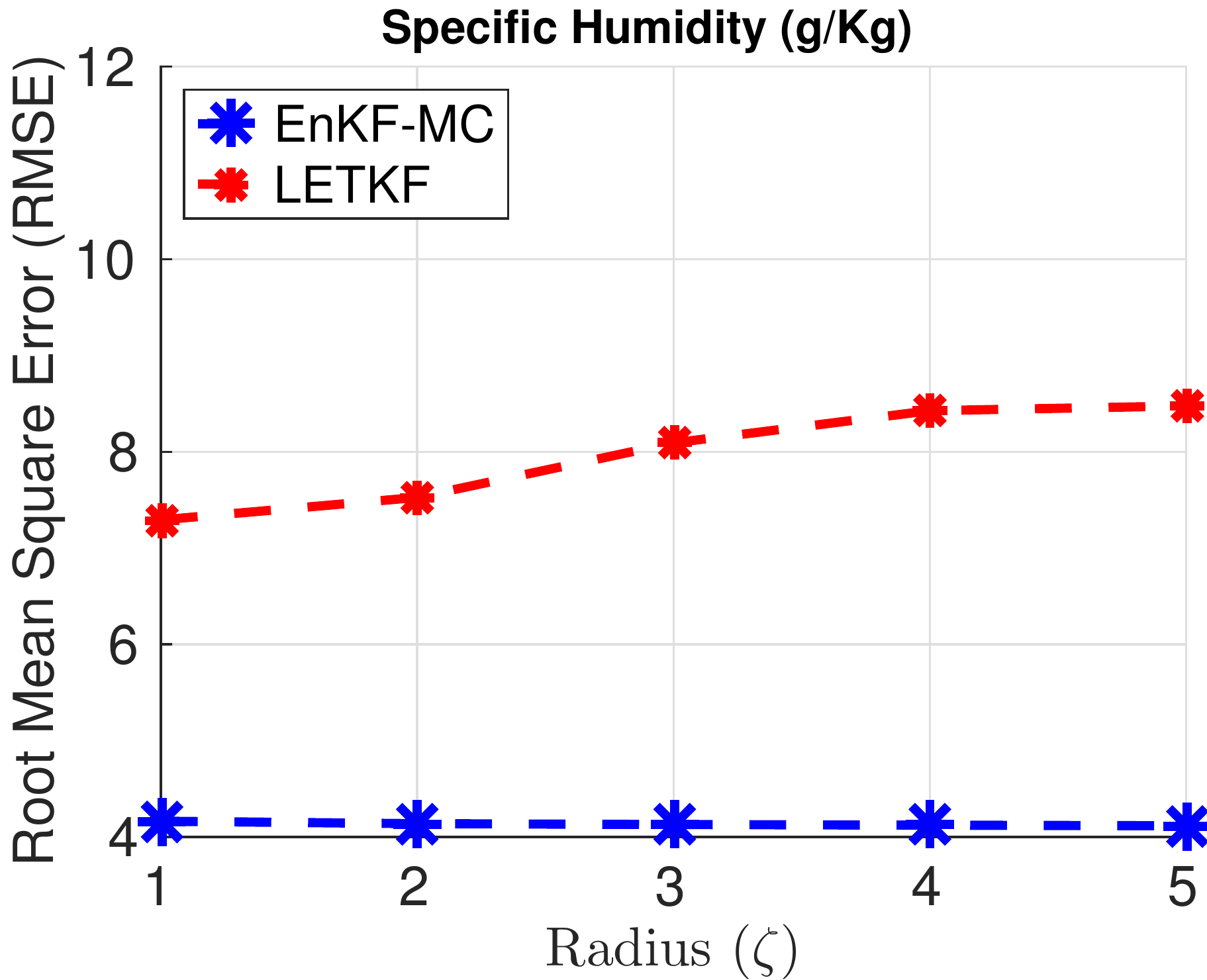}
\caption{$p  = 100\%$.}
\end{subfigure}%
\begin{subfigure}{0.5\textwidth}
\centering
\includegraphics[width=0.9\textwidth,height=0.75\textwidth]{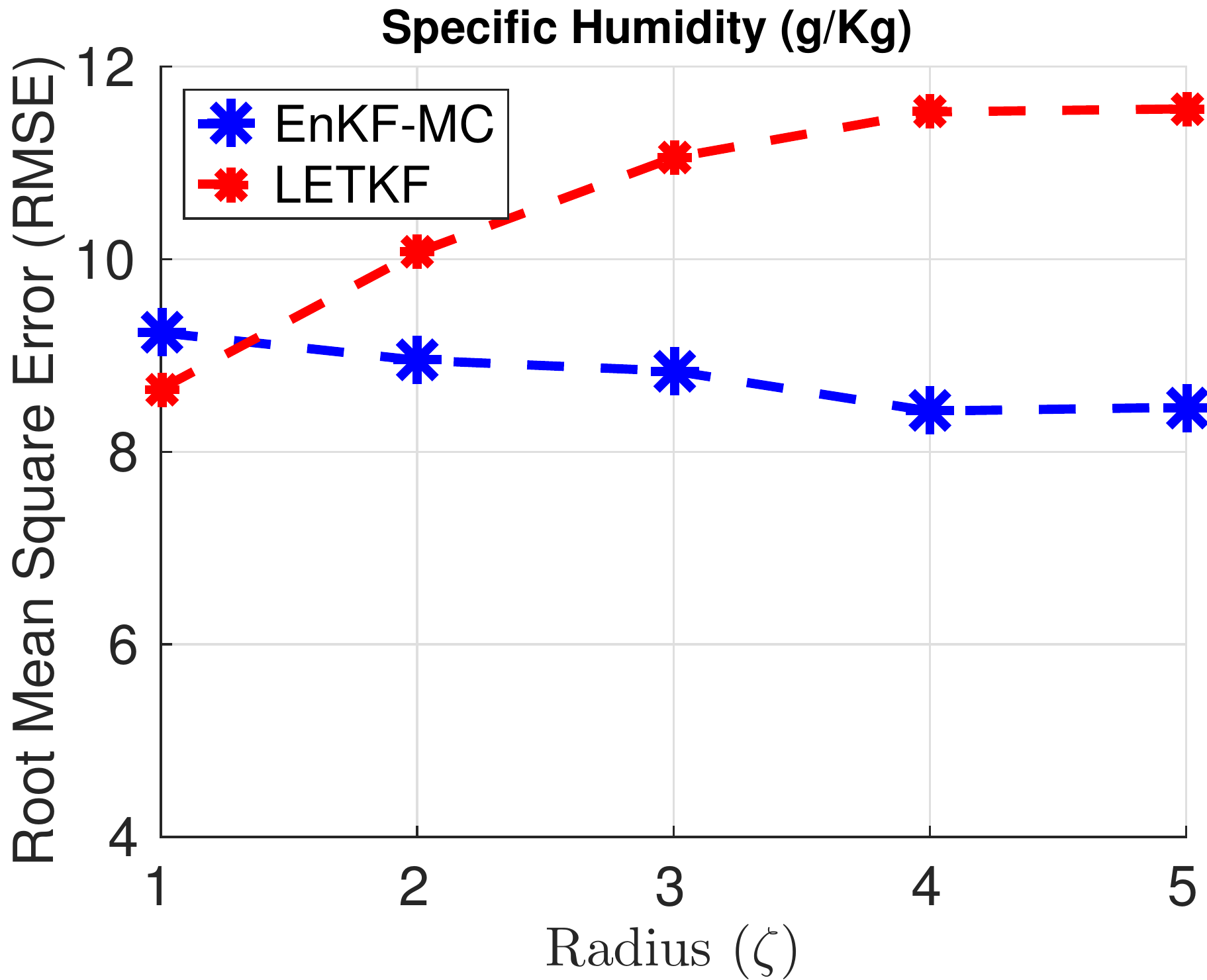}
\caption{$p  = 50\%$.}
\end{subfigure}%
\caption{Analysis RMSE for the specific humidity variable. The RMSE values of the assimilation window are shown for different values of $\ra$ and percentage of observed components $p$. When the local domain sizes are increased the accuracy of the LETKF analysis degrades, while the accuracy of EnKF-MC analysis improves.}
\label{fig:exp-LETKF-RMSE-different-radii-dense}
\end{figure}
\begin{table}[H]
\centering
{\footnotesize
\begin{tabular}{|c|c|c|c|c|} \hline
Variable (units) &$\ra$ & $p$ & EnKF-MC & LETKF \\ \hline
 \multirow{10}{*}{Zonal Wind Component ($u$), ($m/s$)}  &  \multirow{2}{*}{1}  & $100 \%$ & $ 6.012 \times 10^{1}$ & $ 6.394 \times 10^{1}$\\ \cline{3-5} &  & $50 \%$ & $ 4.264 \times 10^{2}$ & $ 9.825 \times 10^{2}$\\ \cline{2-5} &  \multirow{2}{*}{2}  & $100 \%$ & $ 6.078 \times 10^{1}$ & $ 6.820 \times 10^{1}$\\ \cline{3-5} &  & $50 \%$ & $ 2.255 \times 10^{2}$ & $ 1.330 \times 10^{3}$\\ \cline{2-5} &  \multirow{2}{*}{3}  & $100 \%$ & $ 6.080 \times 10^{1}$ & $ 7.969 \times 10^{1}$\\ \cline{3-5} &  & $50 \%$ & $ 2.341 \times 10^{2}$ & $ 1.124 \times 10^{3}$\\ \cline{2-5} &  \multirow{2}{*}{4}  & $100 \%$ & $ 6.088 \times 10^{1}$ & $ 9.687 \times 10^{1}$\\ \cline{3-5} &  & $50 \%$ & $ 2.418 \times 10^{2}$ & $ 1.072 \times 10^{3}$\\ \cline{2-5} &  \multirow{2}{*}{5}  & $100 \%$ & $ 6.092 \times 10^{1}$ & $ 1.190 \times 10^{2}$\\ \cline{3-5} &  & $50 \%$ & $ 2.673 \times 10^{2}$ & $ 1.017 \times 10^{3}$\\ \cline{1-5} \multirow{10}{*}{Meridional Wind Component ($v$) ($m/s$)}  &  \multirow{2}{*}{1}  & $100 \%$ & $ 3.031 \times 10^{1}$ & $ 6.418 \times 10^{1}$\\ \cline{3-5} &  & $50 \%$ & $ 2.632 \times 10^{2}$ & $ 3.247 \times 10^{2}$\\ \cline{2-5} &  \multirow{2}{*}{2}  & $100 \%$ & $ 3.046 \times 10^{1}$ & $ 6.597 \times 10^{1}$\\ \cline{3-5} &  & $50 \%$ & $ 1.641 \times 10^{2}$ & $ 4.138 \times 10^{2}$\\ \cline{2-5} &  \multirow{2}{*}{3}  & $100 \%$ & $ 3.047 \times 10^{1}$ & $ 7.565 \times 10^{1}$\\ \cline{3-5} &  & $50 \%$ & $ 1.964 \times 10^{2}$ & $ 4.418 \times 10^{2}$\\ \cline{2-5} &  \multirow{2}{*}{4}  & $100 \%$ & $ 3.052 \times 10^{1}$ & $ 9.332 \times 10^{1}$\\ \cline{3-5} &  & $50 \%$ & $ 2.084 \times 10^{2}$ & $ 4.832 \times 10^{2}$\\ \cline{2-5} &  \multirow{2}{*}{5}  & $100 \%$ & $ 3.054 \times 10^{1}$ & $ 1.151 \times 10^{2}$\\ \cline{3-5} &  & $50 \%$ & $ 2.428 \times 10^{2}$ & $ 5.029 \times 10^{2}$\\ \cline{1-5} \multirow{10}{*}{Temperature ($K$)}  &  \multirow{2}{*}{1}  & $100 \%$ & $ 9.404 \times 10^{2}$ & $ 5.078 \times 10^{2}$\\ \cline{3-5} &  & $50 \%$ & $ 6.644 \times 10^{2}$ & $ 7.059 \times 10^{2}$\\ \cline{2-5} &  \multirow{2}{*}{2}  & $100 \%$ & $ 9.416 \times 10^{2}$ & $ 4.112 \times 10^{2}$\\ \cline{3-5} &  & $50 \%$ & $ 6.129 \times 10^{2}$ & $ 1.138 \times 10^{3}$\\ \cline{2-5} &  \multirow{2}{*}{3}  & $100 \%$ & $ 9.425 \times 10^{2}$ & $ 3.447 \times 10^{2}$\\ \cline{3-5} &  & $50 \%$ & $ 5.815 \times 10^{2}$ & $ 1.389 \times 10^{3}$\\ \cline{2-5} &  \multirow{2}{*}{4}  & $100 \%$ & $ 9.432 \times 10^{2}$ & $ 2.939 \times 10^{2}$\\ \cline{3-5} &  & $50 \%$ & $ 5.585 \times 10^{2}$ & $ 1.355 \times 10^{3}$\\ \cline{2-5} &  \multirow{2}{*}{5}  & $100 \%$ & $ 9.432 \times 10^{2}$ & $ 2.554 \times 10^{2}$\\ \cline{3-5} &  & $50 \%$ & $ 5.500 \times 10^{2}$ & $ 1.104 \times 10^{3}$\\ \cline{1-5} \multirow{10}{*}{Specific Humidity ($g/Kg$)}  &  \multirow{2}{*}{1}  & $100 \%$ & $ 1.733 \times 10^{1}$ & $ 5.427 \times 10^{1}$\\ \cline{3-5} &  & $50 \%$ & $ 8.680 \times 10^{1}$ & $ 7.602 \times 10^{1}$\\ \cline{2-5} &  \multirow{2}{*}{2}  & $100 \%$ & $ 1.712 \times 10^{1}$ & $ 5.669 \times 10^{1}$\\ \cline{3-5} &  & $50 \%$ & $ 8.204 \times 10^{1}$ & $ 1.045 \times 10^{2}$\\ \cline{2-5} &  \multirow{2}{*}{3}  & $100 \%$ & $ 1.705 \times 10^{1}$ & $ 6.630 \times 10^{1}$\\ \cline{3-5} &  & $50 \%$ & $ 8.089 \times 10^{1}$ & $ 1.298 \times 10^{2}$\\ \cline{2-5} &  \multirow{2}{*}{4}  & $100 \%$ & $ 1.699 \times 10^{1}$ & $ 7.344 \times 10^{1}$\\ \cline{3-5} &  & $50 \%$ & $ 7.525 \times 10^{1}$ & $ 1.431 \times 10^{2}$\\ \cline{2-5} &  \multirow{2}{*}{5}  & $100 \%$ & $ 1.694 \times 10^{1}$ & $ 7.617 \times 10^{1}$\\ \cline{3-5} &  & $50 \%$ & $ 7.642 \times 10^{1}$ & $ 1.458 \times 10^{2}$\\ \cline{1-5}
\end{tabular}
}
\caption{RMSE values for the EnKF-MC and the LETKF analyses with the SPEEDY model and for different values for $\ra$ and $p$. Dense observational networks are considered in this experimental setting.}
\label{tab:exp-RMSE-values-all-dense}
\end{table}

%%%%%%%%%%%%%%%%%%%%%%%%%%%%%%%%%%
\subsection{Results with sparse observation networks}
%%%%%%%%%%%%%%%%%%%%%%%%%%%%%%%%%%

For sparse observational networks, in general, the results obtained by the EnKF-MC are more accurate than those obtained by the LETKF, as reported in the Tables \ref{tab:exp-RMSE-values-wind-components-sparse} and \ref{tab:exp-RMSE-values-others-sparse}. We vary the values of $\ra$ from 1 to 5. Three sparse observational networks with $p=12\%$, $6\%$, and $4\%$, respectively are considered. 

Figure \ref{fig:exp-LETKF-RMSE-different-radii-sparse-network} shows the RMSE values of the specific humidity analyses for different radii of influence and $4\%$ of the model components being observed. The best performance of the LETKF analyses is obtained when the radius of influence is set to 2. Note that for $\ra=1$ the LETKF performs poorly, which is expected since during the assimilation most of model components will not have observations in their local boxes. For $\ra\ge 3$ the effects of spurious correlations degrade the quality of the LETKF analysis. On the other hand, the background error correlations estimated by the modified Cholesky decomposition allows the EnKF-MC formulation to obtain good analyses even for largest radius of influence $\ra=5$. 

Figure \ref{fig:exp-LETKF-RMSE-different-radii-sparse} shows the RMSE values of the LETKF and the EnKF-MC implementations for different radii of influences and two sparse observational networks. Clearly, when the radius of influence is increased, in the LETKF context, the analysis corrections are impacted by spurious correlations. On the other hand, the quality of the results in the EnKF-MC case is considerably better. When data errors components are uncorrelated $\ra$ can be seen as a free parameter and the choice can be based on the ``optimal performance of the filter''. For the largest radius of influence $\ra=5$ the RMSE values of the ENKF-MC and the LETKF implementations differ by one order of magnitude.

Figure \ref{fig:exp-model-variables} reports the RMSE values for the zonal and the meridional wind component analyses, and for different values of $p$ and $\ra$. As can be seen, the estimation of background errors via $\BE$ can reduce the impact of spurious correlations; the RMSE values of the EnKF-MC analyses remain small at all assimilation times, from which we infer that the background error correlations are properly estimated. On the other hand, the impact of spurious correlations is evident in the context of LETKF. Since most of the model components are unobserved, the background error correlations drive the quality of the analysis, and spurious correlations lead to a poor performance of the filter at many assimilation times. 
\begin{figure}[H]
\centering
\begin{subfigure}{0.5\textwidth}
\centering
\includegraphics[width=0.9\textwidth,height=0.75\textwidth]{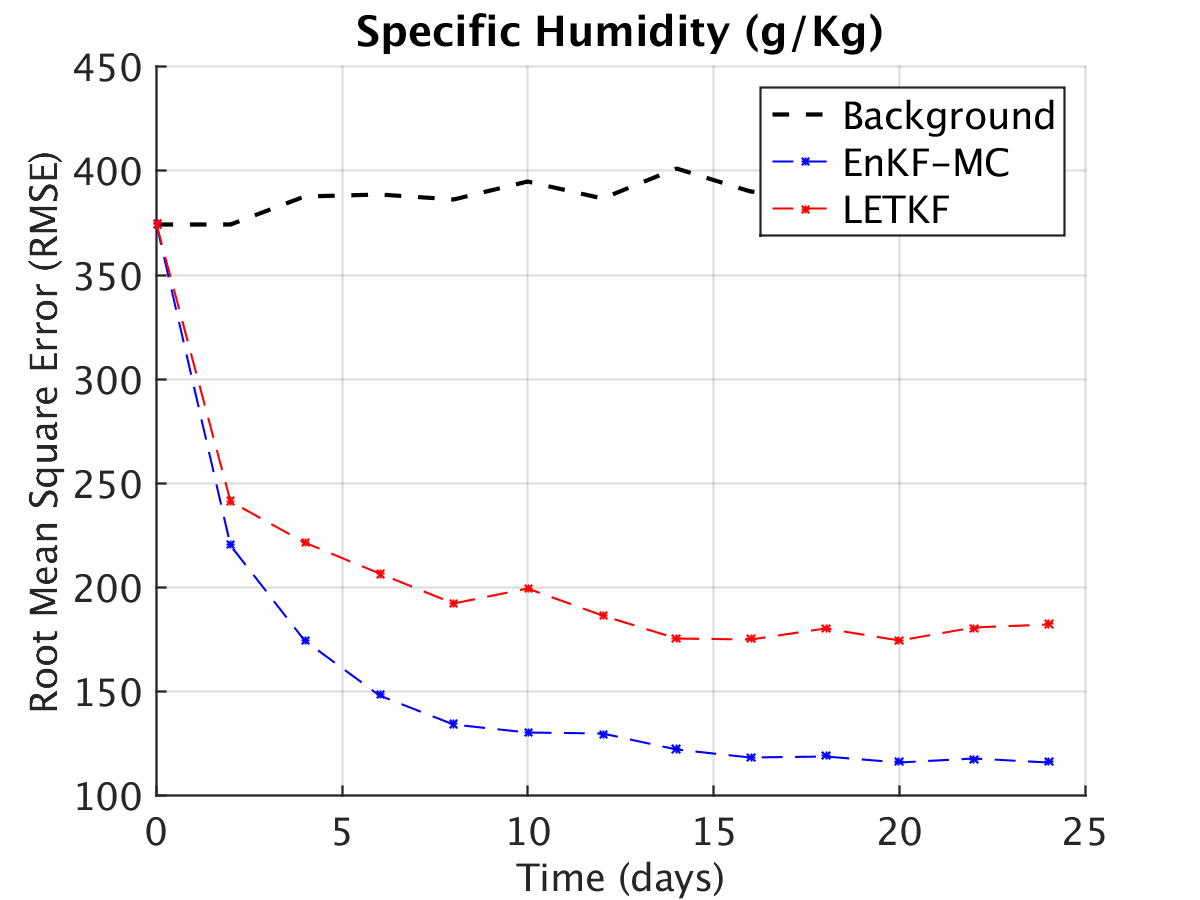}
\caption{$\ra  = 1$.}
\end{subfigure}%
\begin{subfigure}{0.5\textwidth}
\centering
\includegraphics[width=0.9\textwidth,height=0.75\textwidth]{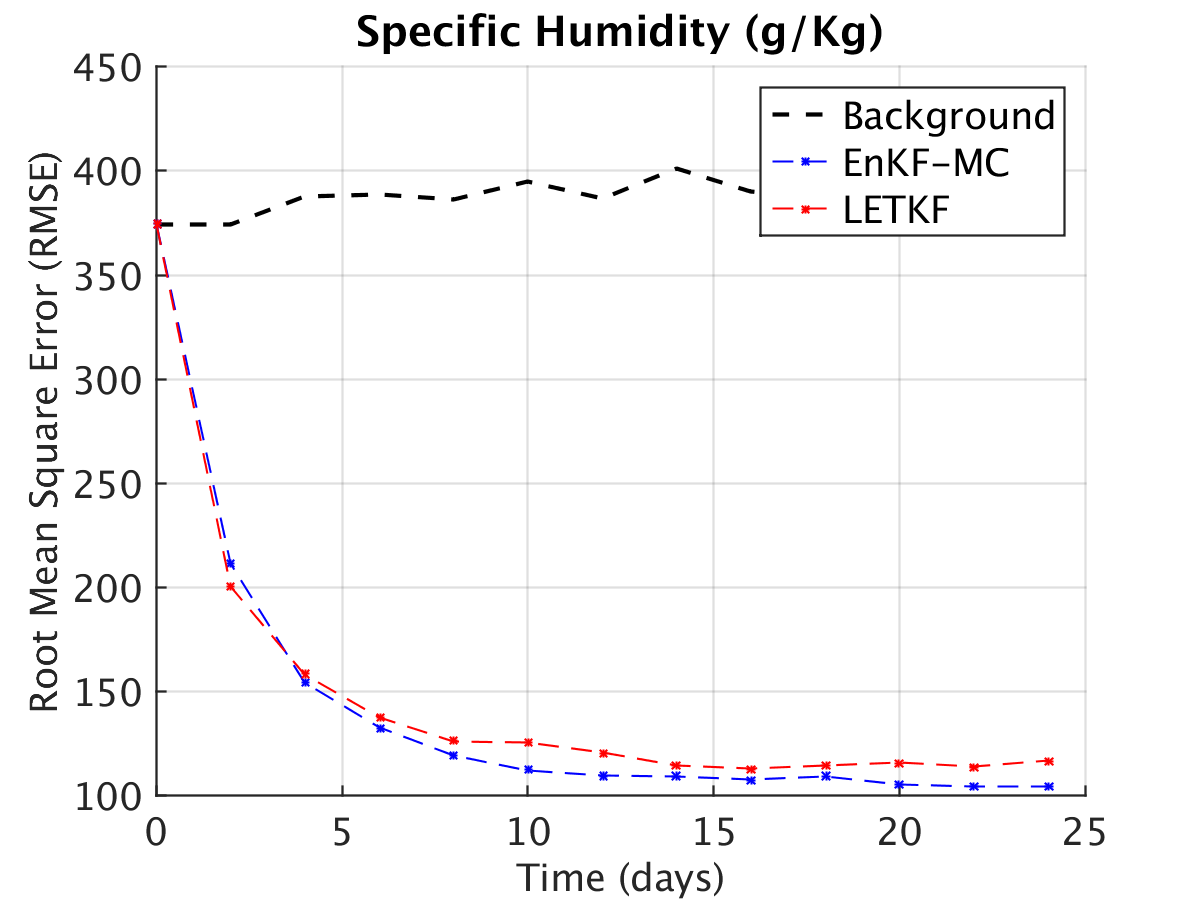}
\caption{$\ra  = 2$.}
\end{subfigure}
\begin{subfigure}{0.5\textwidth}
\centering
\includegraphics[width=0.9\textwidth,height=0.75\textwidth]{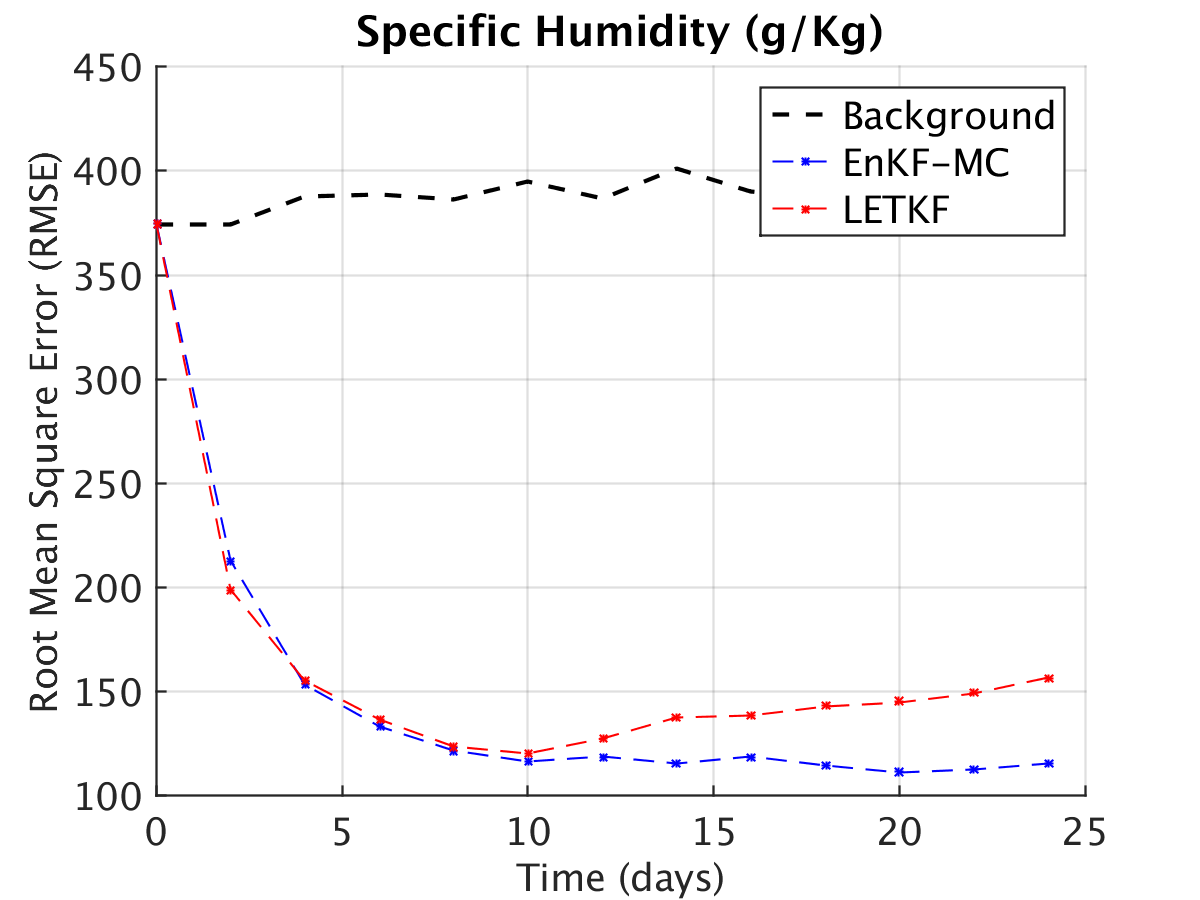}
\caption{$\ra  = 3$.}
\end{subfigure}%
\begin{subfigure}{0.5\textwidth}
\centering
\includegraphics[width=0.9\textwidth,height=0.75\textwidth]{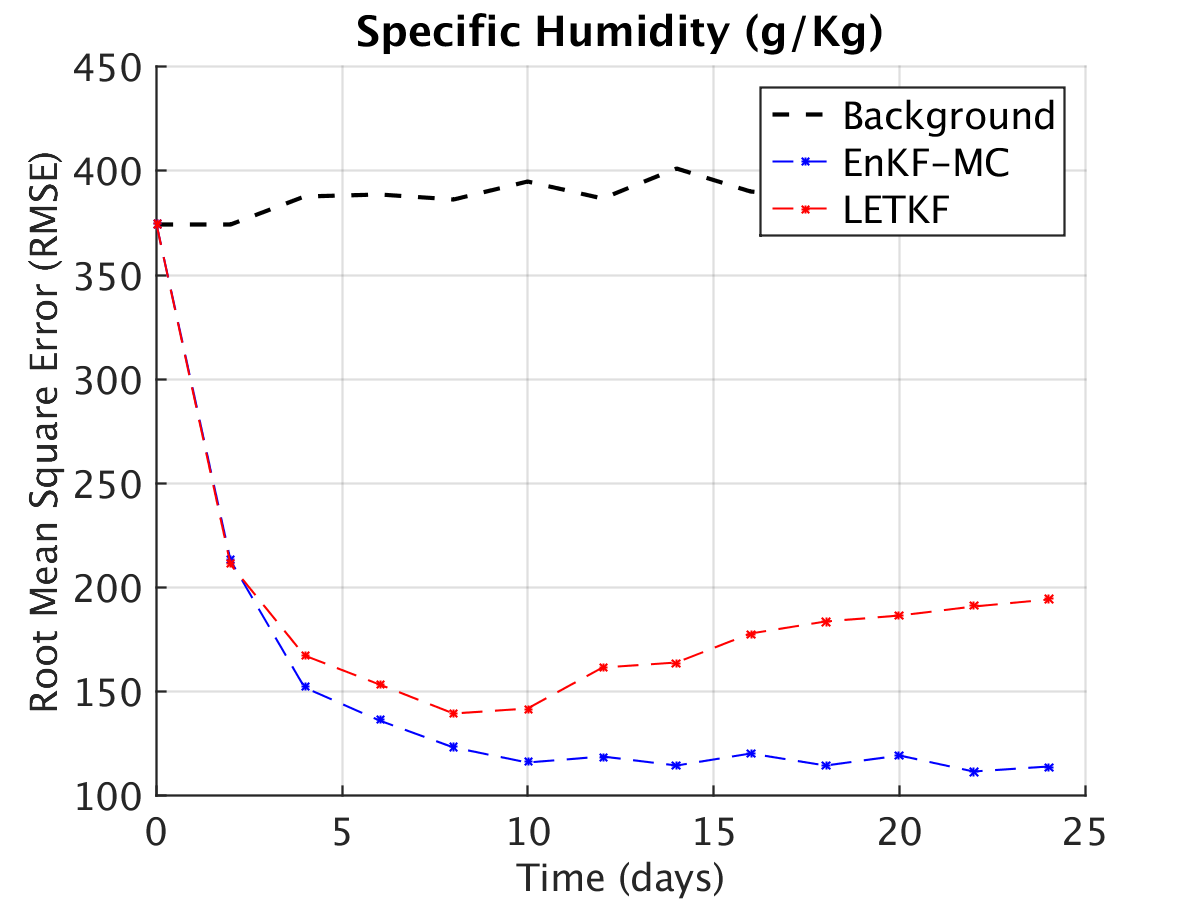}
\caption{$\ra  = 4$.}
\end{subfigure}
\begin{subfigure}{0.5\textwidth}
\centering
\includegraphics[width=0.9\textwidth,height=0.75\textwidth]{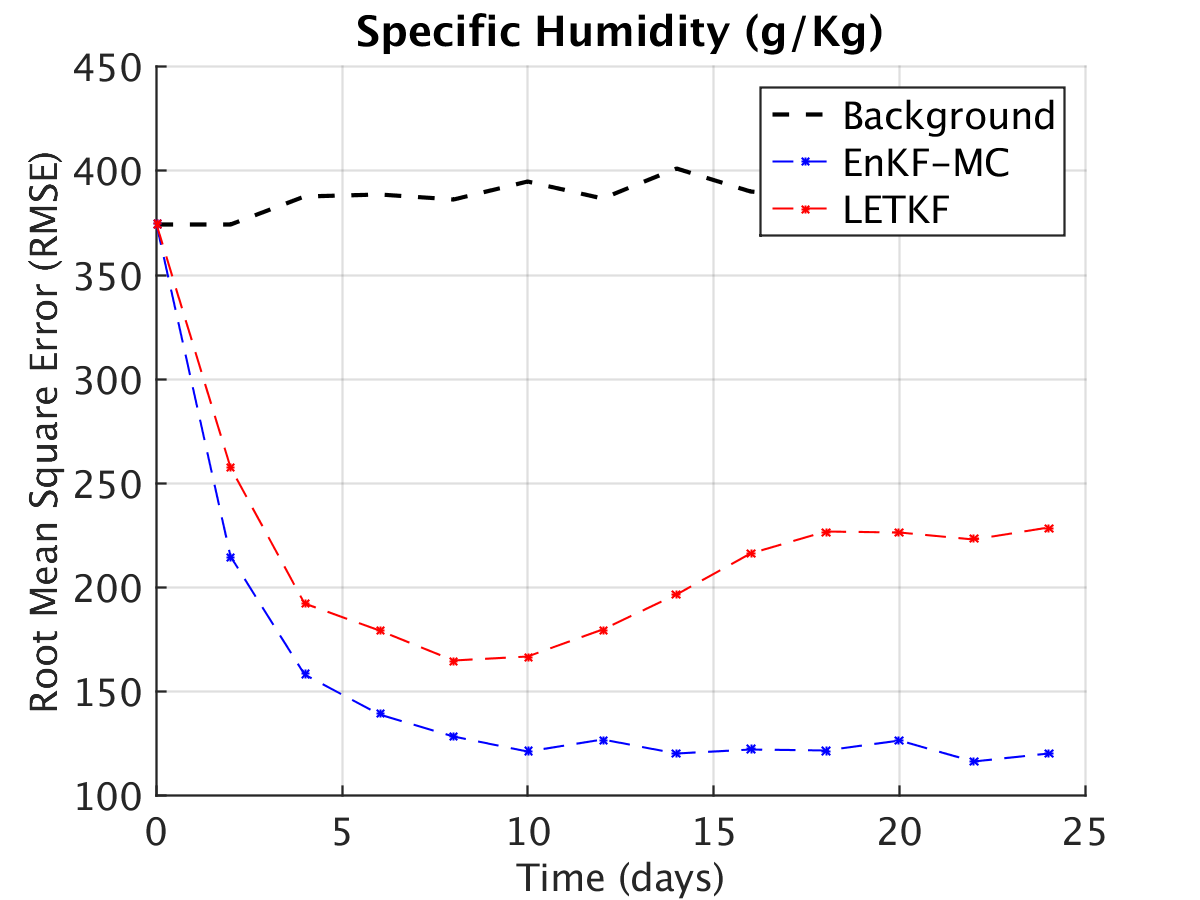}
\caption{$\ra  = 5$.}
\end{subfigure}
\caption{RMSE of specific humidity analyses with a sparse observational network ($p \sim 4\%$) and different values of $\ra$. }
\label{fig:exp-LETKF-RMSE-different-radii-sparse-network}
\end{figure}
\begin{figure}[H]
\centering
\begin{subfigure}{0.5\textwidth}
\centering
\includegraphics[width=0.9\textwidth,height=0.75\textwidth]{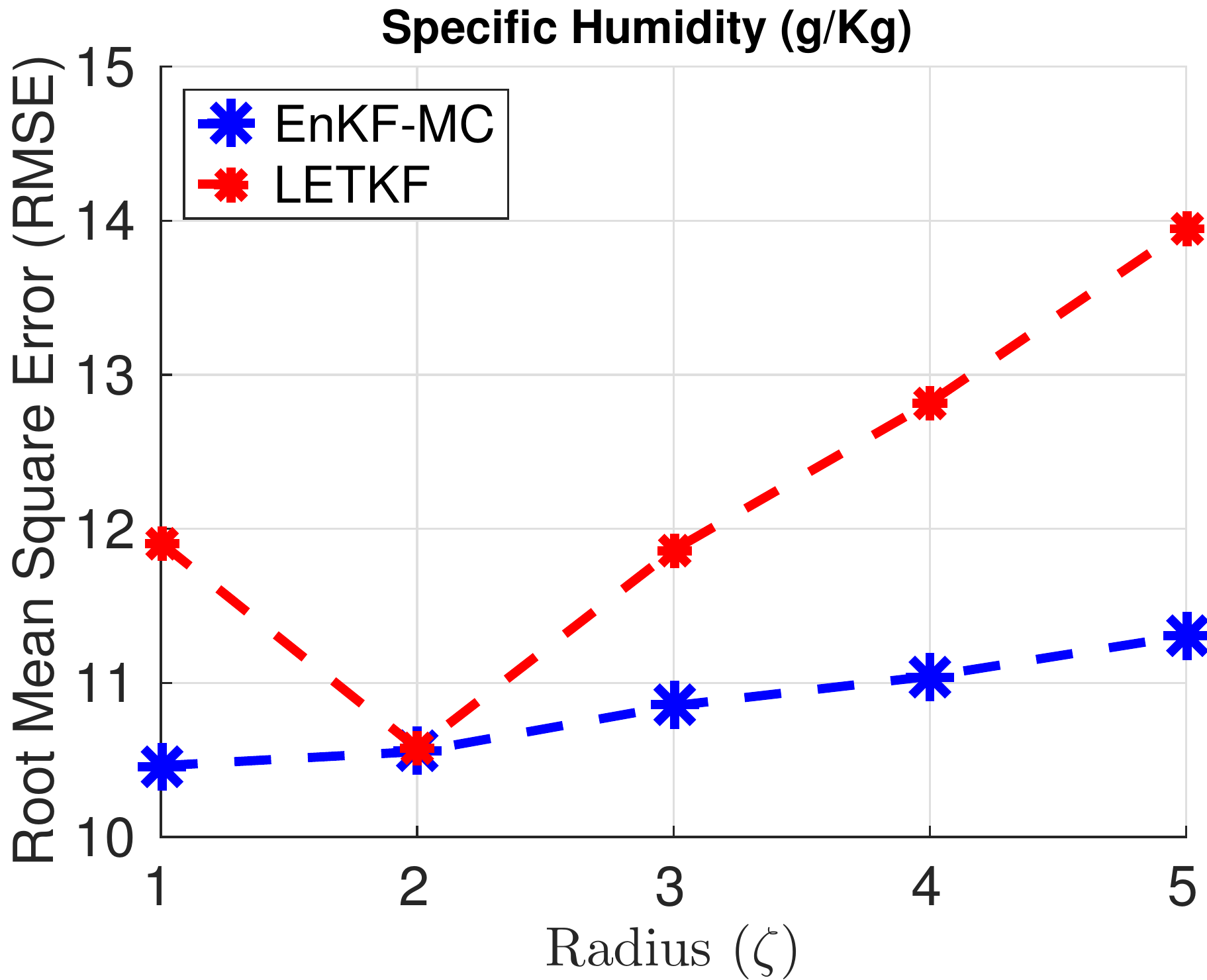}
\caption{$p  = 6\%$.}
\end{subfigure}%
\begin{subfigure}{0.5\textwidth}
\centering
\includegraphics[width=0.9\textwidth,height=0.75\textwidth]{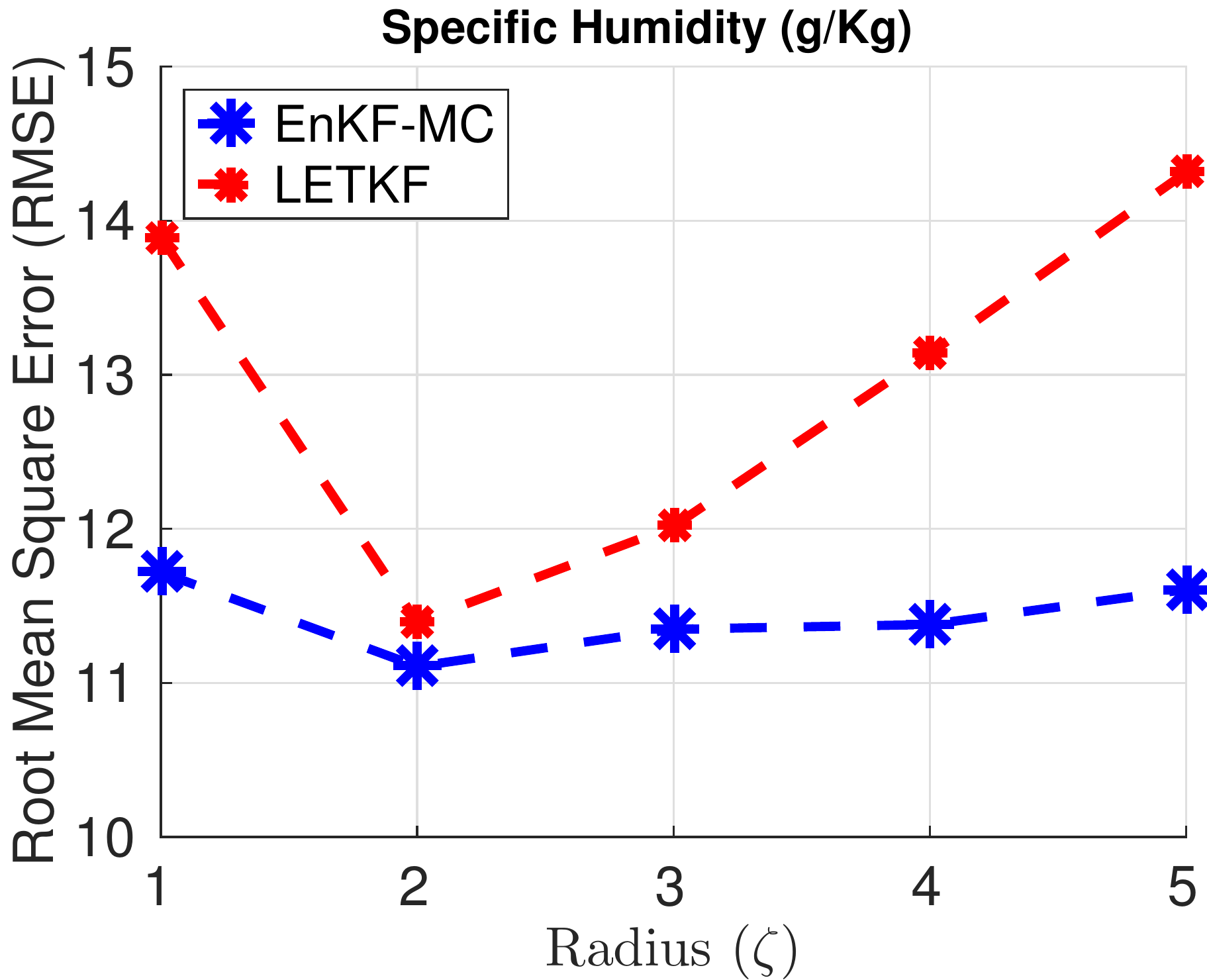}
\caption{$p  = 4\%$.}
\end{subfigure}%
\caption{Analysis RMSE for the specific humidity variable with sparse observation networks. RMSE values are shown for different values of $\ra$ and percentage of observed components $p$.}
\label{fig:exp-LETKF-RMSE-different-radii-sparse}
\end{figure}
Figures \ref{fig:exp-snapshot-meridional-wind-component} and \ref{fig:exp-snapshot-zonal-wind-component} provide snapshots of the meridional and the zonal wind components, respectively, at the first assimilation time. For this particular case the percentage of observed model components is $p=4\%$. At this step, only the initial observation has been assimilated in order to compute the analysis corrections by the EnKF-MC and the LETKF methods. The background solution contains erroneous waves for the zonal and the meridional wind components. For instance, for the $u$ model variable, such waves are clearly present near the poles. After the first assimilation step, the LETKF analysis solution dissipates the erroneous waves but, the numerical values of the wind components are slightly greater than those of the reference solutions. This numerical difference increases at later times due to the highly-nonlinear dynamics of SPEEDY, as can bee seen in Figure \ref{fig:exp-model-variables}. On the other hand, the EnKF-MC implementation recovers the reference shape, and the analysis values of the numerical model components are close to that of the reference solution. This shows again that the use of the modified Cholesky decomposition as the estimator of the background error correlations can mitigate the impact of spurious error correlations. 
\begin{table}[H]
\centering
{\footnotesize
\begin{tabular}{|c|c|c|c|c|} \hline
Variable (units) &$\ra$ & $p$ & EnKF-MC & LETKF \\ \hline
 \multirow{15}{*}{Zonal Wind Component ($u$), ($m/s$)}  &  \multirow{3}{*}{1}  & $12\%$ & $ 5.514 \times 10^{2}$ & $ 5.471 \times 10^{2}$\\ \cline{3-5} &  & $ 6 \%$ & $ 6.972 \times 10^{2}$ & $ 1.168 \times 10^{3}$\\ \cline{3-5} &  & $4\%$ & $ 9.393 \times 10^{2}$ & $ 1.737 \times 10^{3}$\\ \cline{2-5} &  \multirow{3}{*}{2}  & $12\%$ & $ 4.187 \times 10^{2}$ & $ 1.275 \times 10^{3}$\\ \cline{3-5} &  & $ 6 \%$ & $ 6.090 \times 10^{2}$ & $ 7.591 \times 10^{2}$\\ \cline{3-5} &  & $4\%$ & $ 7.853 \times 10^{2}$ & $ 8.569 \times 10^{2}$\\ \cline{2-5} &  \multirow{3}{*}{3}  & $12\%$ & $ 4.388 \times 10^{2}$ & $ 1.661 \times 10^{3}$\\ \cline{3-5} &  & $ 6 \%$ & $ 6.146 \times 10^{2}$ & $ 1.237 \times 10^{3}$\\ \cline{3-5} &  & $4\%$ & $ 7.438 \times 10^{2}$ & $ 9.997 \times 10^{2}$\\ \cline{2-5} &  \multirow{3}{*}{4}  & $12\%$ & $ 4.323 \times 10^{2}$ & $ 1.752 \times 10^{3}$\\ \cline{3-5} &  & $ 6 \%$ & $ 5.990 \times 10^{2}$ & $ 1.608 \times 10^{3}$\\ \cline{3-5} &  & $4\%$ & $ 7.124 \times 10^{2}$ & $ 1.258 \times 10^{3}$\\ \cline{2-5} &  \multirow{3}{*}{5}  & $12\%$ & $ 4.456 \times 10^{2}$ & $ 1.862 \times 10^{3}$\\ \cline{3-5} &  & $ 6 \%$ & $ 6.106 \times 10^{2}$ & $ 1.983 \times 10^{3}$\\ \cline{3-5} &  & $4\%$ & $ 7.160 \times 10^{2}$ & $ 1.602 \times 10^{3}$\\ \cline{1-5} \multirow{15}{*}{Meridional Wind Component ($v$) ($m/s$)}  &  \multirow{3}{*}{1}  & $12\%$ & $ 3.540 \times 10^{2}$ & $ 4.496 \times 10^{2}$\\ \cline{3-5} &  & $ 6 \%$ & $ 5.165 \times 10^{2}$ & $ 1.158 \times 10^{3}$\\ \cline{3-5} &  & $4\%$ & $ 7.770 \times 10^{2}$ & $ 1.749 \times 10^{3}$\\ \cline{2-5} &  \multirow{3}{*}{2}  & $12\%$ & $ 3.009 \times 10^{2}$ & $ 7.285 \times 10^{2}$\\ \cline{3-5} &  & $ 6 \%$ & $ 4.605 \times 10^{2}$ & $ 5.520 \times 10^{2}$\\ \cline{3-5} &  & $4\%$ & $ 6.217 \times 10^{2}$ & $ 7.420 \times 10^{2}$\\ \cline{2-5} &  \multirow{3}{*}{3}  & $12\%$ & $ 3.172 \times 10^{2}$ & $ 9.510 \times 10^{2}$\\ \cline{3-5} &  & $ 6 \%$ & $ 4.735 \times 10^{2}$ & $ 8.334 \times 10^{2}$\\ \cline{3-5} &  & $4\%$ & $ 6.014 \times 10^{2}$ & $ 7.455 \times 10^{2}$\\ \cline{2-5} &  \multirow{3}{*}{4}  & $12\%$ & $ 3.399 \times 10^{2}$ & $ 1.048 \times 10^{3}$\\ \cline{3-5} &  & $ 6 \%$ & $ 4.812 \times 10^{2}$ & $ 1.146 \times 10^{3}$\\ \cline{3-5} &  & $4\%$ & $ 5.913 \times 10^{2}$ & $ 9.026 \times 10^{2}$\\ \cline{2-5} &  \multirow{3}{*}{5}  & $12\%$ & $ 3.626 \times 10^{2}$ & $ 1.101 \times 10^{3}$\\ \cline{3-5} &  & $ 6 \%$ & $ 5.107 \times 10^{2}$ & $ 1.575 \times 10^{3}$\\ \cline{3-5} &  & $4\%$ & $ 6.122 \times 10^{2}$ & $ 1.102 \times 10^{3}$\\ \cline{1-5} 
\end{tabular}
}
\caption{RMSE values of the wind-components for the EnKF-MC and LETKF making use of the SPEEDY model.}
\label{tab:exp-RMSE-values-wind-components-sparse}
\end{table}

\begin{table}[H]
\centering
{\footnotesize
\begin{tabular}{|c|c|c|c|c|} \hline
Variable (units) &$\ra$ & $p$ & EnKF-MC & LETKF \\ \hline
\multirow{15}{*}{Temperature ($K$)}  &  \multirow{3}{*}{1}  & $12\%$ & $ 6.054 \times 10^{2}$ & $ 6.033 \times 10^{2}$\\ \cline{3-5} &  & $ 6 \%$ & $ 5.692 \times 10^{2}$ & $ 6.704 \times 10^{2}$\\ \cline{3-5} &  & $4\%$ & $ 6.522 \times 10^{2}$ & $ 8.073 \times 10^{2}$\\ \cline{2-5} &  \multirow{3}{*}{2}  & $12\%$ & $ 5.680 \times 10^{2}$ & $ 6.693 \times 10^{2}$\\ \cline{3-5} &  & $ 6 \%$ & $ 5.193 \times 10^{2}$ & $ 5.556 \times 10^{2}$\\ \cline{3-5} &  & $4\%$ & $ 5.299 \times 10^{2}$ & $ 5.529 \times 10^{2}$\\ \cline{2-5} &  \multirow{3}{*}{3}  & $12\%$ & $ 5.279 \times 10^{2}$ & $ 1.217 \times 10^{3}$\\ \cline{3-5} &  & $ 6 \%$ & $ 4.982 \times 10^{2}$ & $ 6.458 \times 10^{2}$\\ \cline{3-5} &  & $4\%$ & $ 4.926 \times 10^{2}$ & $ 6.073 \times 10^{2}$\\ \cline{2-5} &  \multirow{3}{*}{4}  & $12\%$ & $ 5.023 \times 10^{2}$ & $ 1.817 \times 10^{3}$\\ \cline{3-5} &  & $ 6 \%$ & $ 4.757 \times 10^{2}$ & $ 1.030 \times 10^{3}$\\ \cline{3-5} &  & $4\%$ & $ 4.766 \times 10^{2}$ & $ 7.464 \times 10^{2}$\\ \cline{2-5} &  \multirow{3}{*}{5}  & $12\%$ & $ 4.898 \times 10^{2}$ & $ 1.600 \times 10^{3}$\\ \cline{3-5} &  & $ 6 \%$ & $ 4.644 \times 10^{2}$ & $ 1.473 \times 10^{3}$\\ \cline{3-5} &  & $4\%$ & $ 4.684 \times 10^{2}$ & $ 1.172 \times 10^{3}$\\ \cline{1-5} \multirow{15}{*}{Specific Humidity ($g/Kg$)}  &  \multirow{3}{*}{1}  & $12\%$ & $ 9.862 \times 10^{1}$ & $ 9.026 \times 10^{1}$\\ \cline{3-5} &  & $ 6 \%$ & $ 1.133 \times 10^{2}$ & $ 1.449 \times 10^{2}$\\ \cline{3-5} &  & $4\%$ & $ 1.405 \times 10^{2}$ & $ 1.941 \times 10^{2}$\\ \cline{2-5} &  \multirow{3}{*}{2}  & $12\%$ & $ 1.029 \times 10^{2}$ & $ 1.125 \times 10^{2}$\\ \cline{3-5} &  & $ 6 \%$ & $ 1.146 \times 10^{2}$ & $ 1.137 \times 10^{2}$\\ \cline{3-5} &  & $4\%$ & $ 1.270 \times 10^{2}$ & $ 1.321 \times 10^{2}$\\ \cline{2-5} &  \multirow{3}{*}{3}  & $12\%$ & $ 1.068 \times 10^{2}$ & $ 1.341 \times 10^{2}$\\ \cline{3-5} &  & $ 6 \%$ & $ 1.205 \times 10^{2}$ & $ 1.418 \times 10^{2}$\\ \cline{3-5} &  & $4\%$ & $ 1.317 \times 10^{2}$ & $ 1.458 \times 10^{2}$\\ \cline{2-5} &  \multirow{3}{*}{4}  & $12\%$ & $ 1.065 \times 10^{2}$ & $ 1.640 \times 10^{2}$\\ \cline{3-5} &  & $ 6 \%$ & $ 1.246 \times 10^{2}$ & $ 1.652 \times 10^{2}$\\ \cline{3-5} &  & $4\%$ & $ 1.324 \times 10^{2}$ & $ 1.739 \times 10^{2}$\\ \cline{2-5} &  \multirow{3}{*}{5}  & $12\%$ & $ 1.089 \times 10^{2}$ & $ 2.078 \times 10^{2}$\\ \cline{3-5} &  & $ 6 \%$ & $ 1.301 \times 10^{2}$ & $ 1.950 \times 10^{2}$\\ \cline{3-5} &  & $4\%$ & $ 1.373 \times 10^{2}$ & $ 2.068 \times 10^{2}$\\ \cline{1-5}
\end{tabular}
}
\caption{RMSE values for the EnKF-MC and LETKF making use of the SPEEDY model.}
\label{tab:exp-RMSE-values-others-sparse}
\end{table}
\begin{figure}[H]
\centering
\begin{subfigure}{0.5\textwidth}
\centering
\includegraphics[width=0.9\textwidth,height=0.9\textwidth]{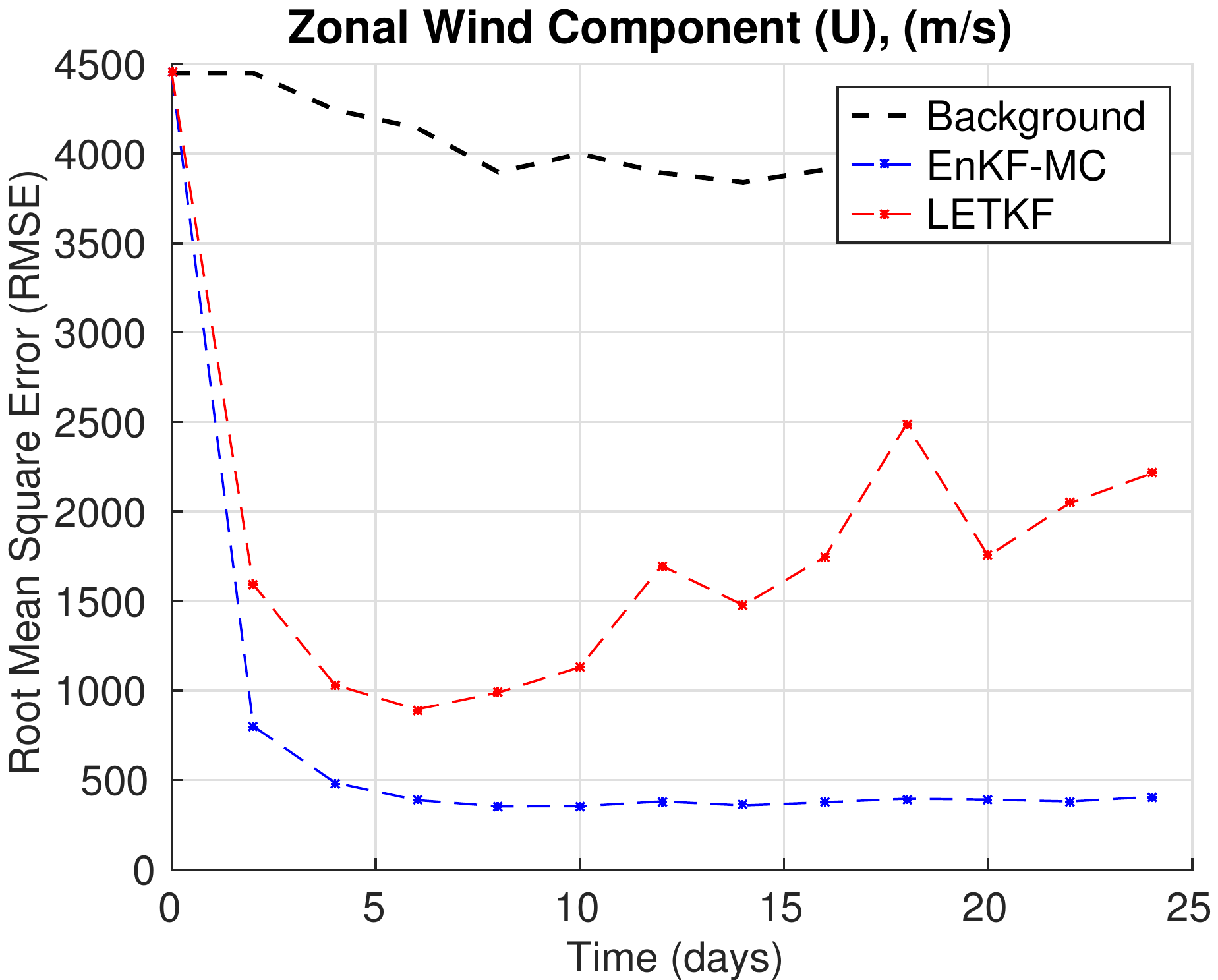}
\caption{$\ra = 3$ and $p=12\%$ }
\end{subfigure}%
\begin{subfigure}{0.5\textwidth}
\centering
\includegraphics[width=0.9\textwidth,height=0.9\textwidth]{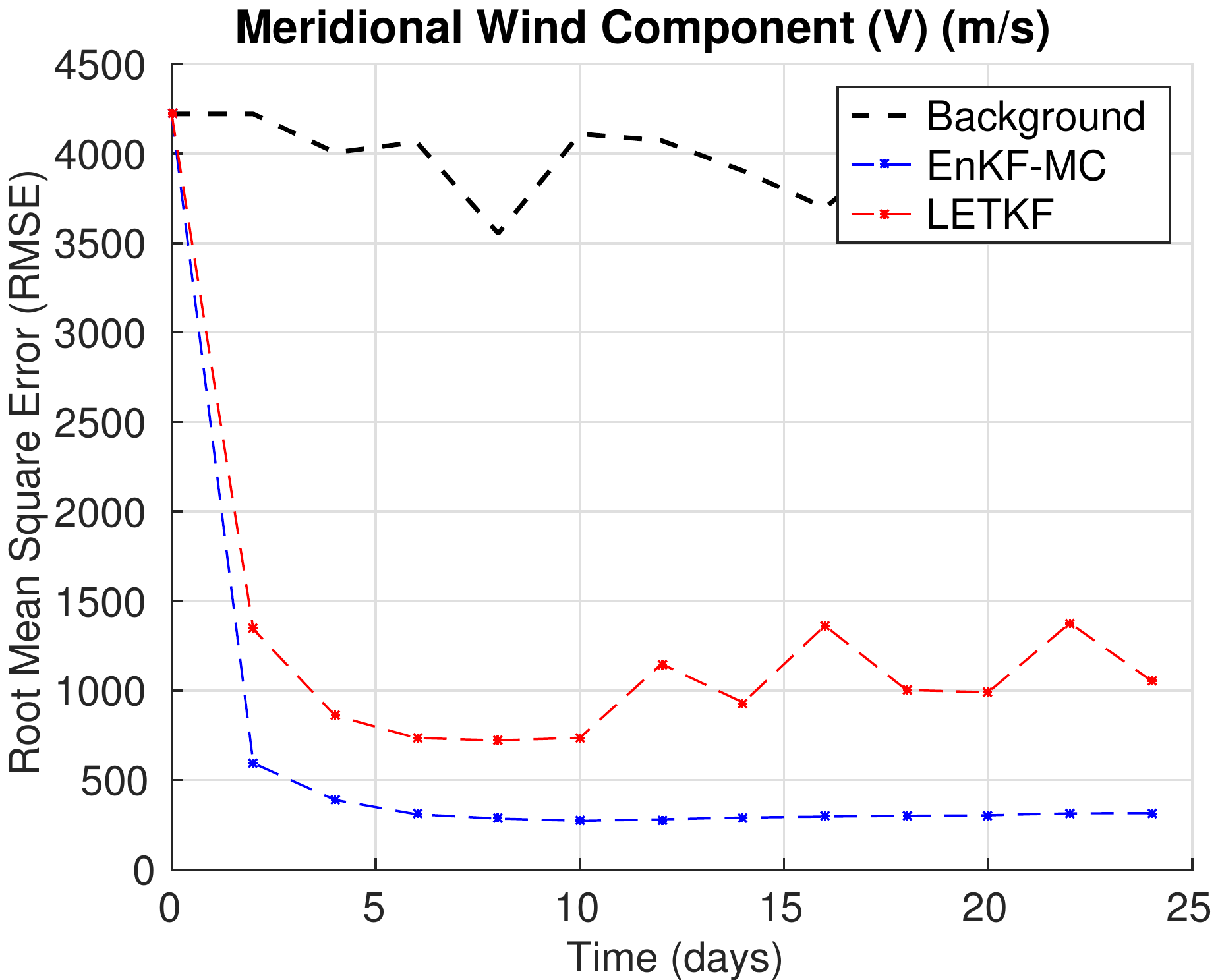}
\caption{$\ra = 4$ and $p=12\%$ }
\end{subfigure}%
\caption{RMSE of the LETKF and EnKF-MC implementations for different model variables, radii of influence and observational networks. }
\label{fig:exp-model-variables}
\end{figure}
\begin{figure}[H]
\centering
\begin{subfigure}{0.5\textwidth}
\centering
\includegraphics[width=0.9\textwidth,height=0.8\textwidth]{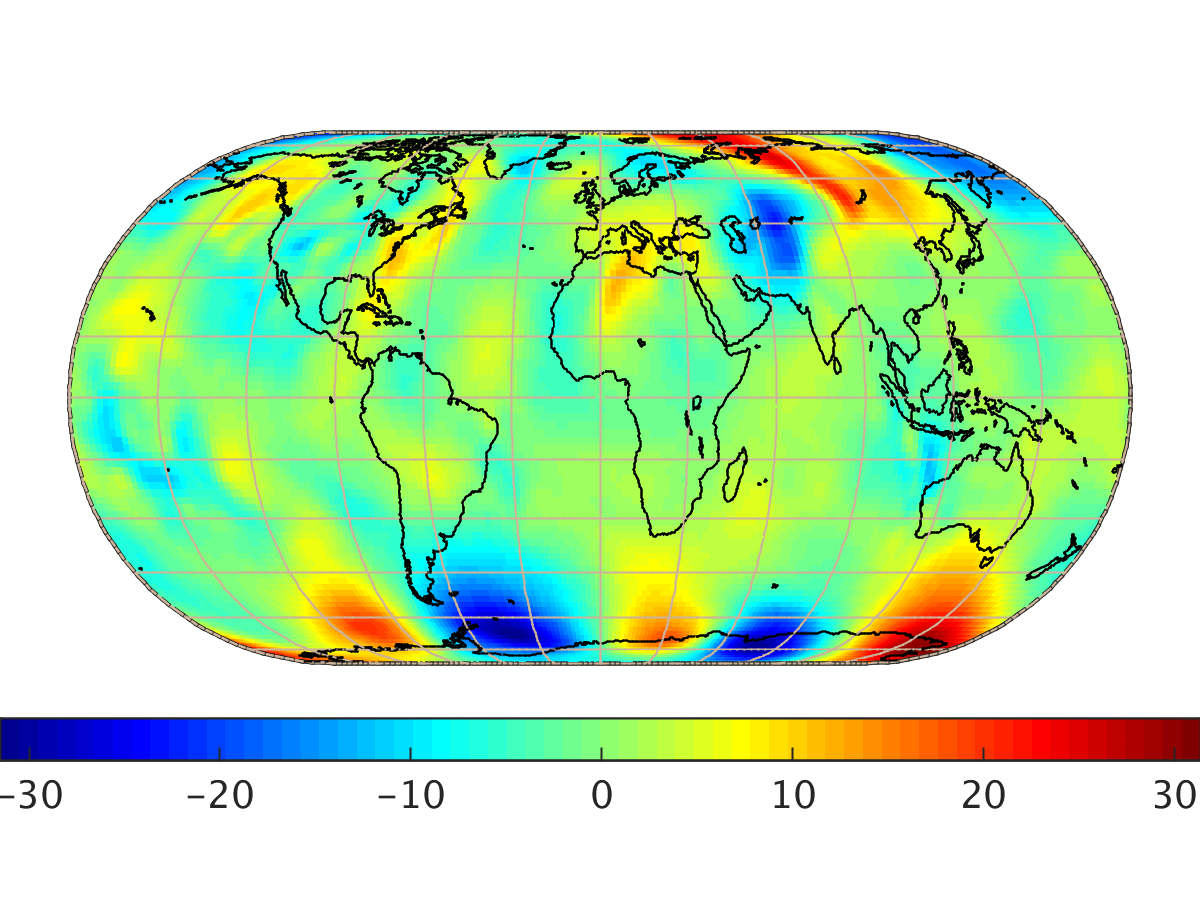}
\caption{Reference}
\end{subfigure}%
\begin{subfigure}{0.5\textwidth}
\centering
\includegraphics[width=0.9\textwidth,height=0.8\textwidth]{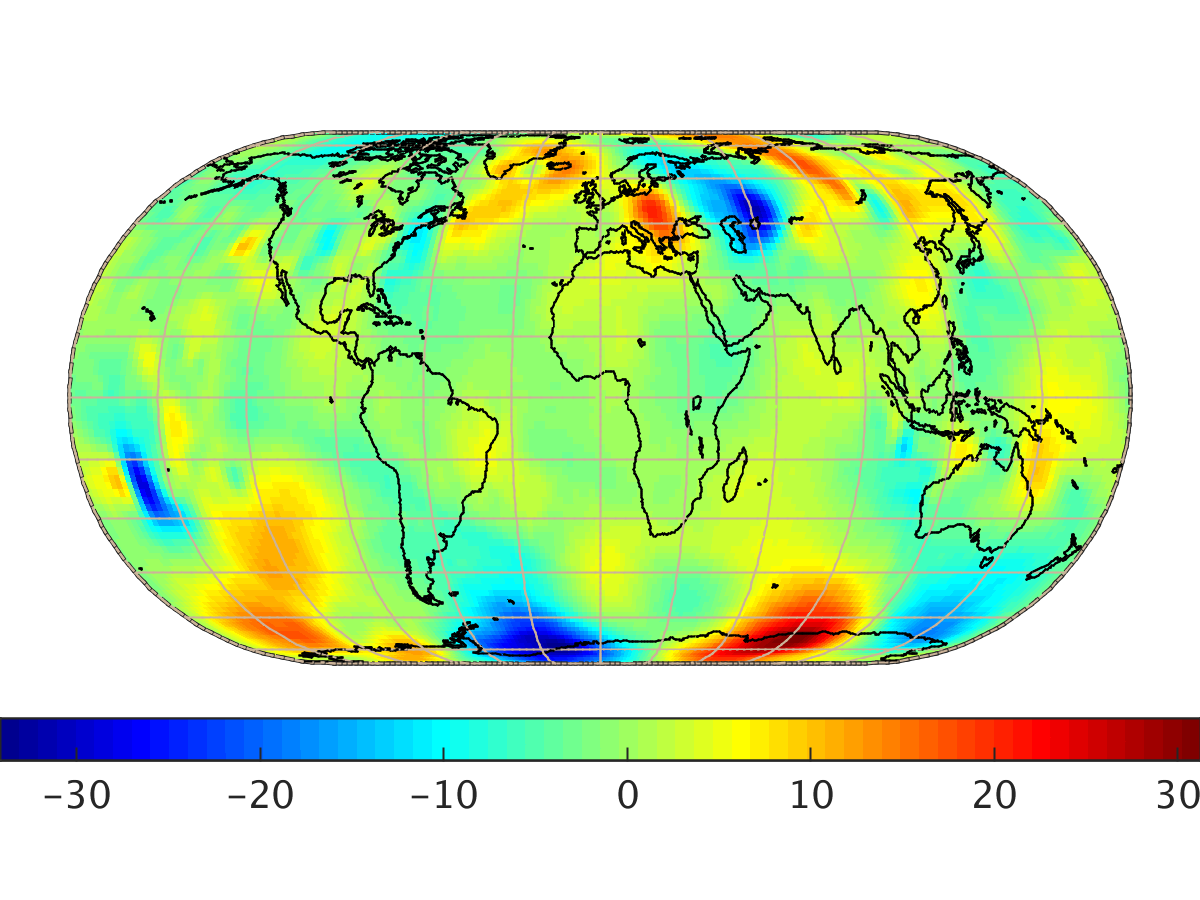}
\caption{Background}
\end{subfigure}

\begin{subfigure}{0.5\textwidth}
\centering
\includegraphics[width=0.9\textwidth,height=0.8\textwidth]{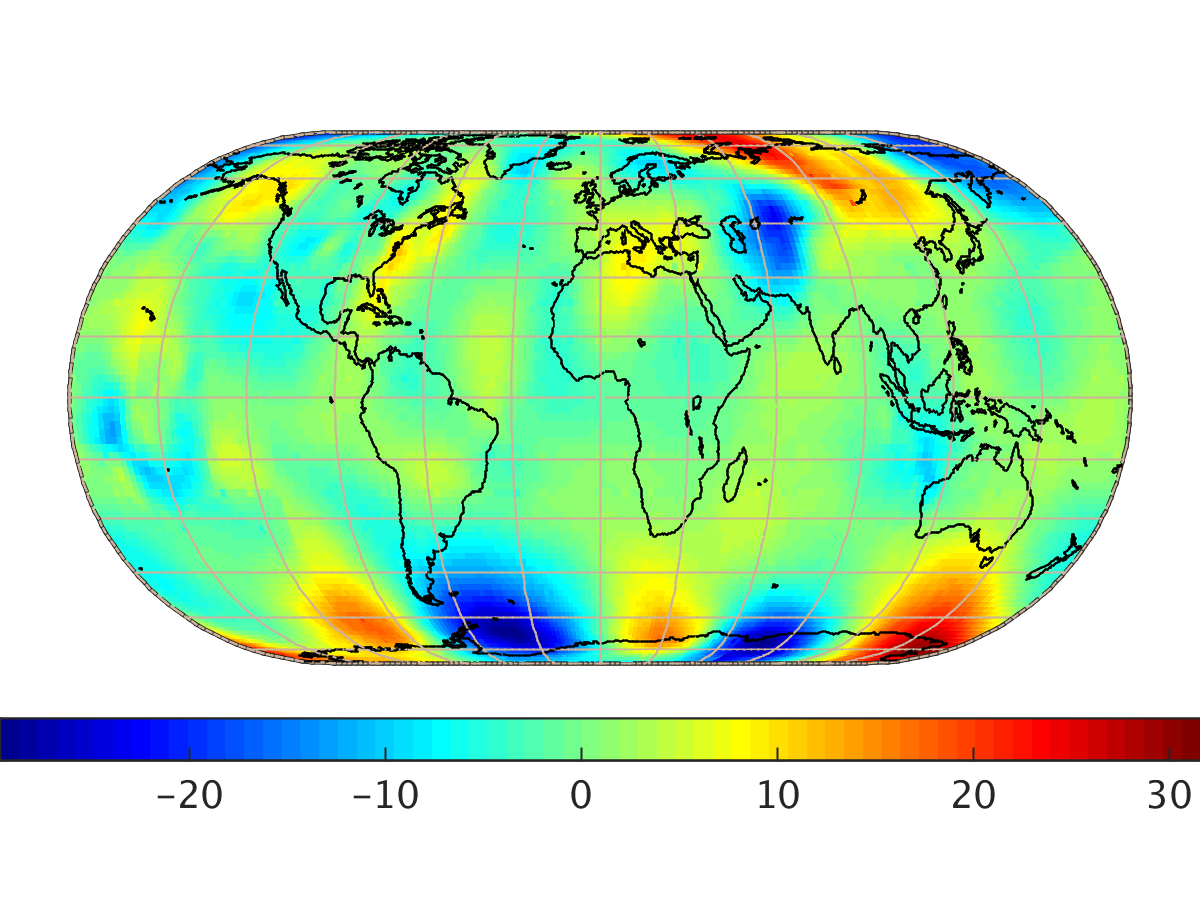}
\caption{EnKF-MC}
\end{subfigure}%
\begin{subfigure}{0.5\textwidth}
\centering
\includegraphics[width=0.9\textwidth,height=0.8\textwidth]{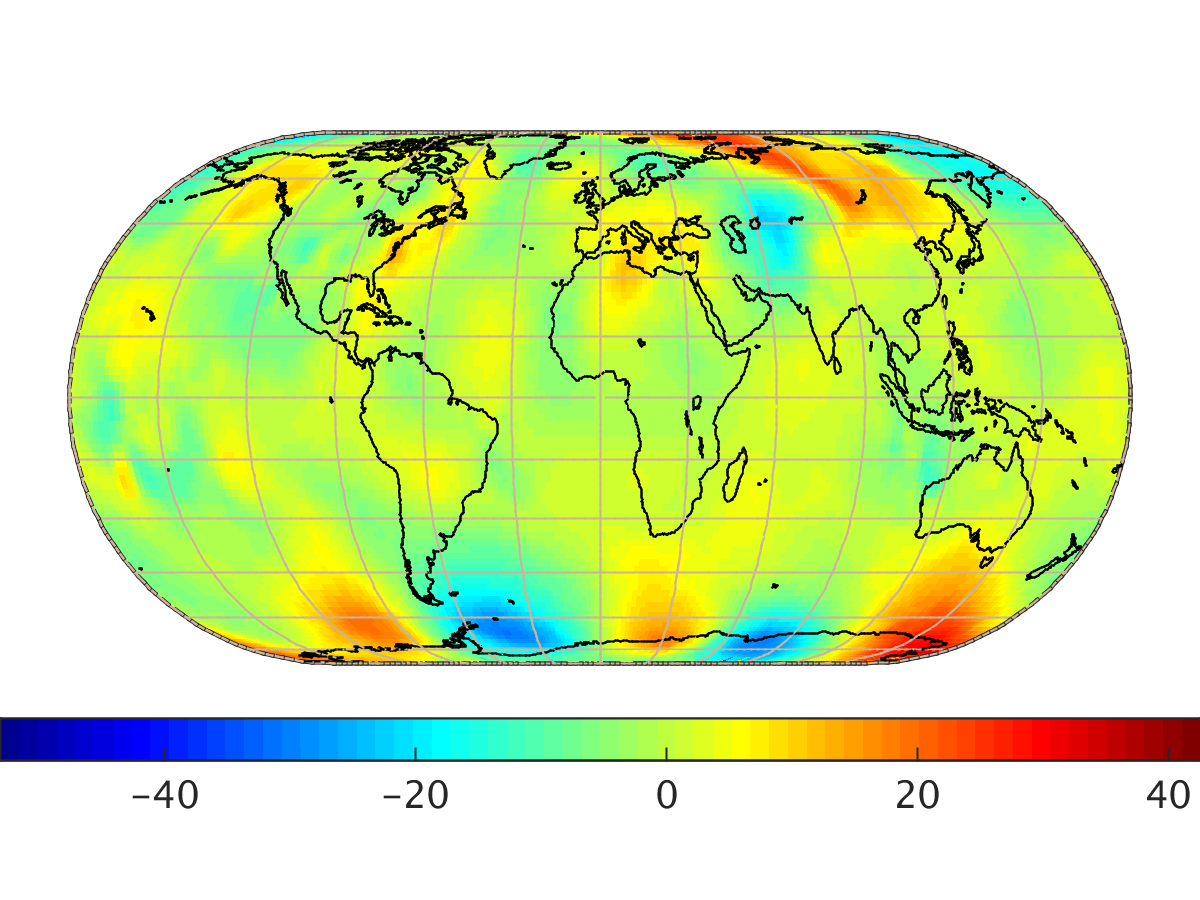}
\caption{LETKF}
\end{subfigure}
\caption{Snapshots of the reference solution, background state, and analysis fields from the EnKF-MC and LETKF for the fifth layer of the meridional wind component ($v$).}
\label{fig:exp-snapshot-meridional-wind-component}
\end{figure}
\begin{figure}[H]
\centering
\begin{subfigure}{0.5\textwidth}
\centering
\includegraphics[width=0.9\textwidth,height=0.8\textwidth]{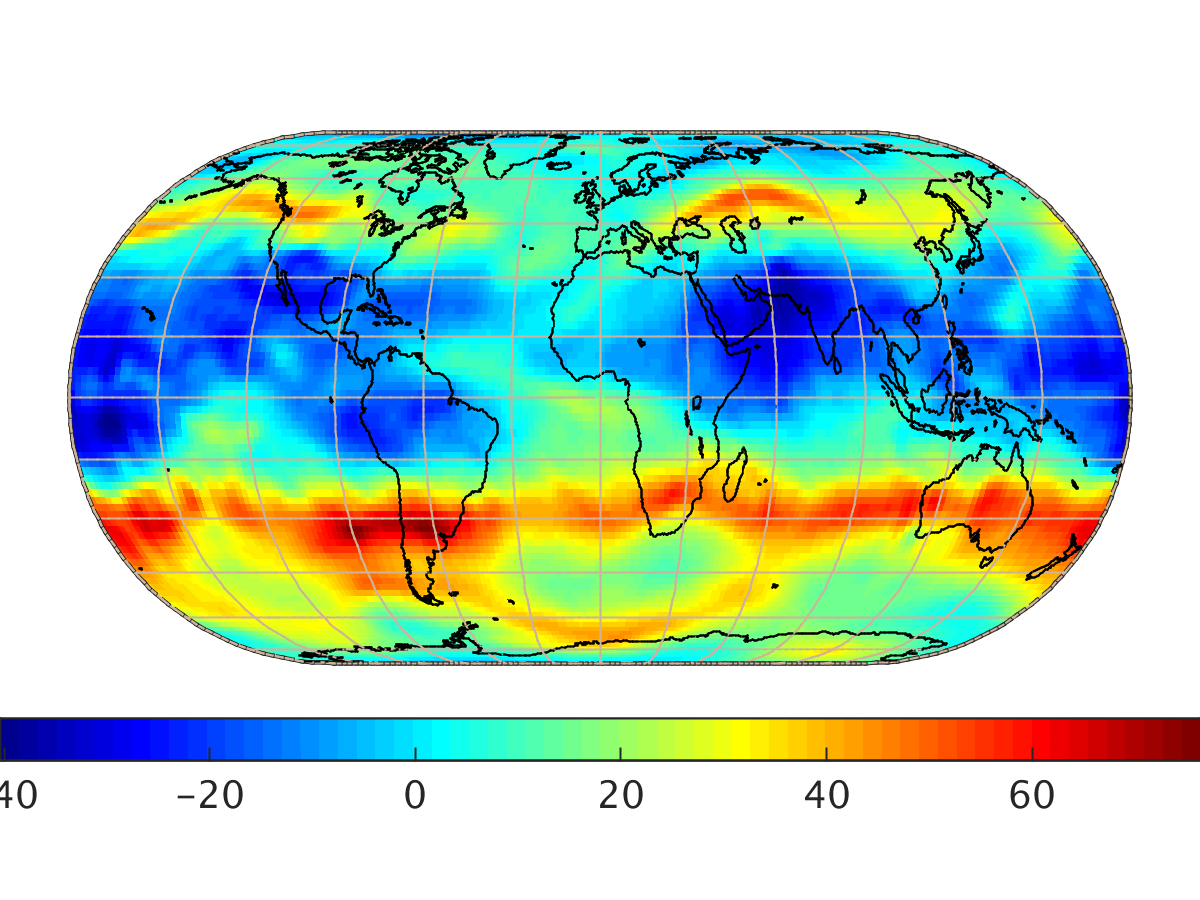}
\caption{Reference}
\end{subfigure}%
\begin{subfigure}{0.5\textwidth}
\centering
\includegraphics[width=0.9\textwidth,height=0.8\textwidth]{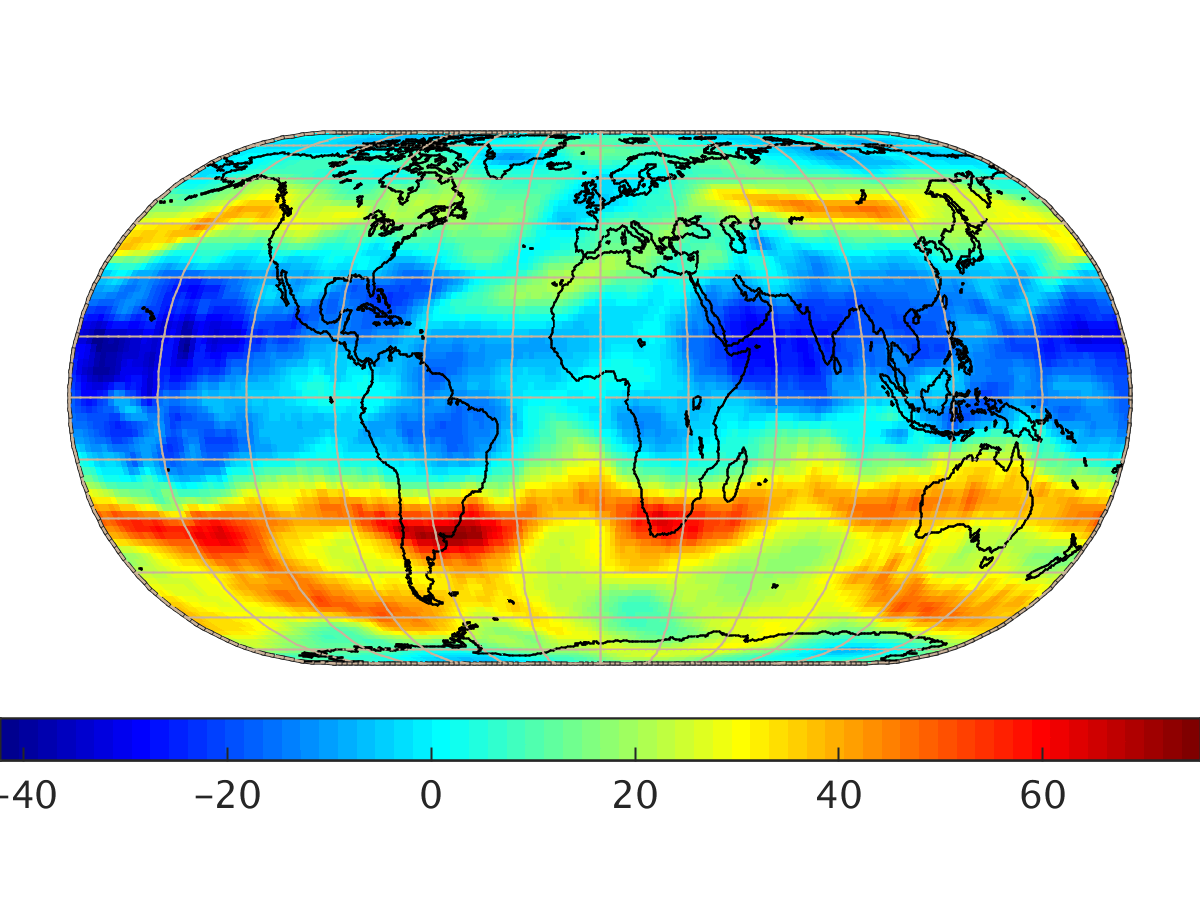}
\caption{Background}
\end{subfigure}

\begin{subfigure}{0.5\textwidth}
\centering
\includegraphics[width=0.9\textwidth,height=0.8\textwidth]{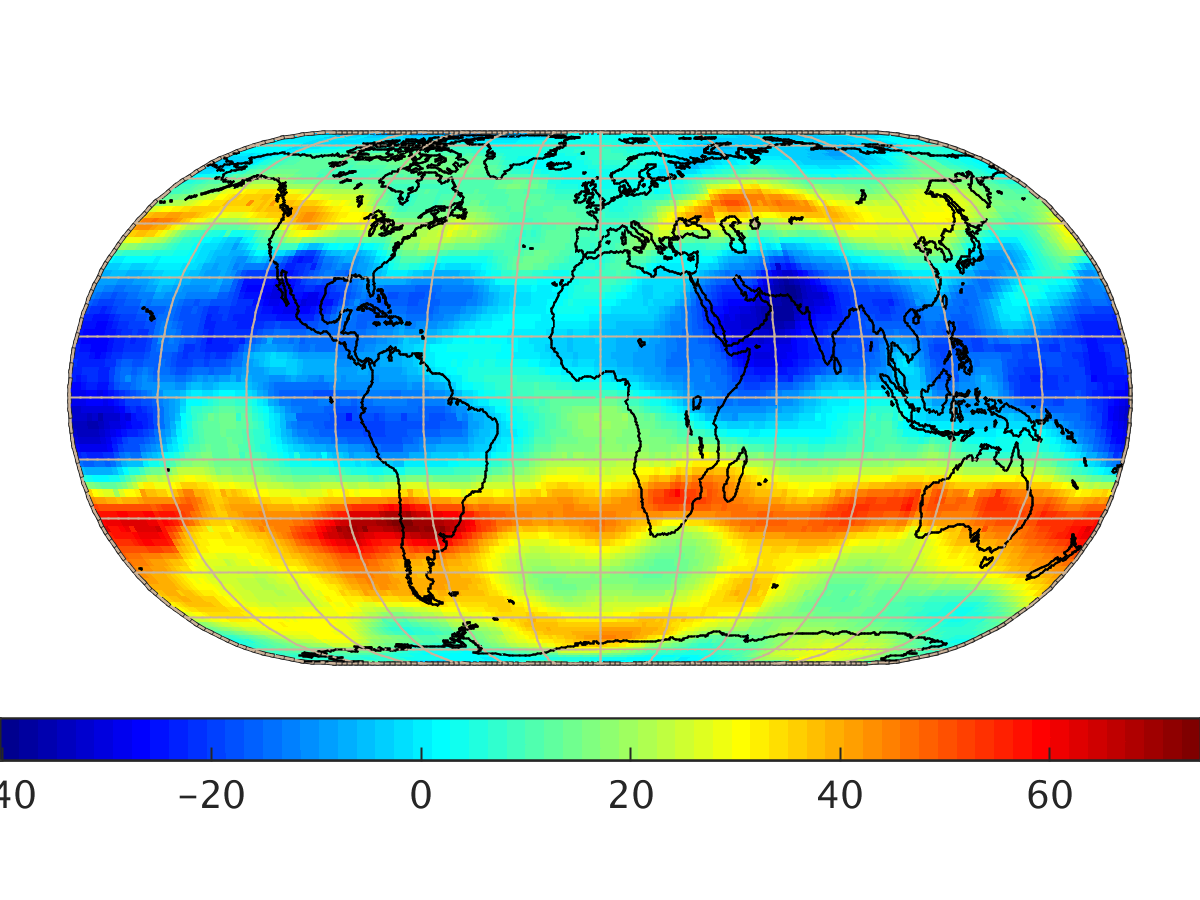}
\caption{EnKF-MC}
\end{subfigure}%
\begin{subfigure}{0.5\textwidth}
\centering
\includegraphics[width=0.9\textwidth,height=0.8\textwidth]{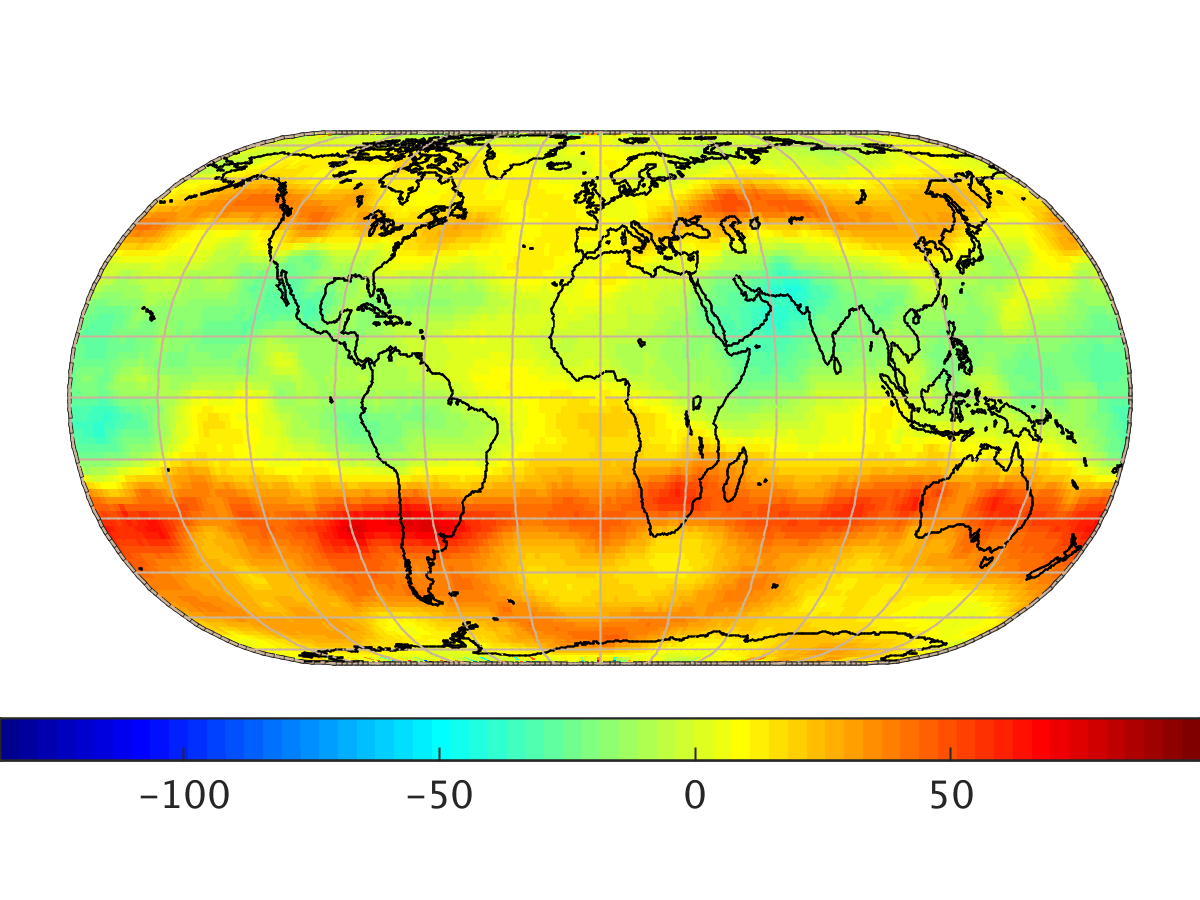}
\caption{LETKF}
\end{subfigure}
\caption{Snapshots of the reference solution, background state, and analysis fields from the EnKF-MC and LETKF for the second layer of the zonal wind component ($u$).}
\label{fig:exp-snapshot-zonal-wind-component}
\end{figure}

\subsection{Statistics of the ensemble}

In this section, we briefly discuss the spread of the ensemble making use of rank histograms. Of course, we do not claim this to be a verification procedure but, it provides useful insights about the dispersion of the members and the level of uncertainty about the ensemble mean.  The plots are based on the 5-th numerical layer of the atmosphere. We collect information across all model variables and the plots are shown in figures \ref{fig:bin-sph}, \ref{fig:bin-tem}, \ref{fig:bin-uwc}, and \ref{fig:bin-vwc}. Based on the results, the proposed implementation seems to be lesser sensitive to the intrinsic need of inflation than the LETKF formulation. For instance, after the assimilation, the ensemble members from the EnKF-MC are spread almost uniformly across different observation times. On the other hand, the spread in the context of the LETKF is impacted by the constant inflation factor used during the experiments (1.04) In practice, the inflation factor is set up according to historical information and/or heuristically with regard to some properties of the dynamics of the numerical model. This implies that, the dispersion of the LETKF members after the analysis will rely in how-well we estimate the optimal inflation factor for such filter. In operational data assimilation, an answer to this question can be hard to find. We think that inflation methodologies such as adaptive inflation can lead to better spread of the ensemble members in the context of the LETKF. For the proposed method, based on the experimental results, such methodology is not needed.
\begin{figure}[H]
\centering
\begin{subfigure}{0.5\textwidth}
\centering
\includegraphics[width=1\textwidth]{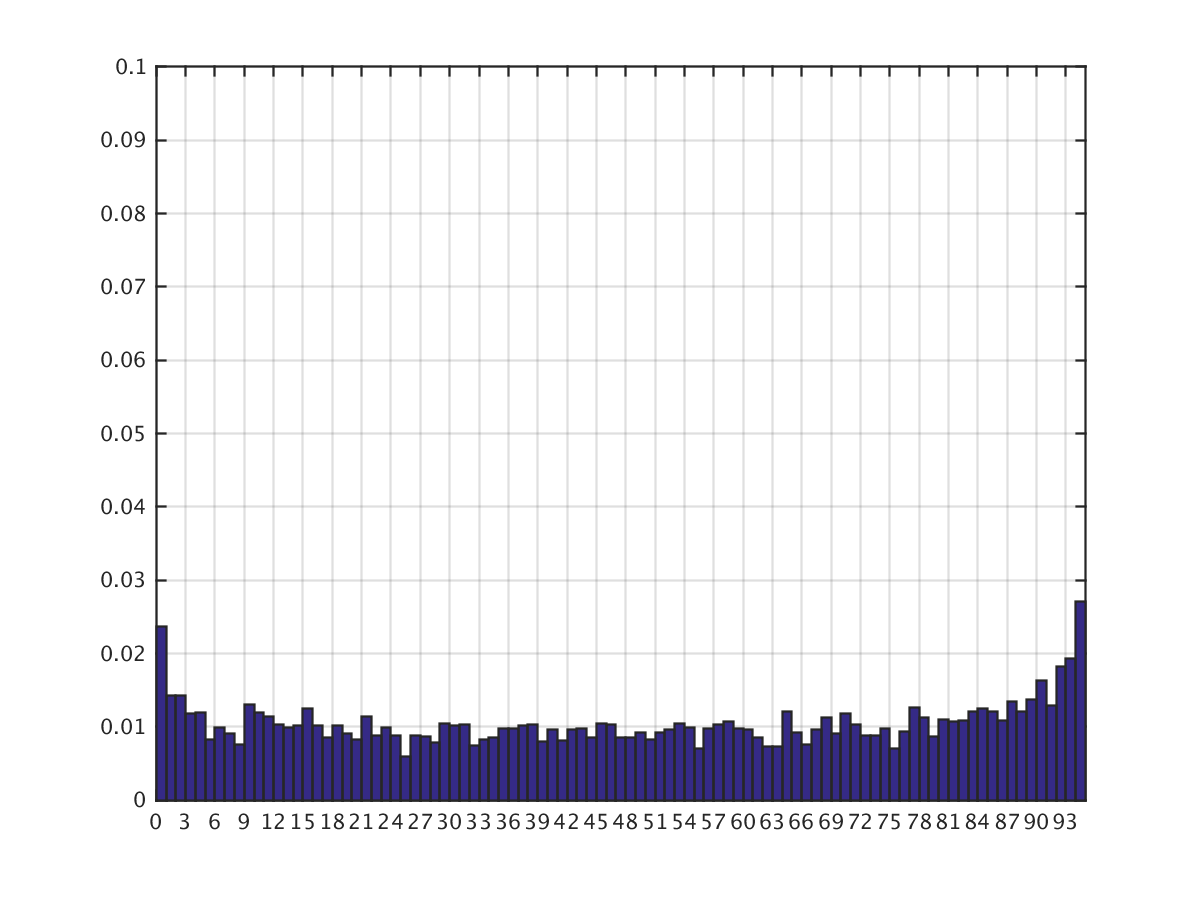}
\caption{EnKF-MC}
\end{subfigure}%
\begin{subfigure}{0.5\textwidth}
\centering
\includegraphics[width=1\textwidth]{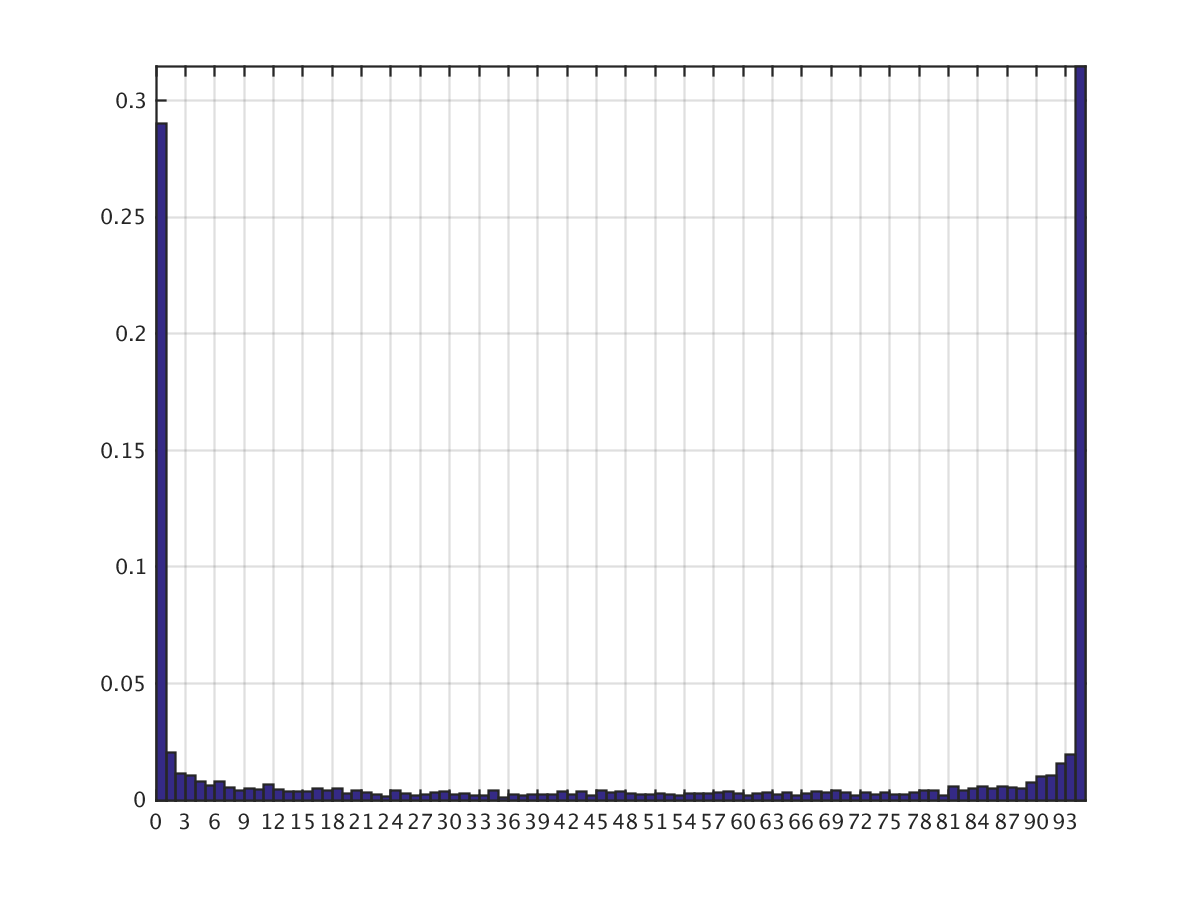}
\caption{LETKF}
\end{subfigure}%
\caption{Rank-histograms for the Specific Humidity model variable. The information is collected from the 5-th model layer.}
\label{fig:bin-sph}
\end{figure}
\begin{figure}[H]
\centering
\begin{subfigure}{0.5\textwidth}
\centering
\includegraphics[width=1\textwidth]{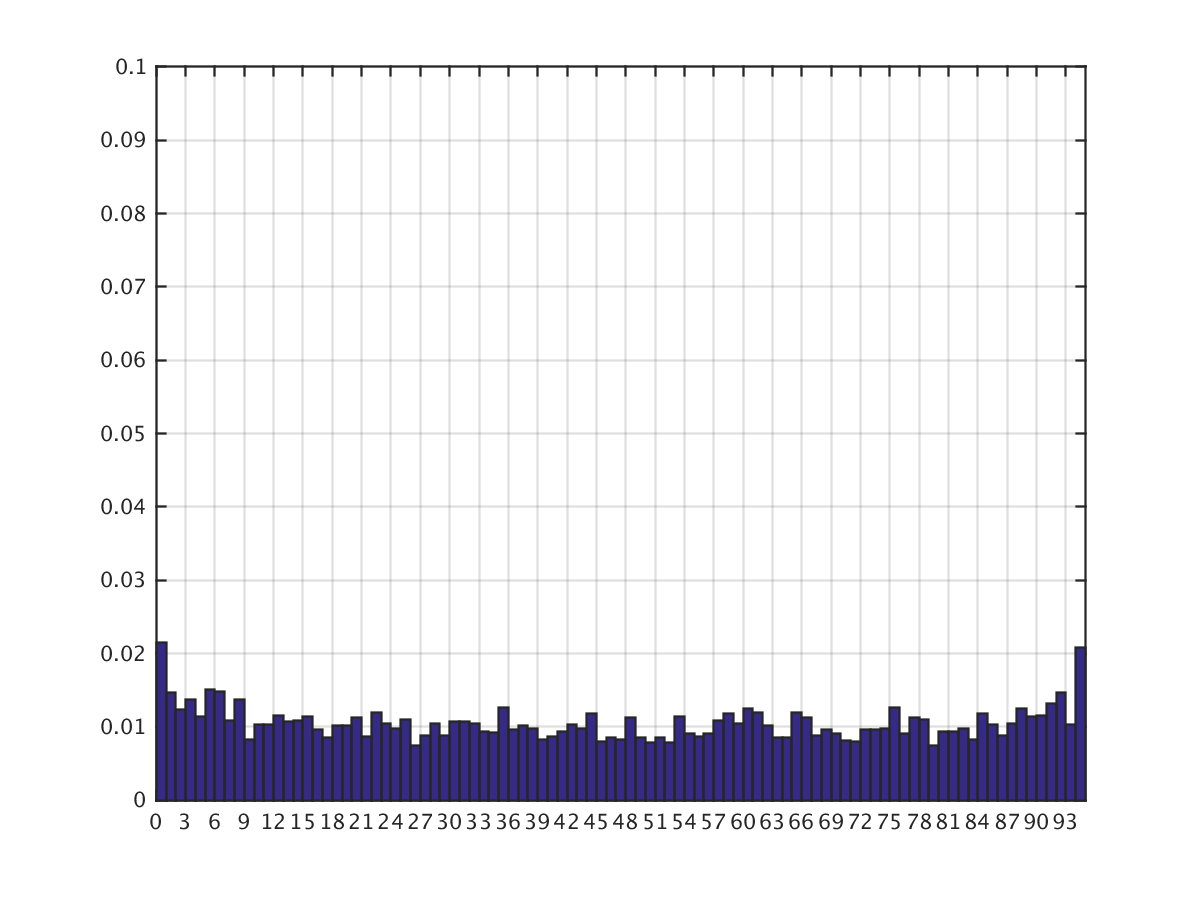}
\caption{EnKF-MC}
\end{subfigure}%
\begin{subfigure}{0.5\textwidth}
\centering
\includegraphics[width=1\textwidth]{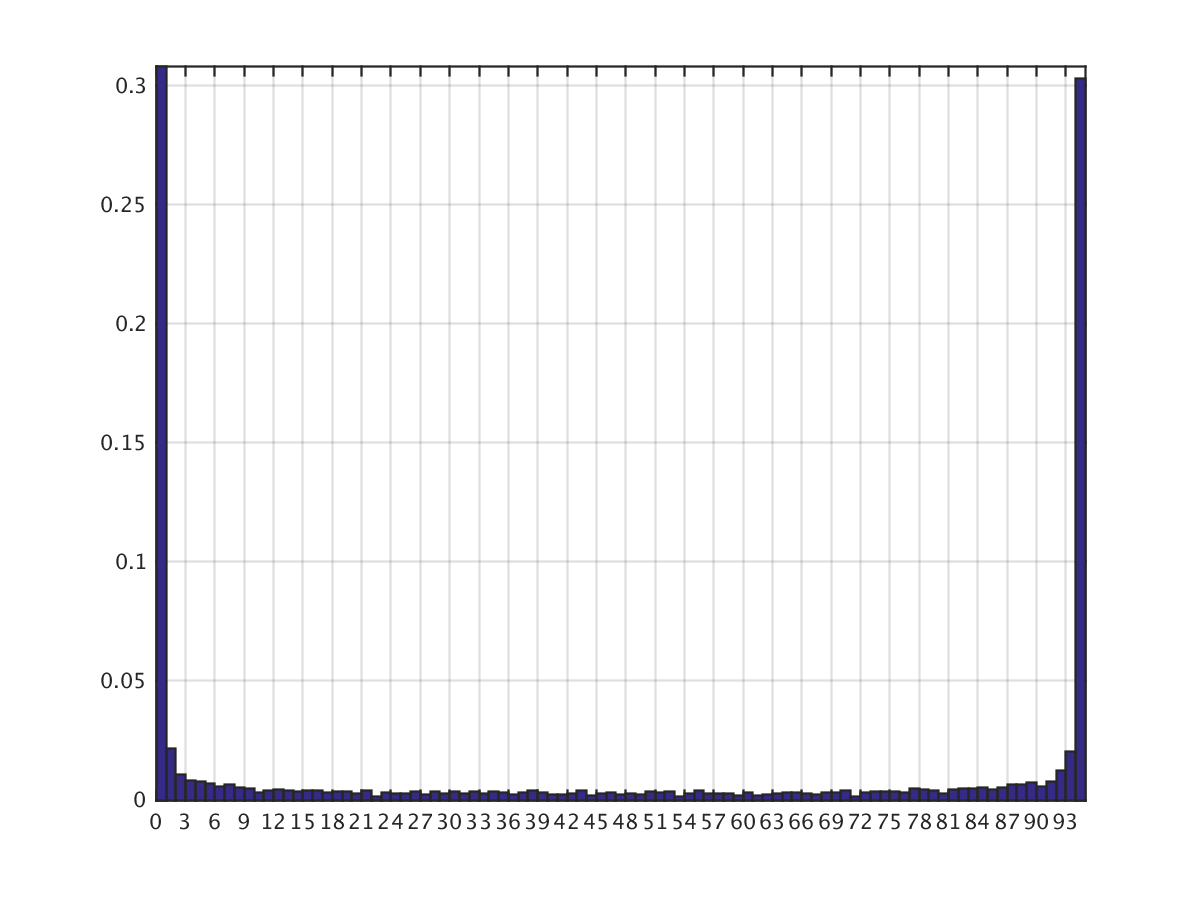}
\caption{LETKF}
\end{subfigure}%
\caption{Rank-histograms for the Zonal Wind Component model variable. The information is collected from the 5-th model layer.}
\label{fig:bin-uwc}
\end{figure}
\begin{figure}[H]
\centering
\begin{subfigure}{0.5\textwidth}
\centering
\includegraphics[width=1\textwidth]{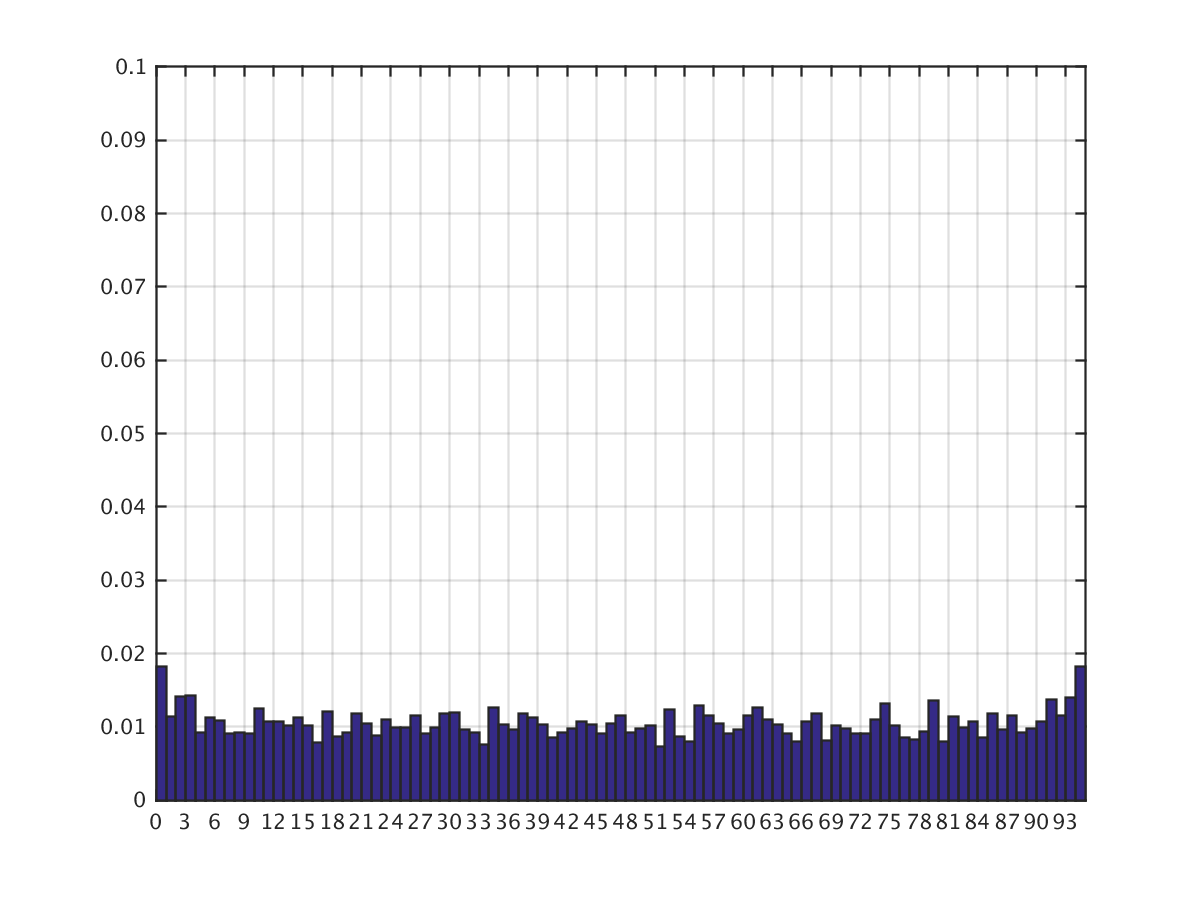}
\caption{EnKF-MC}
\end{subfigure}%
\begin{subfigure}{0.5\textwidth}
\centering
\includegraphics[width=1\textwidth]{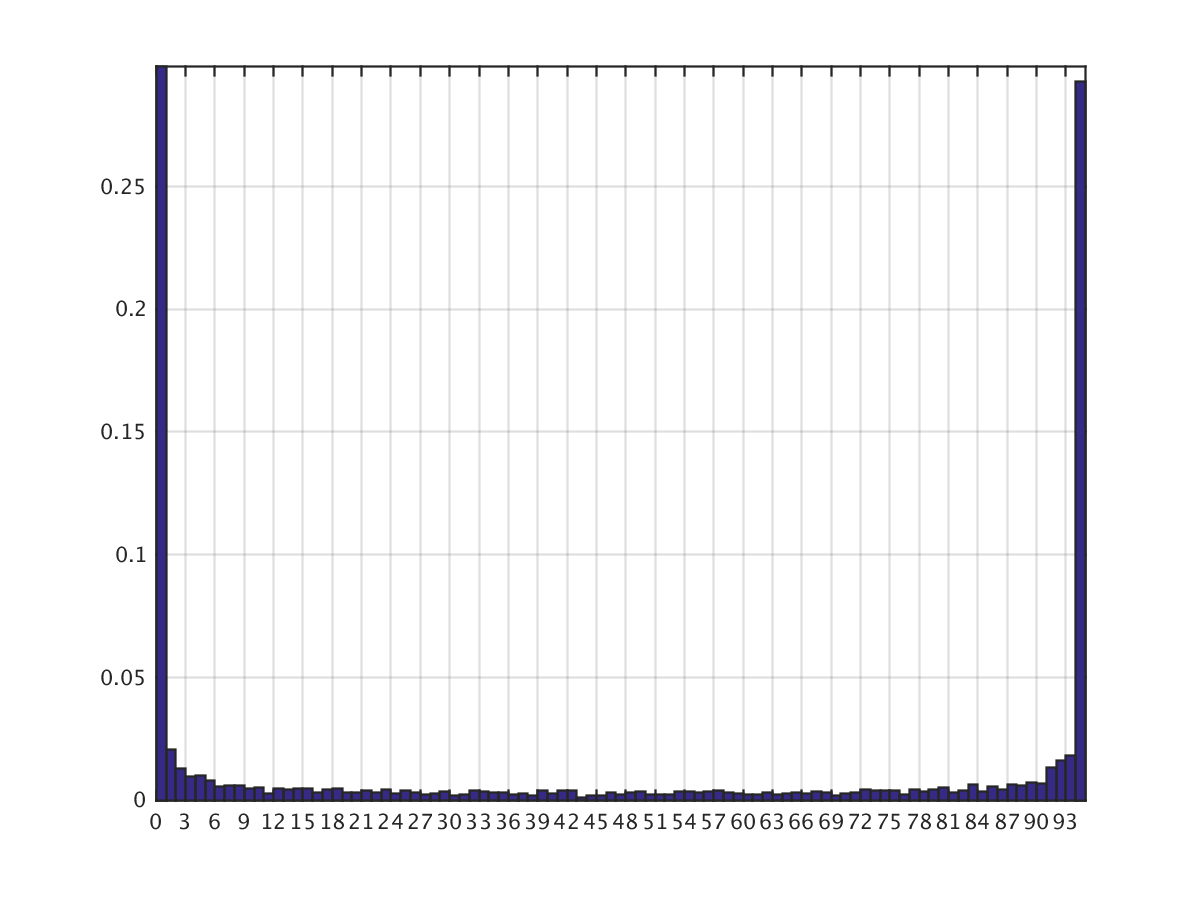}
\caption{LETKF}
\end{subfigure}%
\caption{Rank-histograms for the Meridional Wind Component model variable. The information is collected from the 5-th model layer.}
\label{fig:bin-vwc}
\end{figure}
\begin{figure}[H]
\centering
\begin{subfigure}{0.5\textwidth}
\centering
\includegraphics[width=1\textwidth]{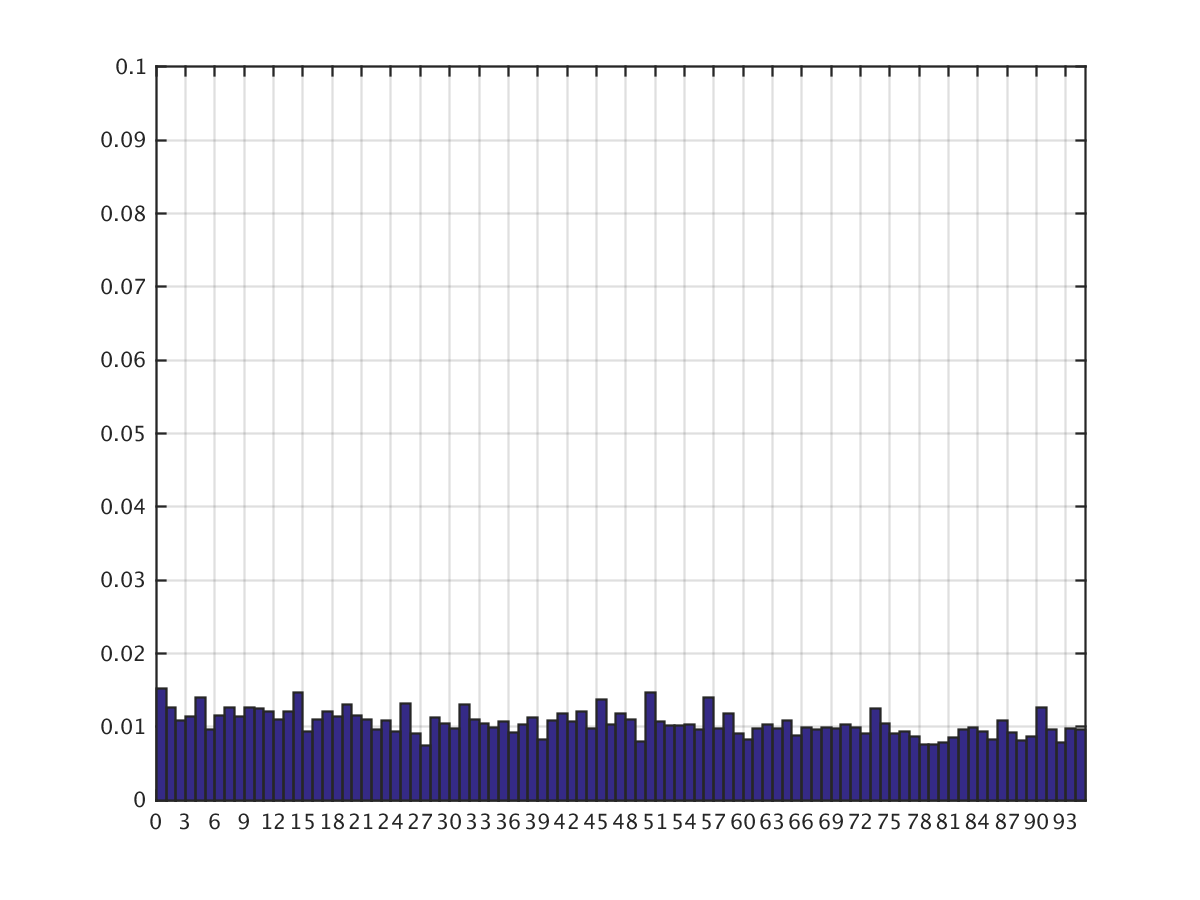}
\caption{EnKF-MC}
\end{subfigure}%
\begin{subfigure}{0.5\textwidth}
\centering
\includegraphics[width=1\textwidth]{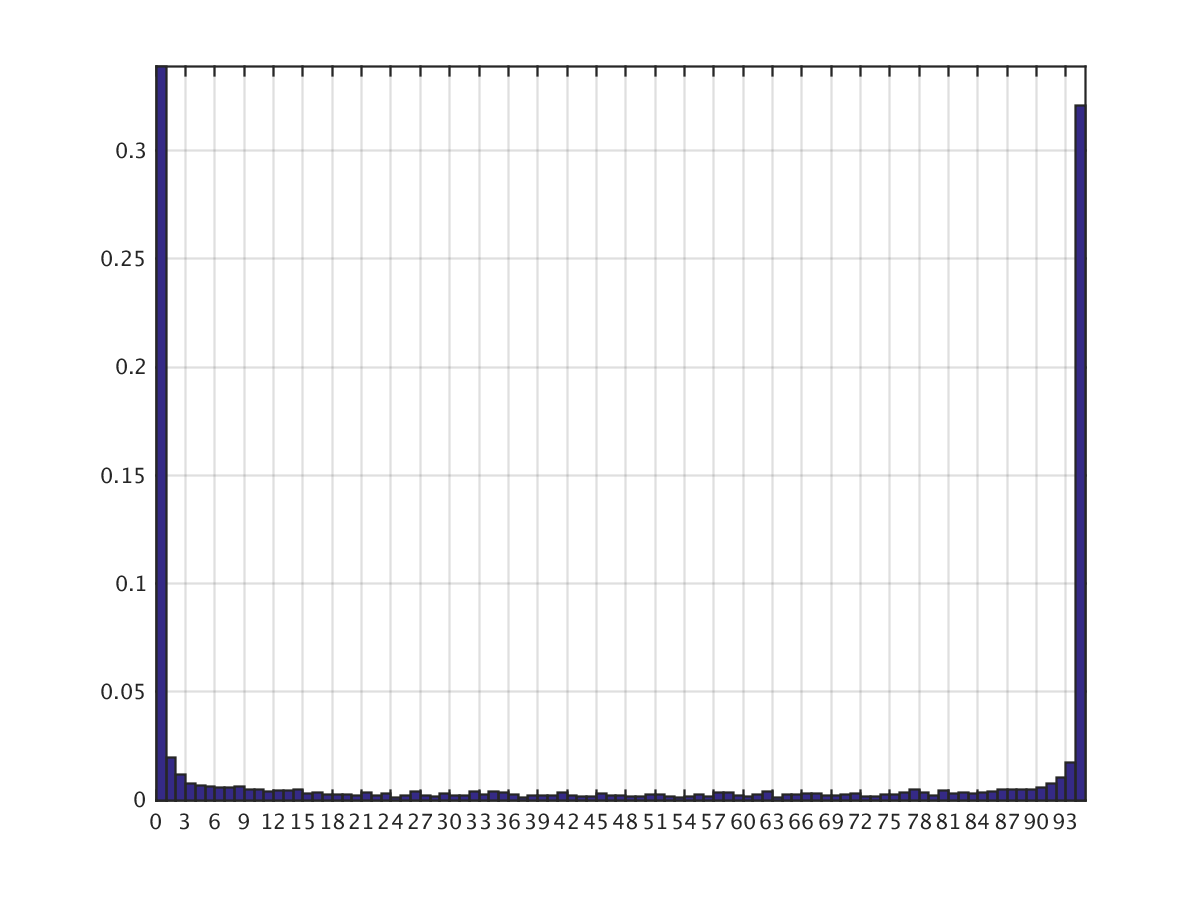}
\caption{LETKF}
\end{subfigure}%
\caption{Rank-histograms for the Temperature model variable. The information is collected from the 5-th model layer.}
\label{fig:bin-tem}
\end{figure}

%%%%%%%%%%%%%%%%%%%%%%%%%%%%%%%
\subsection{The impact of SVD truncation threshold}
\label{sec:future-work}
%%%%%%%%%%%%%%%%%%%%%%%%%%%%%%%

An important question arising from this research is the number of singular values/vectors to be used in \eqref{eq:EnKF-MC-truncated-SVD}. To study this question we use the same experimental setting and the sparse observational network where only $4\%$ of the model components are observed. We apply EnKF-MC algorithm and truncate the summation \eqref{eq:EnKF-MC-truncated-SVD} based on different thresholds $\sigma_r$. 

The results are reported in Figure \ref{fig:exp-RMSE-different-thresholds}. Different thresholds lead to different levels of accuracy for the EnKF-MC analyses. There is no unique value of $\sigma_r$ that provides the best ensemble trajectory in general; for instance, the best performance at the beginning of the assimilation window is obtained for $\sigma_r = 0.05$, but, at the end the best solution is obtained with $\sigma_r = 0.2$. This indicates that the results can be improved when $\sigma_r$ is dynamically and optimally chosen. Note that, on average, the results obtained by the EnKF-MC with $\sigma_r \in \lle 0.15,\, 0.20,\, 0.25 \rle$ are much better than those when $\sigma_r = 0.10$ (and therefore much better than the results obtained by the LETKF). In Figure \ref{fig:exp-RMSE-snapshots-thresholds} snapshots of the specific humidity for different $\sigma_r$ are shown. It can be seen that the spurious errors can be quickly decreased when $\sigma_r$ is chosen accordingly. 

In order to understand the optimal truncation level note that the summation \eqref{eq:EnKF-MC-truncated-SVD} can be written as follows:
\begin{eqnarray}
\label{eq:alpha_j}
\be_{[i]} &=& \sum_{j=1}^{\Nens} \alpha_j \cdot \ub^{\Z_{[i]}}_j,\,  \\
\nonumber
\alpha_j &=& \frac{1}{\tau_j} \cdot {\vv^{\Z_{[i]}}_j}^T \cdot \x_{[i]} = \frac{1}{\tau_j} \cdot {\vv^{\Z_{[i]}}_j}^T \cdot \lb \widetilde{\x}_{[i]} + \errR_{[i]} \rb  \\
\nonumber
&=& \underbrace{  \frac{1}{\tau_j}  \cdot {\vv^{\Z_{[i]}}_j}^T \cdot \widetilde{\x}_{[i]} }_{\text{Uncorrupted data}} + \underbrace{\frac{1}{\tau_j} \cdot {\vv^{\Z_{[i]}}_j}^T \cdot \errR_{[i]}}_{\text{Error}}
\end{eqnarray}
where $\widetilde{\x}_{[i]}$ is the perfect data ($\x_{[i]} = \widetilde{\x}_{[i]}+\errR_{[i]}$). The components with small singular values $\tau_j$ will amplify the error more. The threshold should be large enough to include useful information from $\widetilde{\x}_{[i]}$, but small enough in order to prune out the components with large error amplification.
%Consider $\Nens$ samples $\lle \errbac^{[j]} \rle_{j=1}^{\Nens} \in \Re^{\Nstate \times 1}$ of the background error distribution \eqref{eq:background-errors}; the $j$-th component $\theta_j^{[i]}$ of the error $\errR_{[i]} = \lb  \theta_1^{[i]},\, \theta_2^{[i]},\, \ldots,\, \theta_{\Nens}^{[i]}\rb^T \in \Re^{\Nens \times 1}$, for $1 \le j \le \Nens$, corresponds to the $i$-th component of $\errbac^{[j]}$. 
%A prior information of $\errR_{[i]}$ can be difficult to obtain under realistic model scenarios but, we can still making use of machine learning in order to adaptively choose the ``best value'' of $\sigma_r$ according to historical information, information based on the radius of influence $\ra$, the variances of $\x_{[i]}$ and the ensemble size $\Nens$. 
We expect that model components with large variances will need more basis vectors from \eqref{eq:EnKF-MC-truncated-SVD} than those with lesser variance. An upper bound for the number of basis vectors (and therefore the threshold $\sigma_r$) can be obtained by inspection of the values $\alpha_j$ in \eqref{eq:alpha_j}. Figure \ref{fig:exp-random-noise-effect} shows the weights $\alpha_j$ for different singular values for the 500-th model component of the SPEEDY model. The large zig-zag behaviors are evidence of error amplifications and therefore, we can truncate the summation \eqref{eq:alpha_j} before this pattern starts to take place in the values of $\alpha_j$.
\begin{figure}[H]
\centering
\begin{subfigure}{0.5\textwidth}
\centering
\includegraphics[width=0.9\textwidth,height=0.75\textwidth]{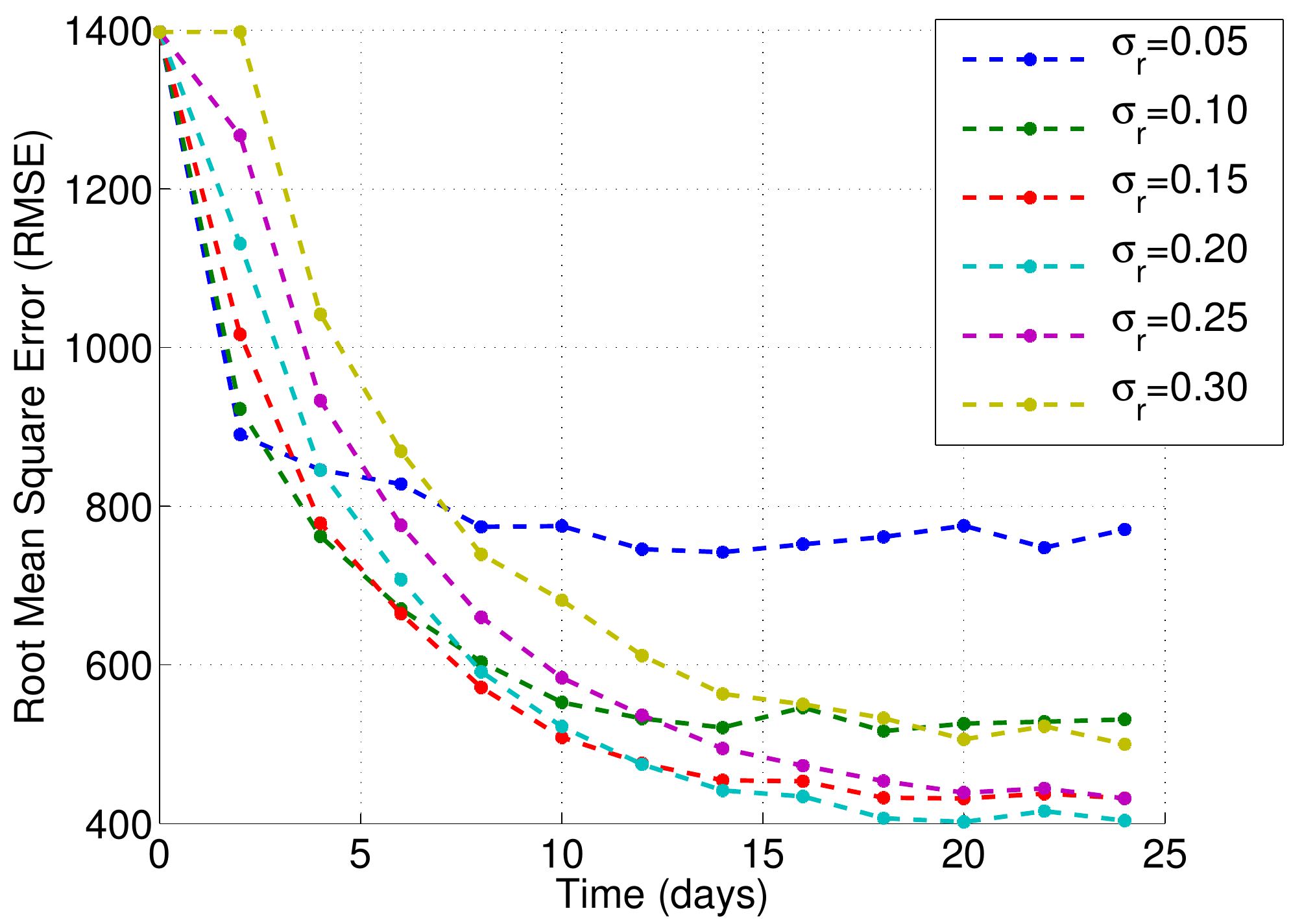}
\caption{Meridional wind component ($m/s$)}
\end{subfigure}%
\begin{subfigure}{0.5\textwidth}
\centering
\includegraphics[width=0.9\textwidth,height=0.75\textwidth]{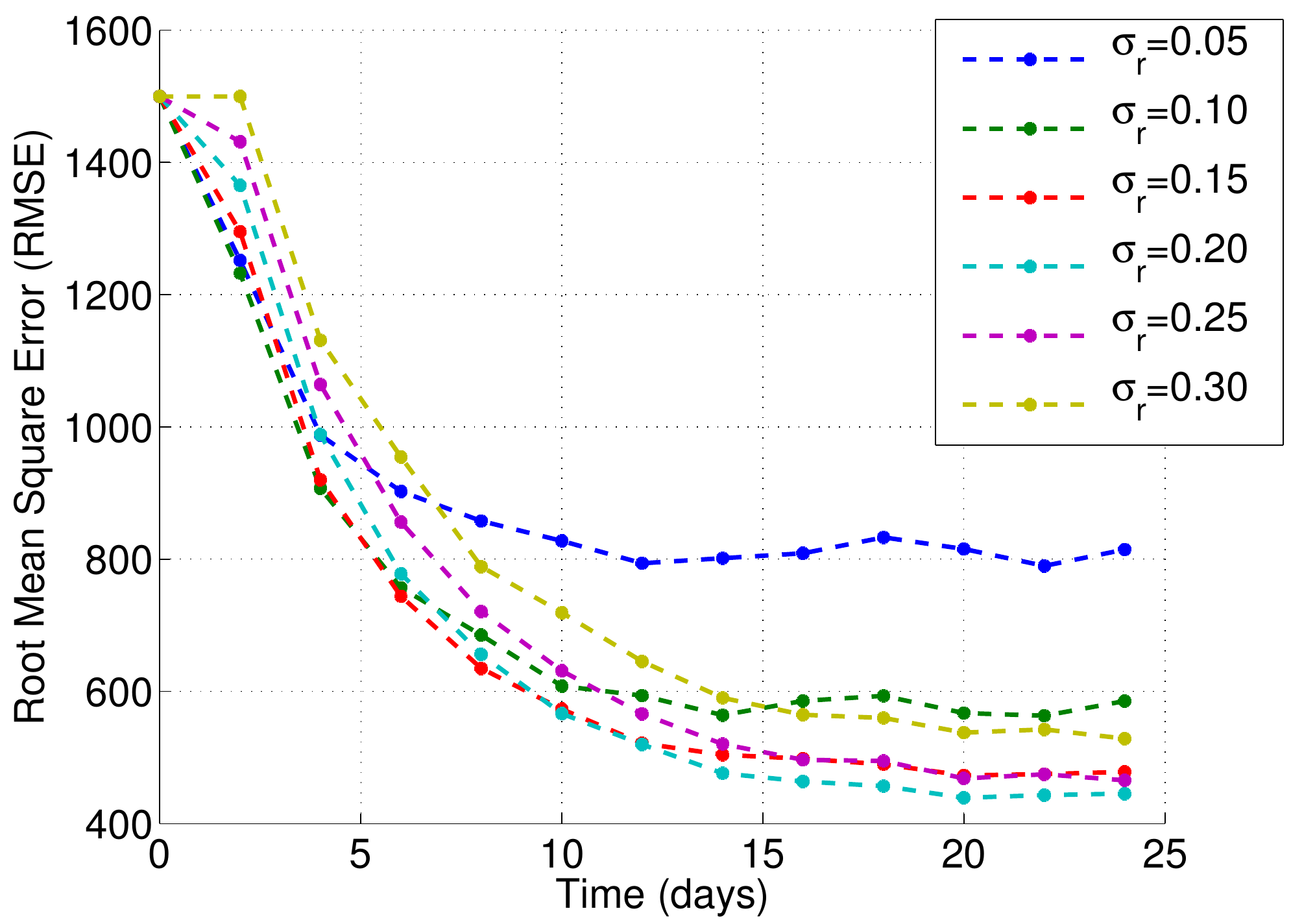}
\caption{Zonal wind component ($m/s$)}
\end{subfigure}

\begin{subfigure}{0.5\textwidth}
\centering
\includegraphics[width=0.9\textwidth,height=0.75\textwidth]{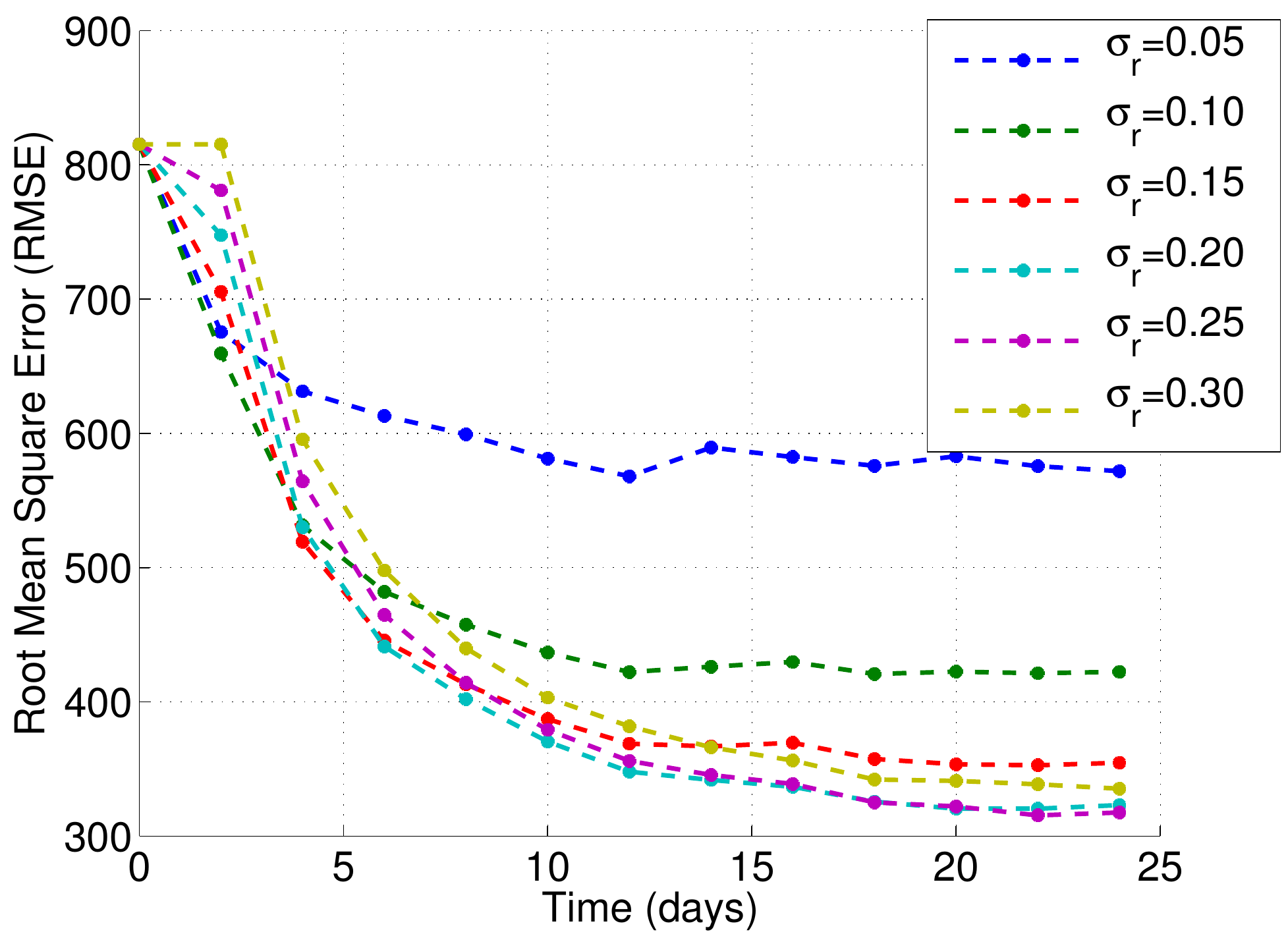}
\caption{Temperature ($K$)}
\end{subfigure}%
\begin{subfigure}{0.5\textwidth}
\centering
\includegraphics[width=0.9\textwidth,height=0.75\textwidth]{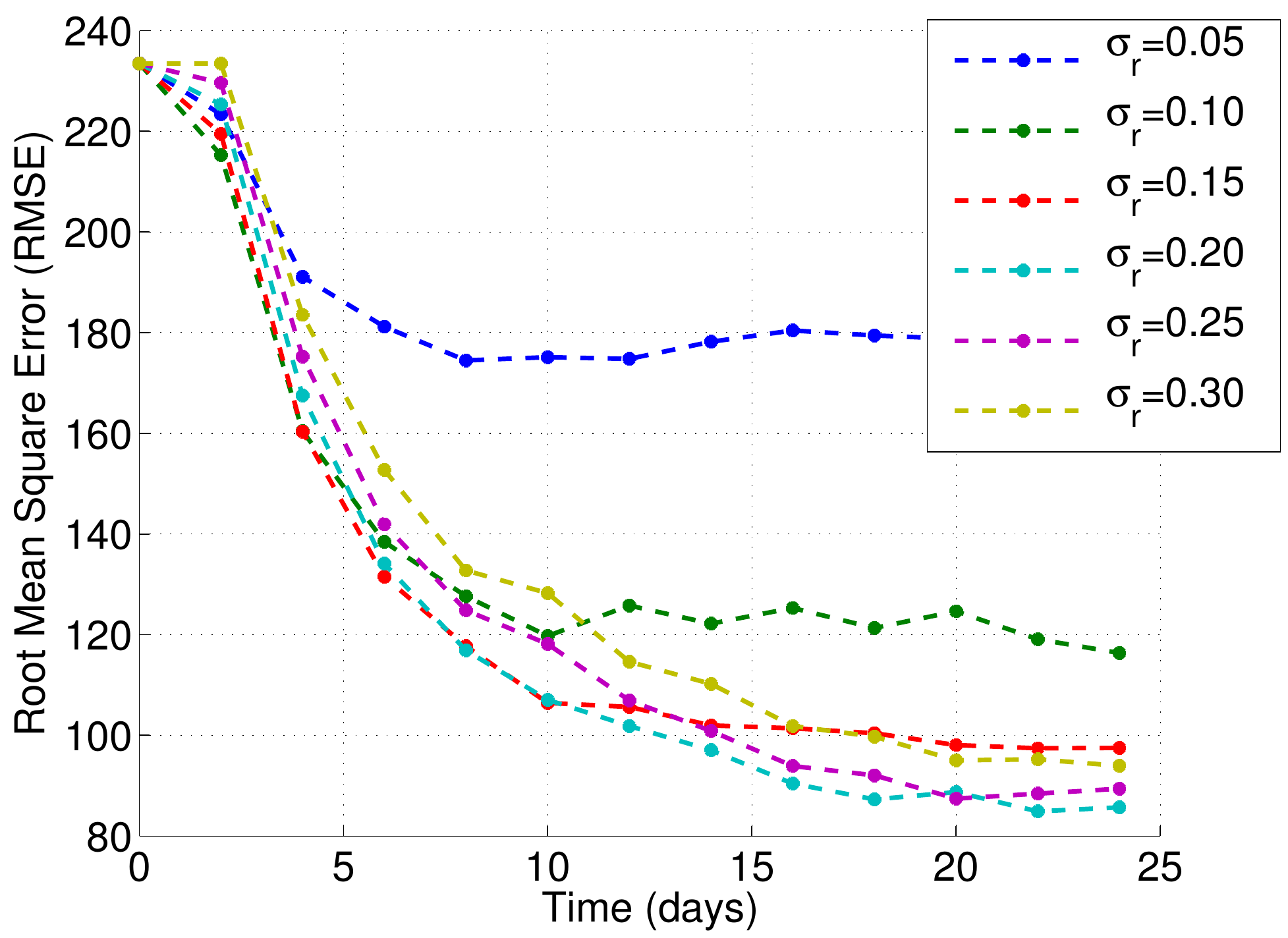}
\caption{Specific humidity ($g/kg$)}
\end{subfigure}
\caption{RMSE for the SPEEDY analyses obtained using different SVD truncation levels based on the $\sigma_r$ values.}
\label{fig:exp-RMSE-different-thresholds}
\end{figure}

\begin{figure}[H]
\centering
\begin{subfigure}{0.5\textwidth}
\centering
\includegraphics[width=0.9\textwidth,height=0.9\textwidth]{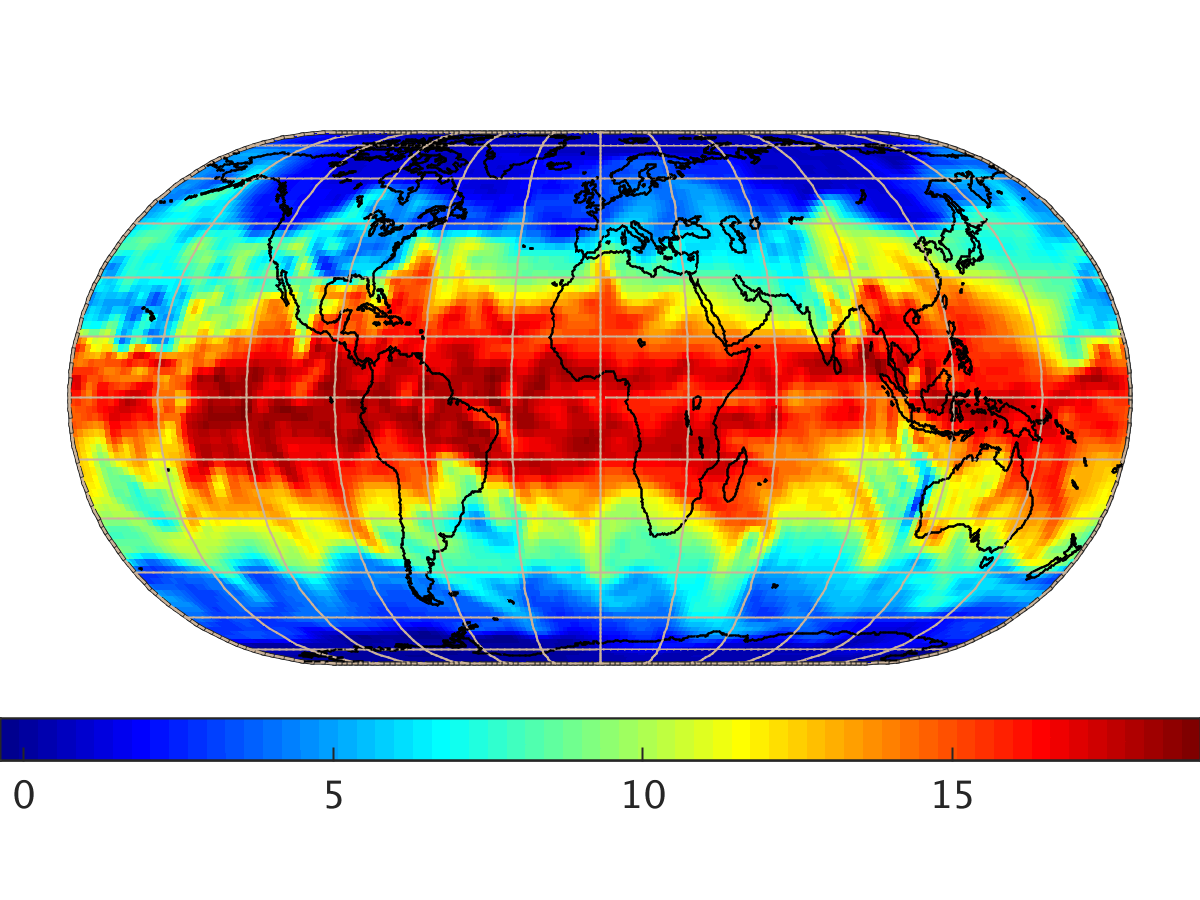}
\caption{Reference}
\end{subfigure}%
\begin{subfigure}{0.5\textwidth}
\centering
\includegraphics[width=0.9\textwidth,height=0.9\textwidth]{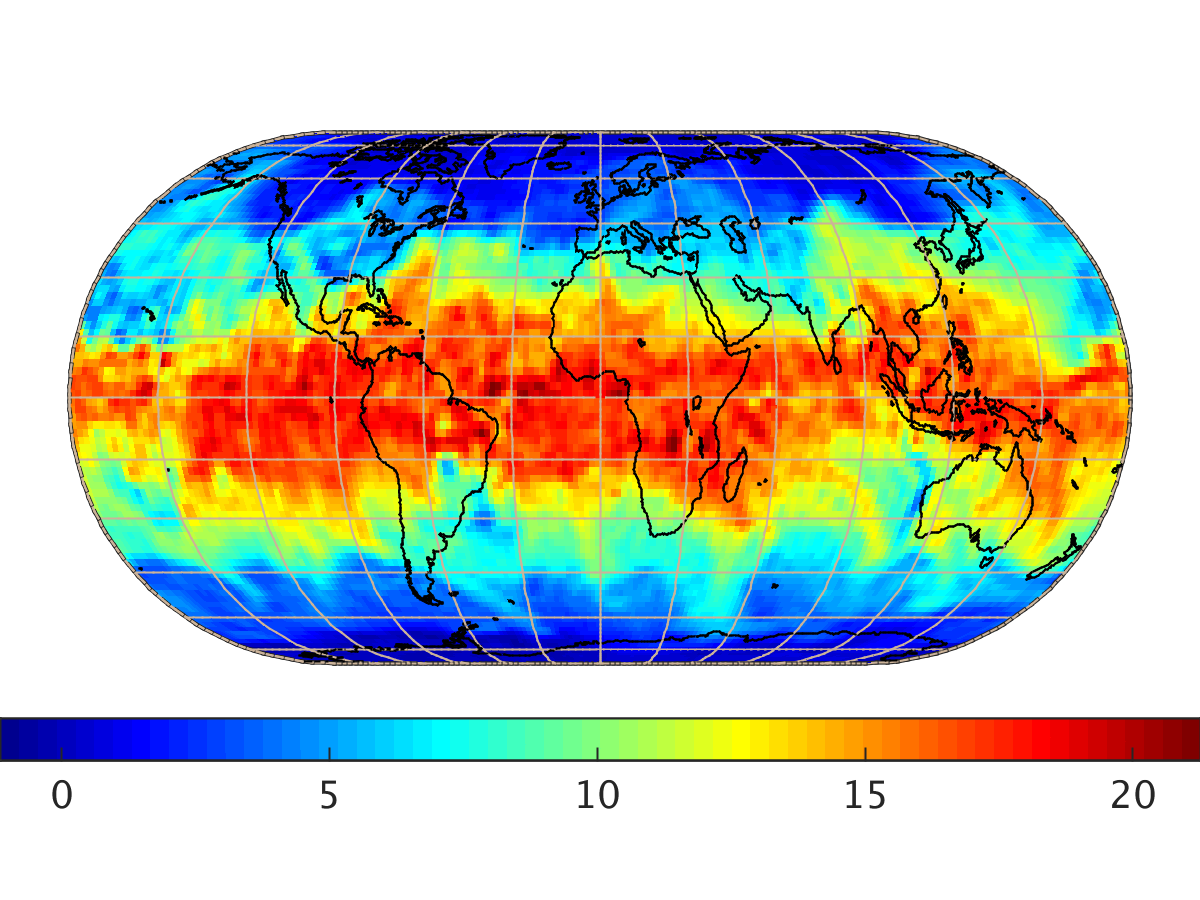}
\caption{$\sigma_r = 0.05$}
\end{subfigure}

\begin{subfigure}{0.5\textwidth}
\centering
\includegraphics[width=0.9\textwidth,height=0.9\textwidth]{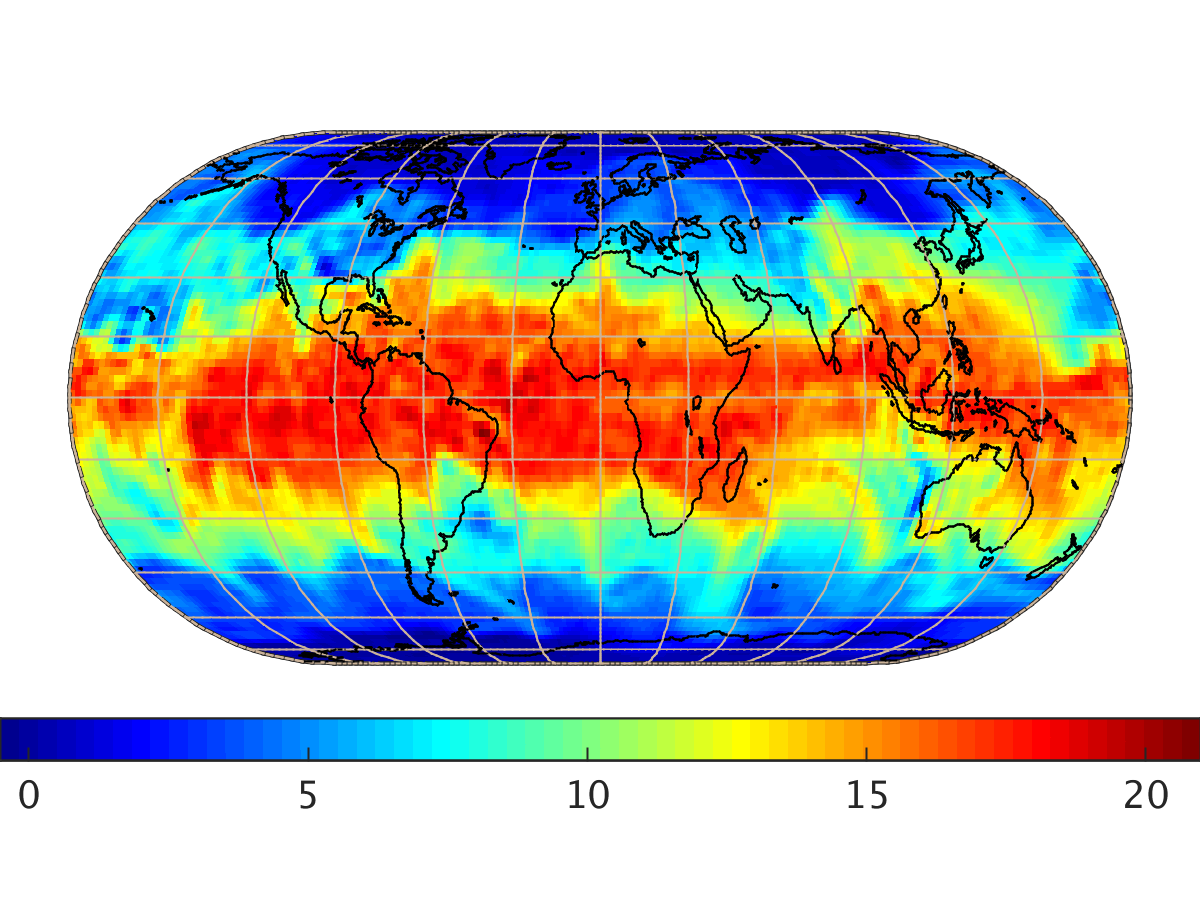}
\caption{$\sigma_r = 0.10$}
\end{subfigure}%
\begin{subfigure}{0.5\textwidth}
\centering
\includegraphics[width=0.9\textwidth,height=0.9\textwidth]{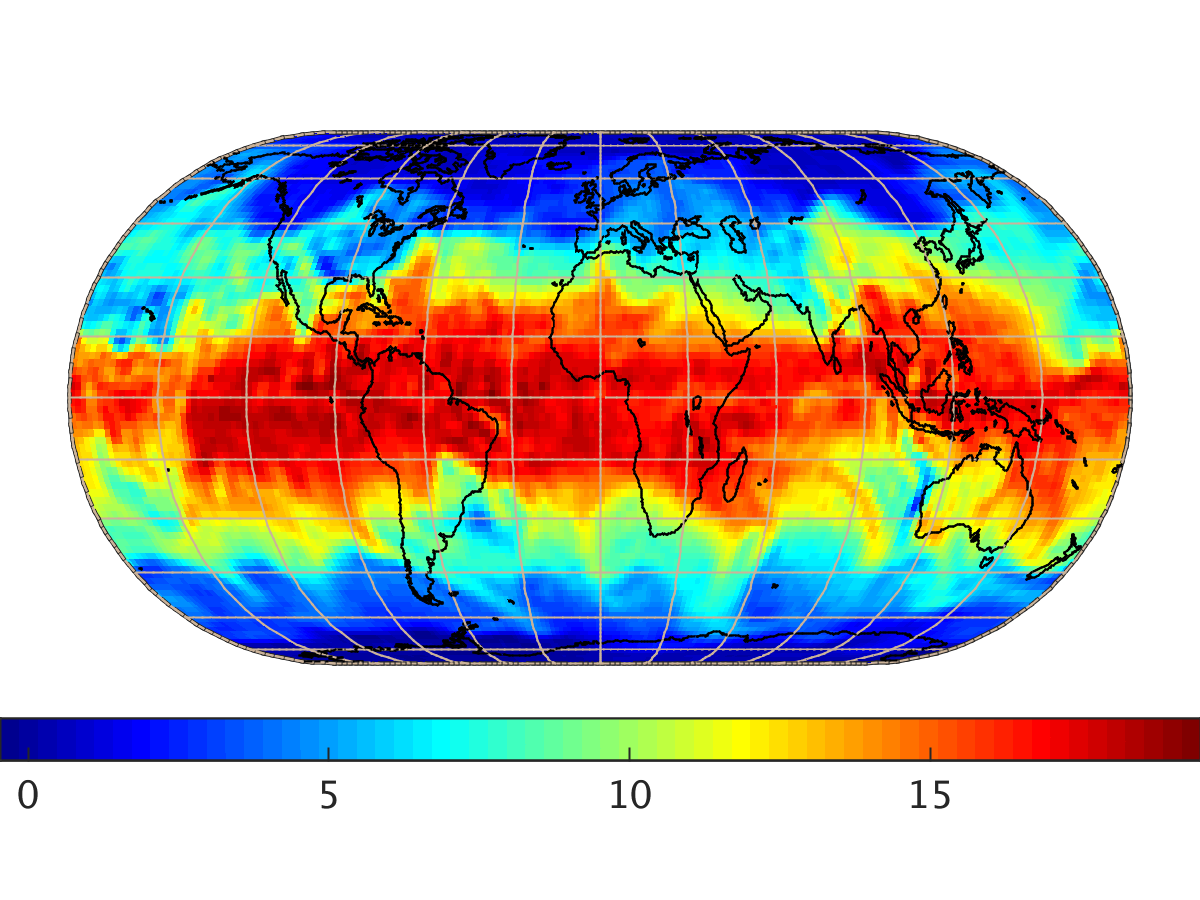}
\caption{$\sigma_r = 0.15$}
\end{subfigure}

\begin{subfigure}{0.5\textwidth}
\centering
\includegraphics[width=0.9\textwidth,height=0.9\textwidth]{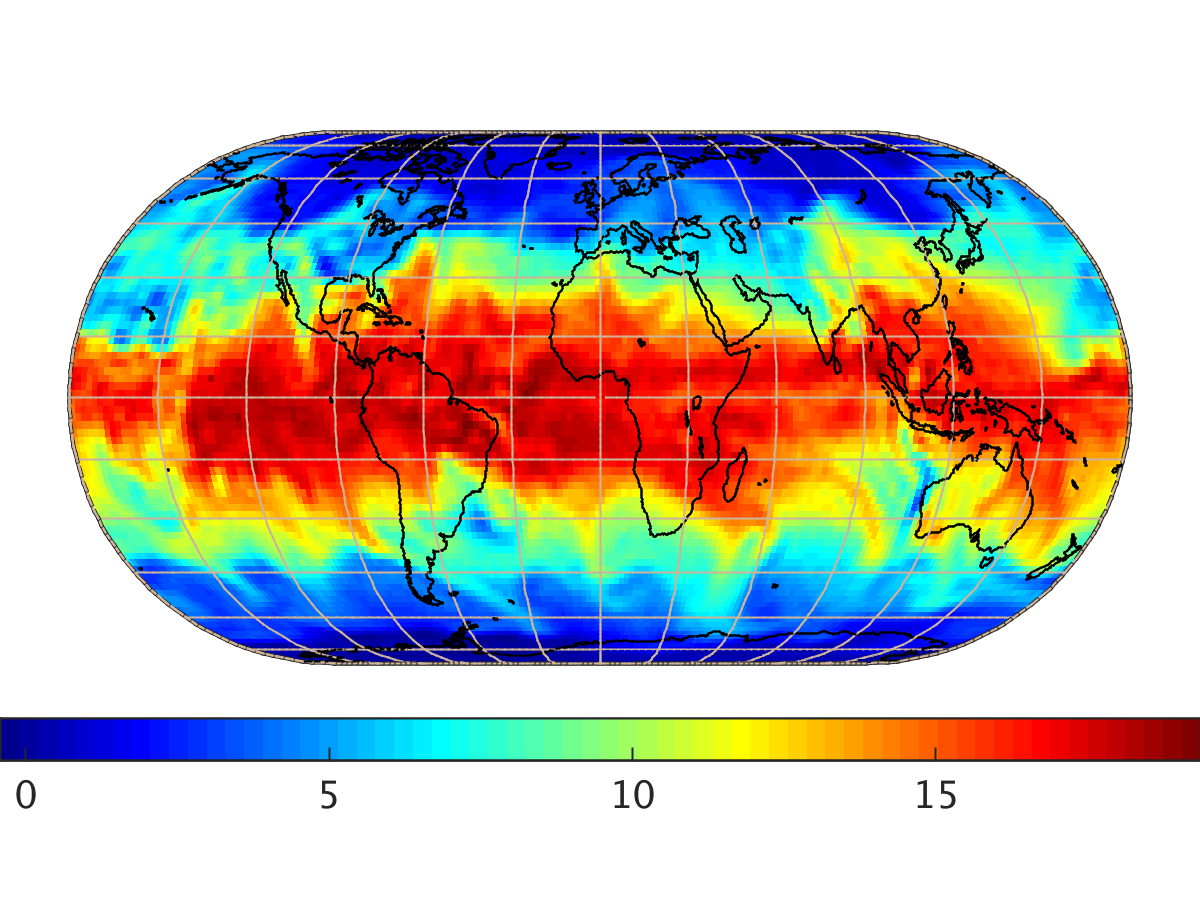}
\caption{$\sigma_r = 0.20$}
\end{subfigure}%
\begin{subfigure}{0.5\textwidth}
\centering
\includegraphics[width=0.9\textwidth,height=0.9\textwidth]{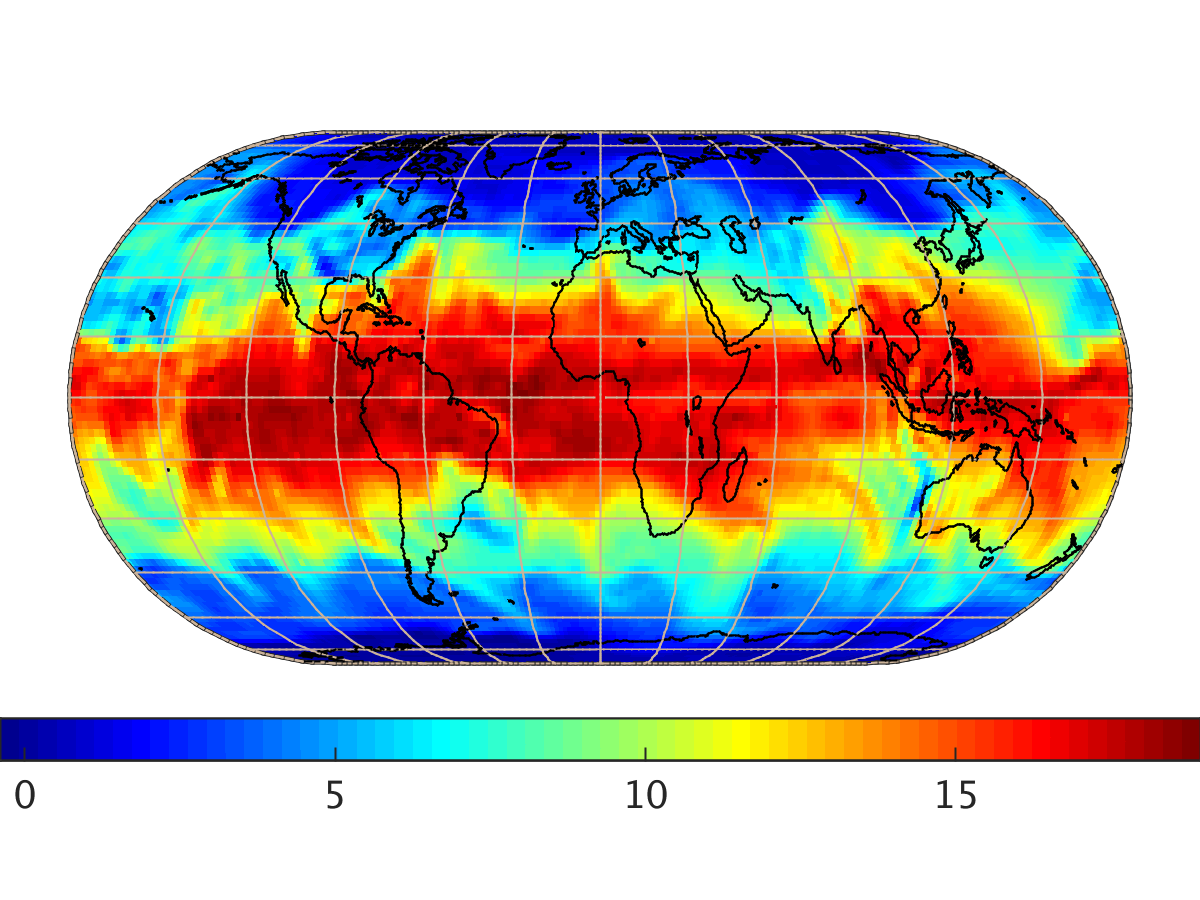}
\caption{$\sigma_r = 0.30$}
\end{subfigure}
\caption{Snapshots at the final assimilation time (day 22) of the EnKF-MC analysis making use of different thresholds $\sigma_r$ for $\ra = 5$ and $p = 4 \%$. }
\label{fig:exp-RMSE-snapshots-thresholds}
\end{figure}
\begin{figure}[H]
\centering
\begin{subfigure}{0.5\textwidth}
\centering
\includegraphics[width=0.9\textwidth,height=0.9\textwidth]{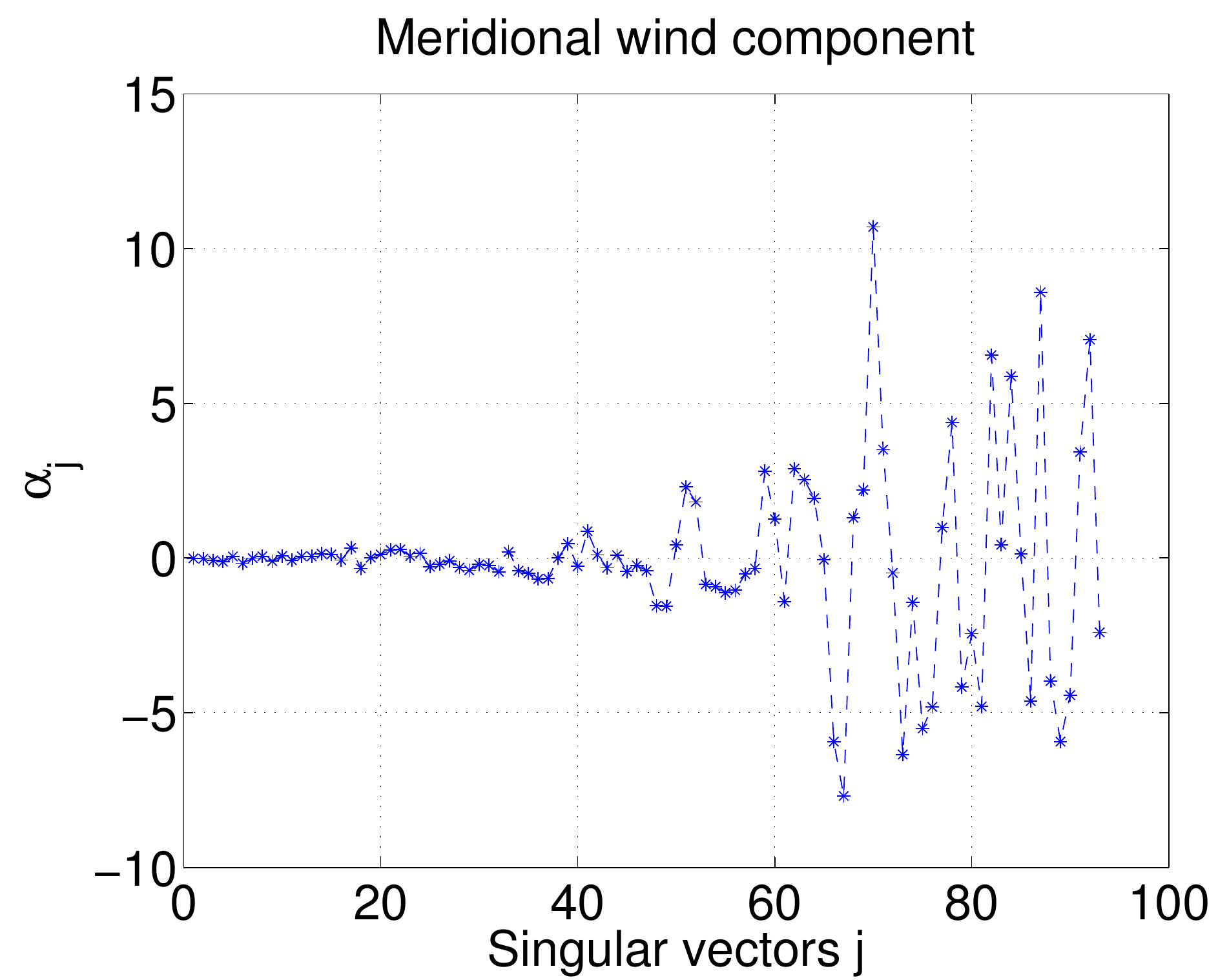}
\caption{Meridional wind component ($m/s$)}
\end{subfigure}%
\begin{subfigure}{0.5\textwidth}
\centering
\includegraphics[width=0.9\textwidth,height=0.9\textwidth]{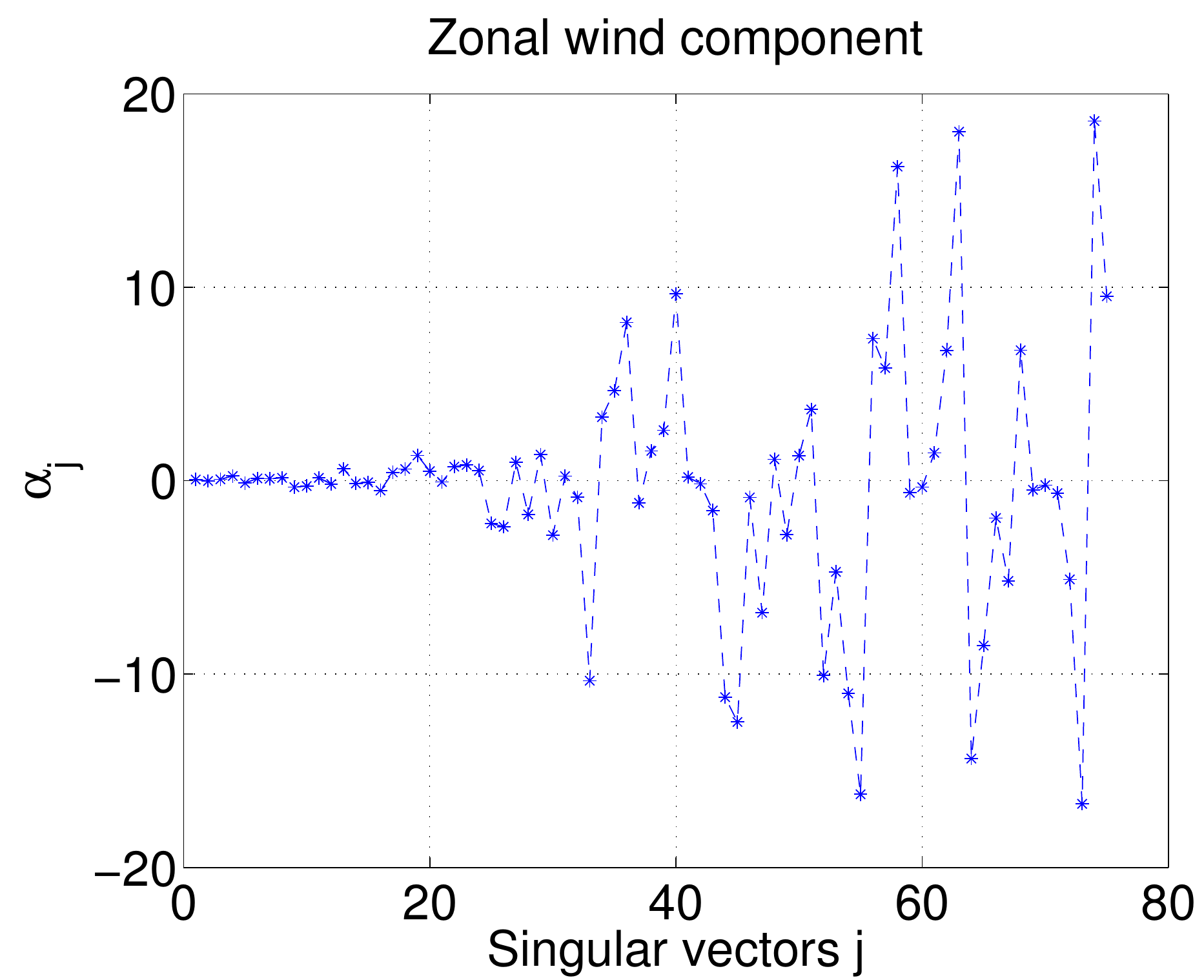}
\caption{Zonal wind component ($m/s$)}
\end{subfigure}

\begin{subfigure}{0.5\textwidth}
\centering
\includegraphics[width=0.9\textwidth,height=0.9\textwidth]{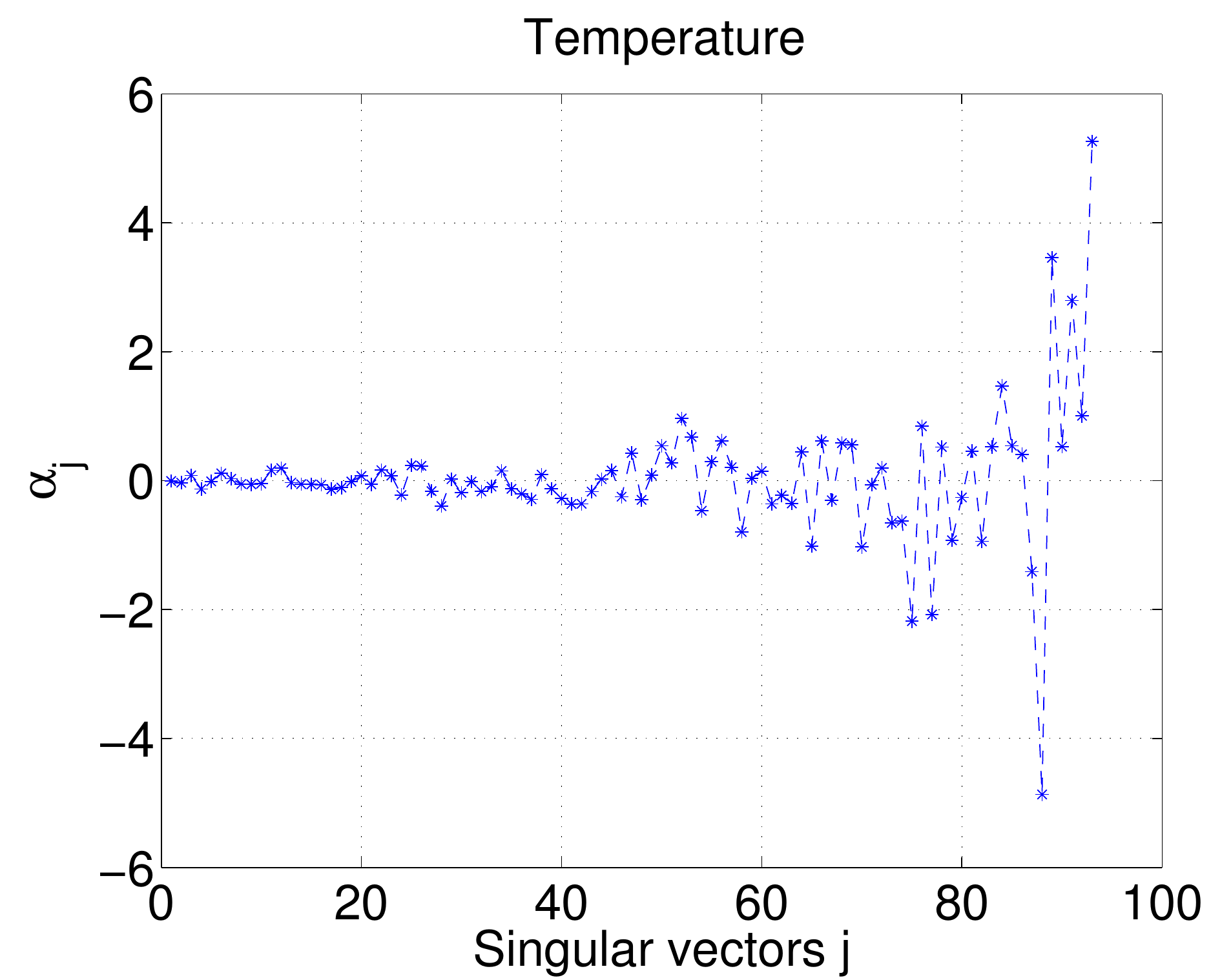}
\caption{Temperature ($K$)}
\end{subfigure}%
\begin{subfigure}{0.5\textwidth}
\centering
\includegraphics[width=0.9\textwidth,height=0.9\textwidth]{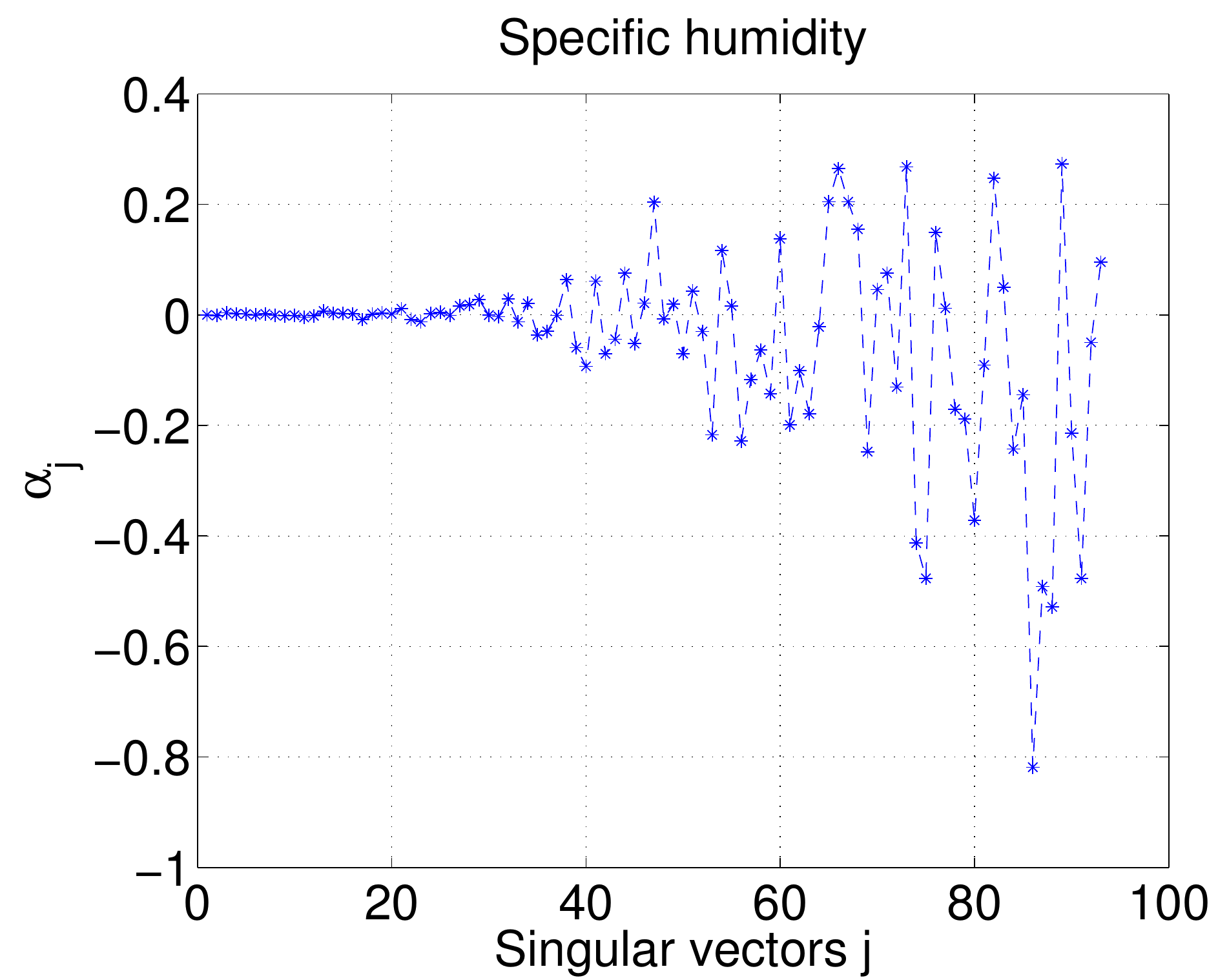}
\caption{Specific humidity ($g/kg$)}
\end{subfigure}
\caption{The effect of $\errR$ on the weights $\alpha_j$ for some model component $i$ of the SPEEDY model when $\ra=5$ and $p=4\%$.}
\label{fig:exp-random-noise-effect}
\end{figure}

%%%%%%%%%%%%%%%%%%%%%%%%%%%%%%%
\section{Conclusions}
\label{sec:conclusions}
%%%%%%%%%%%%%%%%%%%%%%%%%%%%%%%

This paper develops an efficient implementation of the ensemble Kalman filter, named EnKF-MC, that is based on a modified Cholesky decomposition to estimate the inverse background covariance matrix. This new approach has several advantages over classical formulations. First, a predefined sparsity structure can be built into the factors of the inverse covariance. This reflects the fact that if two distant model components are uncorrelated then the corresponding entry in the inverse covariance matrix is zero; the only nonzero entries in the Cholesky factors correspond to components of the model that are located in each other's proximity. Therefore, imposing a sparsity structure on the inverse background covariance matrix is a form of covariance localization. Second, the formulation allows for a rigorous theoretical analysis; we prove the convergence of the covariance estimator for a number of ensemble members that is proportional to the logarithm of the number of states of the model therefore, when $\Nens \approx \log \Nstate$, the background error correlations can be well-estimated making use of the modified Cholesky decomposition.

We discuss different implementations of the new EnKF-MC, and asses their computational effort. We show that domain decomposition can be used in order to decrease even more the computational effort of the proposed implementation. Numerical experiments are carried out using the Atmospheric General Circulation Model SPEEDY reveal that the analyses obtained by EnKF-MC are better than those of the LETKF in the root mean square sense when sparse observations are used in the analysis. For dense observation grids the EnKF-MC solutions are improved when the radius of influence increases, while the opposite holds true for LETKF analyses. (We stress the fact that these conclusions are true for our implementation of the basic LETKF; other implementations may incorporate advances that could make the filter perform considerably better). The use of modified Cholesky decomposition can mitigate the impact of spurious correlation during the assimilation of observations.

%%%%%%%%%%%%%%%%%%%%%%%%%%%%%%%%%%%%%%%
\section*{Acknowledgements}
%%%%%%%%%%%%%%%%%%%%%%%%%%%%%%%

This work was supported in part by awards NSF CCF--1218454,
AFOSR FA9550--12--1--0293--DEF, and by the Computational Science Laboratory at Virginia Tech.

%%%%%%%%%%%%%%%%%%%%%%%%%%%%%%%%%%%%%%%

%%%%%%%%%%%%%%%%%%%%%%%%%%%%%%%%%%%%%%%
\newcommand{\etalchar}[1]{$^{#1}$}

%%%%%%%%%%%%%%%%%%%%%%%%%%%%%%%%%%%%%%%

\end{document}